\documentclass[reqno]{amsart}
\usepackage{amsfonts,amsmath,amsthm,amssymb,mathrsfs}
\newtheorem{theorem}{Theorem}[section]

\newtheorem{lemma}[theorem]{Lemma}
\newtheorem{proposition}[theorem]{Proposition}
\newtheorem{remark}[theorem]{Remark}
\newtheorem{corollary}[theorem]{Corollary}

\numberwithin{equation}{section}
\usepackage[pdfstartview=FitH,colorlinks=true]{hyperref}
\hypersetup{urlcolor=red, linkcolor=blue,citecolor=red, CJKbookmarks=true}
\allowdisplaybreaks
\usepackage{color}

\begin{document}
\title[Prandtl equation]
{ Long time well-posdness\\ of the Prandtl equations in Sobolev space}
\author[C.-J. Xu]{Chao-Jiang Xu}
\date{May 8, 2016}
\address{Chao-Jiang XU
\newline\indent
School of Mathematics and Statistics, Wuhan University
\newline\indent
430072, Wuhan, P. R. China
\newline \indent
and 
\newline \indent
Universit\'e de Rouen, CNRS, UMR 6085-Laboratoire de Math\'ematiques
\newline \indent
76801 Saint-Etienne du Rouvray, France
}
\email{chao-jiang.xu@univ-rouen.fr }

\author[X. Zhang]{Xu Zhang}

\address{Xu ZHANG
\newline\indent
School of Mathematical Sciences, Xiamen University, Xiamen, Fujian 361005, China
\newline \indent
and
\newline \indent
Universit\'e de Rouen, CNRS, UMR 6085-Laboratoire de Math\'ematiques
	\newline \indent
	76801 Saint-Etienne du Rouvray, France}
\email{xu.zhang1@etu.univ-rouen.fr}


\begin{abstract}
In this paper, we study the long time well-posedness for the nonlinear Prandtl boundary layer equation on the half plane. While the initial data  are small perturbations of some monotonic shear profile, we prove the existence, uniqueness and stability of solutions in weighted Sobolev space by energy methods. The key point is that the life span of  the solution could be any large $T$ as long as its initial date is a perturbation around the monotonic shear profile of small size like  $e^{-T}$.  The nonlinear cancellation properties of  Prandtl equations under the monotonic assumption are the main ingredients to establish a new energy estimate.
\end{abstract}

\maketitle
\tableofcontents

\section{Introduction}

In this work, we study the initial-boundary value problem for the Prandtl boundary layer equation in two dimension, which reads
\begin{equation*}
\begin{cases}
\partial_t u + u \partial_x u + v\partial_y u + \partial_x p
= \partial^2_y u,\quad t>0,\, (x, y)\in\mathbb{R}^2_+, \\
\partial_x u +\partial_y v =0, \\
u|_{y=0} = v|_{y=0} =0 , \ \lim\limits_{y\to+\infty} u = U(t,x), \\
u|_{t=0} =u_0 (x,y)\, ,
\end{cases}
\end{equation*}
where $\mathbb{R}^2_+=\{(x, y)\in \mathbb{R}^2;\, y>0\}$, $u(t,x,y)$  represents the tangential velocity, $v(t, x, y)$ normal
velocity. $p(t, x)$ and $U(t, x)$ are the values on the boundary of  the Euler's pressure and
Euler's tangential velocity and determined by the Bernoulli's law: $
\partial_t U(t,x) +  U(t,x)\partial_x U(t,x) + \partial_x p =0.$

Prandtl equations is a major achievement in the progress of understanding the famous D'Alembert's paradox in fluid mechanics. In a word, D'Alembert's paradox can be stated as: while a solid body moves in an incompressible and inviscid potential flow,  it undergoes neither drag or buoyancy. This of course disobeys our everyday experiences. In 1904, Prandtl said that, in  fluid of small viscosity,  the behavior of fluid near the boundary is completely different from that away from the boundary. Away from the boundary part can be almost considered as ideal fluid, but the near boundary part is deeply affected by the viscous force and is described by  Prandtl boundary layer equation which  was firstly derived formally  by Prandtl in 1904 (\cite{prandtl1904uber}).

 From the mathematical point of view,  the well-posedness and justification of the Prandtl boundary layer theory don't have satisfactory  theory yet,  and remain open for general cases. During the past century, lots of mathematicians have investigated this problems. The Russian school has contributed a lot to the boundary layer theory and their works were collected in \cite{oleinik1999prandtl}. Up to now, the local existence theory for the Prandtl boundary layer equation has been achieved when the initial data belong to some special functional spaces: 1) the analytic space or analytic with respect to the tangential variable \cite{vicol2013,sammartino2003onlyx,sammartino1998analytic-1,sammartino1998analytic-2}; 2) Sobolev spaces or H\"older spaces under monotonicity assumption \cite{xu2014local,wangyaguang2014three,masmoudi2012local,oleinik1999prandtl,xin2004global}; 3) recently \cite{masmoudi2013gevrey} in Gevrey class with non-degenerate critical point. See also \cite{vicol2014} where the initial data  is monotone on a number of intervals and analytic on the complement.  

Except explaining the D'Alembert's Parabox,   Prandtl equations play a vital role in the challenging problem: inviscid limit problem. { In deed,  as pointed out by Grenier-Guo-Nguyen \cite{ggn1,ggn2,ggn3}, the long time behavior of the Prandtl equations is important to make progress towards the inviscid limit of the Navier-Stokes equations.  We must understand  behaviors of solutions to  on a longer time interval than the one which causes the instability used to prove ill-posedness. } 

To the best of our knowledge, under the monotonic assumption, by using the Crocco transformation, Oleinik (\cite{oleinik1999prandtl}) obtained the long-time smooth solution in H\"older space for the Prandtl equation defined on the interval $0\le x\le L$ with $L$ very small.
Xin-Zhang (\cite{xin2004global}) proved the global existence of weak solutions if the pressure gradient has a favorable sign, that is $\partial_x p\le 0$.  See \cite{wangyaguang2015global} for a similar work in 3-D case. The global existence of smooth solutions in the monotonic case remains open. 

In the analytical frame, Ignatova-Vicol (\cite{vicol2015}) recently get an almost global-in-time  solution which is analytic with respect to the tangential variable, see also \cite{zhangping2014longtime} for a same attempt work by using a refined Littlewood-Paley analysis. On the other side, without the monotonicity assumption, E and Engquist in \cite{e-2} constructed finite time blowup solutions to the Prandtl equation. After this work, there are many un-stability or strong ill-posedness results. In particular,  G\'erard-Varet and Dormy \cite{gerard2010ill} showed that the linearized Prandtl equation around the shear flow with a non-degenerate critical point is ill-posed in the sense of Hadamard in Sobolev spaces. See also \cite{e-1,gerard2010remark,grenier-2,guo2011nonlinear,renardy2009ill-steady-hydrostatic} for the relative works.

Besides, Crocoo transformation can't be used to Navier-Stokes equations.  The best choice left for us is to get the long time wellposedness by  energy method, since energy  method works well for both  Navier-Stokes equations and Euler equations.  Recently, there are two works\cite{xu2014local,masmoudi2012local} where the local-in-time wellposedness is obtained by different kinds of energy methods. One is by Nash-Moser-H\"ormander iteration. The other is by using uniform estimates of  the regularized parabolic equation and Maximal Principle.  

Motivated by above analysis, in this work, using directly energy method, we will prove the long time existence of smooth solutions of Prandtl equations in Sobolev space. { In details, for any fixed $T>0$, we will show that if the initial perturbation  are size of $e^{-T}$ small enough, then the  life time of solutions to Prandtl equations could  at least be  $T$.}

In what follows,  we choose the uniform outflow $
U(t,x)= 1$ which implies $p_x=0$. In other words the following problem for the Prandtl equation is considered :
\begin{equation}\label{full-prandtl}
\begin{cases}
\partial_t u + u \partial_x u + v\partial_y u
= \partial^2_y u,\,\, t>0,\,\, (x, y)\in\mathbb{R}^2_+, \\
\partial_x u +\partial_y v =0, \\
u|_{y=0} = v|_{y=0} =0 , \ \lim\limits_{y\to+\infty} u =1, \\
u|_{t=0} =u_0(x, y).
\end{cases}
\end{equation}

The weighted Sobolev spaces (similar to  \cite{masmoudi2012local}) are defined as follows: 
\begin{equation*}
\|f\|_{H^n_\lambda(\mathbb{R}^2_+)}^2 = \sum\limits_{|\alpha_1+\alpha_2|\le n}\int_{\mathbb{R}^2_+} \langle y \rangle^{2\lambda+2\alpha_2} |\partial^{\alpha_1}_{x}\partial^{\alpha_2}_{y} f|^2 dx dy\, ,~~~ \lambda > 0,~n \in \mathbb{N}^+.
\end{equation*}
Specially,  $\|f\|_{L^2_\lambda(\mathbb{R}^2_+)} = \|f\|_{H^0_\lambda(\mathbb{R}^2_+)} $ and $H^n$ stands for the usual Sobolev space.

{\bf Initial data of shear flow.}  Loosely speaking, shear flow is a  solution to Prandtl equations and is independent of $x$.   For more details, please check the {\it analysis of shear flow } part in Section \ref{section2} and Lemma \ref{shear-profile}. We denote shear flow as $u^s$. From now on,  we   consider  solutions to   Prandtl equations   as   their perturbations around some  shear flow. That is to say,
\[  u(t, x, y) =  u^s(t, y) + \tilde{u}(t, x, y),  t \ge 0.\]
Assume that $u^s_0$(initial datum of shear flow) satisfies the following conditions: 
\begin{align} \label{shear-critical-momotone}
\begin{cases}
u^s_0\in C^{m+4}([0, +\infty[),\,\,\, \lim\limits_{y \to + \infty} u^s_0(y)=1;\\
(\partial^{2p}_y u^s_0)(0) = 0,\,\,\,0\le 2p\le m+4;\\
c_1\langle y \rangle^{-k}\le (\partial_y u^s_0)(y)\le  c_2 \langle y \rangle^{-k}, ~~ \forall\,y\ge 0,\\
|(\partial_y^p u^s_0)(y)| \le c_2 \langle y \rangle^{-k-p+1},\,\, \forall\,\,y\ge 0,\,\, 1\le p\le m+4,
\end{cases}
\end{align}
for certain $c_1, c_2>0$ and even integer $m$.

  We have the following long time wellposedness results.

\begin{theorem}\label{main-theorem}
Let $m\ge 6$ be an even integer, $k>1$ and  $-\frac12<\nu<0$. Assume that $u^s_0$ satisfies \eqref{shear-critical-momotone}, the initial data
$\tilde u_0=(u_0-u^s_0) \in H^{m+3}_{k + \nu }(\mathbb{R}^2_+)$, 
and $\tilde u_0$  satisfies the compatibility condition up to order $m+2$.  Then for any $T>0$, there exists $\delta_0>0$ small enough such that if
\begin{equation}\label{initial-small}
\|\tilde u_0 \|_{H^{m+1}_{k + \nu }(\mathbb{R}^2_+)}\le \delta_0,
\end{equation}
then the initial-boundary value problem \eqref{full-prandtl} admits a unique solution $(u, v)$ with
$$
(u-u^s)\in L^\infty([0, T]; H^{m}_{k+\nu-\delta'}(\mathbb{R}^2_+)),\,\, 	v\in L^\infty([0, T]; L^\infty(\mathbb{R}_{y, +}; H^{m-1}(\mathbb{R}_x)),
$$
where $\delta'>0$ satisfying  $\nu+\frac 12<\delta'<\nu+1$ and $k+\nu-\delta'>\frac 12$. 

Moreover, we have the stability with respect to the initial data in the following sense:  given any two
initial data
$$
u^1_0=u^s_0+\tilde{u}^1_0,\quad u^2_0=u^s_0+\tilde{u}^2_0,
$$
if $u^s_0$ satisfies \eqref{shear-critical-momotone} and $\tilde{u}^1_0, \,\tilde{u}^2_0$ satisfies \eqref{initial-small}, then the solutions $u^1$ and $u^2$  to  \eqref{full-prandtl}  satisfy,
$$
\|u^1-u^2\|_{L^\infty([0, T]; H^{m-3}_{k+\nu-\delta'}(\mathbb{R}^2_+))} \le C\|u^1_0-u^2_0 \|_{ H^{m+1}_{k +\nu}(\mathbb{R}^2_+)}, 
$$
where the constant $C$ depends on the norm of $\partial_y{u}^1, \partial_y{u}^2$ in $L^\infty([0, T]; H^m_{k+\nu-\delta'+1}(\mathbb{R}^2_+))$.
\end{theorem}

\begin{remark}～
	\begin{itemize}
		\item [1.] We also can verify ,
		\begin{align*}
\partial_y (u-u^s)\in L^\infty([0, T]; H^{m}_{k+\nu-\delta' + 1}(\mathbb{R}^2_+)),\,
\partial_y v\in L^\infty([0, T]; H^{m-1}_{k+\nu-\delta'}(\mathbb{R}^2_+)).
		\end{align*} 
		\item [2.] From \eqref{c-tilde} and \eqref{bound-2},  the relationship between the life span $T$ and the size of initial data is:　
		$$
		\delta_0\,\,\approx\,\, e^{-T}.
		$$
		\item [3.] The results of main Theorem can be generated to the periodic case where $x$ is in torus.
		\item [4.] We find that the weight of solution $u(t) - u^s(t)$  is smaller than that of initial dates $u_0 - u^s_0$. 
		There means that there exist decay loss of order $\delta'>0$ which may be very small. It results from the term $v\,\partial_y u$ which is the major difficulty for the analysis of Prandtl equation. 
	\end{itemize}
\end{remark}

This article is arranged as follows. In Section \ref{section2}, we explain the main difficulties for the study of the Prandtl equation and present
an outline of our approach. In Section \ref{section3}, we study
the approximate solutions to \eqref{full-prandtl}
by a parabolic regularization. In Section \ref{section4}, we prepare some
technical tools and the formal transformation for the Prandtl equations. Sections \ref{section5} is dedicated to the uniform estimates of approximate solutions obtained in Section \ref{section3}. We prove finally
the main theorem in Section \ref{section7}-\ref{section8}.

\smallskip
\noindent
{\bf Notations: } The letter $C$ stands for various suitable constants, independent with functions and the special parameters, which may vary from line to line and step to step. When it depends on some crucial parameters in particular, we put a sub-index such as $C_\epsilon$ etc, which may also vary from line to line.

\section{Preliminary}\label{section2}

\noindent {\bf Difficulties and our approach.} Now, we  explain the main difficulties in proving Theorem	 \ref{main-theorem}, and present the strategies
of our approach.

It is well-known that the major difficulty for the study of the Prandtl equation \eqref{full-prandtl} is the term $v\,\partial_y u$, where the vertical velocity behaves like
$$
v(t, x, y)=-\int^y_0 \partial_x u(t, x, \tilde y)d\tilde y,
$$
by using the divergence free condition and boundary conditions. So it introduces a loss of $x$-derivative. The $y$-integration create also a loss of weights with respect to $y$-variable. Then the standard energy estimates do not work. This explains why there are few existence results in the literatures.

Recalling that in \cite{xu2014local} (see also \cite{masmoudi2012local} for a similar transformation), under the monotonic assumption $\partial_y u>0$,
we divide the Prandtl equations by $\partial_y u$ and then take derivative with respect to $y$, to obtain an equation of the new
unknown function $f=\left(\frac{u}{\partial_y u}\right)_y$ .
In the new equation, the term $v$ disappears by using the divergence free condition.  Here a little different from  \cite{xu2014local}, we use $ g_m =\left(\frac{\partial_x^m u}{\partial_y u}\right)_y$, where $m$ stands for the highest derivative with $x$. From  \cite{masmoudi2012local}, we can observe that we only need to worry about the highest derivative with $x$. This is why we only define $g_m$.

In order to prove the existence of solutions, following the idea of  Masmoudi-Wong (\cite{masmoudi2012local}), we will construct an approximate scheme and study the parabolic regularized
Prandtl equation \eqref{shear-prandtl-approxiamte}, which preserves the nonlinear structure of the
original Prandlt equation \eqref{full-prandtl}, as well as  the nonlinear cancellation properties.
Then by  uniform energy estimates of the approximate solutions,   the existence of solutions to
the original Prandlt equation \eqref{full-prandtl} follows. This energy estimate  also implies the uniqueness
and the stability. The uniform energy estimate for the approximate solutions is the main duty of this paper.

\noindent {\bf Analysis of shear flow.} We write the solution $(u, v)$ of system \eqref{full-prandtl} as
 \begin{align*}
   u(t, x, y) = u^s(t, y) + \tilde{u}(t, x, y),\,\, v(t, x, y)=\tilde v(t, x, y),
 \end{align*}
where $u^s(t,y)$ is the solution of the following heat equation
\begin{align}
\label{shear-flow}
\begin{cases}
\partial_t u^s - \partial_y^2 u^s =0,\\
u^s|_{y=0} = 0, \lim\limits_{y \to + \infty} u^s(t,y) = 1,\\
u^s|_{t=0} = u^s_0(y).
\end{cases}
   \end{align}
Then \eqref{full-prandtl} can be written as
\begin{equation}\label{non-shear-prandtl}
\begin{cases} \partial_t\tilde{u} + (u^s + \tilde{u}) \partial_x\tilde{u} + \tilde v (u^s_y +\partial_y  \tilde{u})
= \partial^2_y\tilde{u}, \\
\partial_x\tilde{u} +\partial_y\tilde{v} =0, \\
\tilde{u}|_{y=0} = \tilde{v}|_{y=0} =0 , \ \lim\limits_{y\to+\infty} \tilde{u} = 0, \\
\tilde{u}|_{t=0} =\tilde{u}_0 (x,y)\, .
\end{cases}
\end{equation}
We first study the shear flow,

\begin{lemma} \label{shear-profile}
Assume that the initial date $u^s_0$ satisfy \eqref{shear-critical-momotone}, then for any $T>0$, there exist $\tilde{c}_1, \tilde{c}_2, \tilde{c}_3>0$ such that the solution $u^s(t,y)$ of the initial boundary value problem \eqref{shear-flow} satisfies
\begin{align} \label{shear-critical-momotone-2}
\begin{cases}
\tilde{c}_1\langle y \rangle^{-k}\le \partial_y u^s(t, y) \le  \tilde{c}_2 \langle y \rangle^{-k}, ~~
\forall\,(t, y)\in [0, T]\times \bar{\mathbb{R}}_+,\\
|\partial_y^p u^s(t, y)| \le \tilde{c}_3 \langle y \rangle^{-k-p+1},\,\, \forall\,\,(t, y)\in [0, T]\times \bar{\mathbb{R}}_+,\,\, 1\le p\le m+4,
 \end{cases}
\end{align}
where $\tilde{c}_1, \tilde{c}_2, \tilde{c}_3$ depend on $T$.
\end{lemma}

\begin{proof}
Firstly, the solution of \eqref{shear-flow} can be written as
\begin{align*}
u^s (t,y) &=\frac{1}{2\sqrt {\pi t}} \int^{+\infty}_{0} \Big(
e^{-\frac{(y-\tilde{y})^2}{4 t}}-e^{-\frac{(y+\tilde{y})^2}
{4 t}}\Big)  u_0^s (\tilde{y}) d\tilde{y}\\
&=\frac{1}{\sqrt {\pi}} \Big(\int^{+\infty}_{- {\frac{y}{2\sqrt t}}}
 e^{-\xi^2} u_0^s (2\sqrt t \xi +y) d\xi - \int^{+\infty}_{
  {\frac{y}{2\sqrt t}}}  e^{-\xi^2}u_0^s (2\sqrt t \xi -y)d\xi \Big),
\end{align*}
which gives
\begin{align*}
\partial_t u^s(t, y) =& \frac{1}{\sqrt {\pi t}} \Big(
\int^{+\infty}_{- {\frac{y}{2\sqrt t}}} {\xi}\, e^{-\xi^2}
(\partial_y u_0^s) (2\sqrt t \xi +y) d\xi  \\
&\qquad- \int^{+\infty}_{ {\frac{y}{2\sqrt t}}}{\xi}\,
e^{-\xi^2}(\partial_y u_0^s) (2\sqrt t \xi -y)d\xi \Big).
\end{align*}
By using $(\partial_y^{2j}u_0^s)(0)=0$ for $0\le 2j\le m+4$, it follows
\begin{align}\label{u-0}
\begin{split}
\partial^p_y u^s(t, y) =& \frac{1}{\sqrt \pi} \Big( \int^{+\infty}_{-
{\frac{y}{2\sqrt t}}}  e^{-\xi^2} (\partial^p_yu_0^s) (2\sqrt t \xi+y) d\xi\\
 &\quad+ (-1)^{p+1}\int^{+\infty}_{ {\frac{y}{2\sqrt t}}}
 e^{-\xi^2}(\partial^p_yu_0^s) (2\sqrt t \xi -y)d\xi \Big)\\
 &=\frac{1}{2\sqrt {\pi t}} \int^{+\infty}_{0} \Big(
e^{-\frac{(y-\tilde{y})^2}{4 t}}+ (-1)^{p+1}e^{-\frac{(y+\tilde{y})^2}
{4 t}}\Big)  (\partial^p_yu_0^s) (\tilde{y}) d\tilde{y},
\end{split}
\end{align}
for all $1\le p\le m+4$.

For $p=1$, we have, 
\begin{align*}
\partial_y u^s(t, y) =& \frac{1}{\sqrt \pi} \Big( \int^{+\infty}_{-
{\frac{y}{2\sqrt t}}}  e^{-\xi^2} (\partial_yu_0^s) (2\sqrt t \xi+y) d\xi\\
 &\quad+\int^{+\infty}_{ {\frac{y}{2\sqrt t}}}
 e^{-\xi^2}(\partial_yu_0^s) (2\sqrt t \xi -y)d\xi \Big)\\
 &=\frac{1}{2\sqrt {\pi t}} \int^{+\infty}_{0} \Big(
e^{-\frac{(y-\tilde{y})^2}{4 t}}+e^{-\frac{(y+\tilde{y})^2}
{4 t}}\Big)  (\partial_yu_0^s) (\tilde{y}) d\tilde{y}\,.
\end{align*}
Thanks to the monotonic assumption \eqref{shear-critical-momotone}, we have that
\begin{align*}
\partial_y u^s(t, y) &\approx \frac{1}{2\sqrt {\pi t}} \int^{+\infty}_{0} \Big(
e^{-\frac{(y-\tilde{y})^2}{4 t}}+e^{-\frac{(y+\tilde{y})^2}
{4 t}}\Big) \langle \tilde{y}\rangle^{-k} d\tilde{y}\\
&\approx \frac{1}{2\sqrt {\pi t}} \int^{+\infty}_{-\infty} e^{-\frac{(y+\tilde{y})^2}
{4 t}} \langle \tilde{y}\rangle^{-k} d\tilde{y}\,.
\end{align*}
Recalling now Peetre's inequality, for any $\lambda\in\mathbb{R}$
\begin{equation*}
\tilde{c}_0\langle y\rangle^{\lambda}\langle y +\tilde{y}\rangle^{-|\lambda|}\le \langle \tilde{y}\rangle^{\lambda}\le
\tilde{c}^{-1}_0\langle {y}\rangle^{\lambda}\langle y+ \tilde{y}\rangle^{|\lambda|},
\end{equation*}
then for $\lambda=-k$, we get the first estimate of \eqref{shear-critical-momotone-2} with
\begin{equation}\label{c-tilde}
\tilde{c}_1=c_1\tilde{c}_0 (1+T)^{-\frac k2},\,\,\tilde{c}_2=c_2\tilde{c}^{-1}_0 (1+T)^{\frac k2}.
\end{equation}

For the second estimate of \eqref{shear-critical-momotone-2}, \eqref{u-0} implies
\begin{align*}
|\partial^p_y u^s(t, y)| &\le \frac{c_2}{2\sqrt {\pi t}} \int^{+\infty}_{0} \Big(
e^{-\frac{(y-\tilde{y})^2}{4 t}}+e^{-\frac{(y+\tilde{y})^2}
{4 t}}\Big) \langle \tilde{y}\rangle^{-k-p+1} d\tilde{y}\\
&\le \frac{c_2}{2\sqrt {\pi t}} \int^{+\infty}_{-\infty} e^{-\frac{(y+\tilde{y})^2}
{4 t}} \langle \tilde{y}\rangle^{-k-p+1} d\tilde{y}\,.
\end{align*}
Using now Peetre's inequality, with $\lambda=-k-p+1$, we get
\begin{align*}
|\partial^p_y u^s(t, y)|\le c_2\tilde{c}^{-1}_0(1+T)^{\frac {k+p-1}2}\langle {y}\rangle^{-k-p+1},
\end{align*}
for any $(t, y)\in [0, T]\times \mathbb{R}_+$. 
\end{proof}

\noindent {\bf Compatibility conditions and reduction of boundary data.} We give now the precise version of the compatibility condition for the nonlinear system \eqref{non-shear-prandtl} and the reduction properties of boundary data.

\begin{proposition}\label{prop-comp}
Let $m\ge 6$  be  an even integer, and assume that $\tilde{u}$ is a smooth solution of the system \eqref{non-shear-prandtl}, then the initial data $\tilde u_0$ have to satisfy the following compatibility conditions up to order $m+2$:
\begin{equation}\label{compatibility-a1}
\begin{cases}
&\tilde{u}_0(x, 0)=0, \quad\,(\partial^2_y \tilde{u}_0)(x, 0)=0, \,\,\forall x\in \mathbb{R},\\
&(\partial^4_y \tilde u_0)(x, 0)=\big(\partial_yu^s_0(0) + (\partial_y\tilde{u}_0)(x, 0)\big) (\partial_y\partial_x\tilde{u}_0)(x, 0),\forall x\in \mathbb{R},
\end{cases}
\end{equation}
and for $4\le 2p\le m$,
\begin{equation}\label{compatibility-a2}
(\partial^{2(p+1)}_y \tilde{u}_0)(x, 0)=\sum^p_{q=2}\sum_{(\alpha, \beta)\in \Lambda_q}C_{\alpha,\beta}\prod\limits_{j=1}^q \partial_x^{\alpha_j}\partial_y^{\beta_j +1} \big( u^s_0 + \tilde{u}_0 \big)\big|_{y=0}\,,\, \,\,\forall x\in \mathbb{R},
\end{equation}
where
\begin{align}\label{Lambda-p}
\begin{split}	
\Lambda_q=&\bigg\{
(\alpha, \beta)=(\alpha_1, \cdots, \alpha_q; \beta_1, \cdots, \beta_q)\in \mathbb{N}^{q}\times\,\mathbb{N}^{q};\\
&\qquad\alpha_j+\beta_j\le 2p-1,\,\,\, 1\le j\le q;\,\,~\sum^q_{j=1}3\alpha_j + \beta_j = 2p +1;\\
&\qquad\quad\quad\qquad~~\sum\limits_{j=1}^{q}\beta_j \le 2p -2,~\,0<\sum\limits_{j=1}^{q} \alpha_j \le p - 1\bigg\}\, .
\end{split}
\end{align}
\end{proposition}
Remark that for $\alpha_j>0$, we have  $\partial_x^{\alpha_j}\partial_y^{\beta_j +1} \big( u^s+ \tilde{u} \big)=\partial_x^{\alpha_j}\partial_y^{\beta_j +1}  \tilde{u}$. So
the condition $0<\sum\limits_{j=1}^{q} \alpha_j$ implies that, for each terms of  \eqref{compatibility-a2}, there is at last one factor  like $\partial_x^{\alpha_j}\partial_y^{\beta_j +1}  \tilde{u}_0$.

\begin{proof}  By the assumption of this Proposition, $\tilde u$ is a smooth solution.  If we need the existence of the trace of $\partial_y^{m+2} \tilde u$ on $y=0$, then we at least need to  assume that $\tilde{u}\in L^\infty([0, T]; H^{m+3}_{k+\ell-1}(\mathbb{R}^2_+))$.
	
Recalling the boundary condition in \eqref{non-shear-prandtl}:
\begin{equation*}
\tilde u(t, x, 0)=0, \quad \tilde v(t, x, 0)=0,\quad (t, x)\in [0, T]\times \mathbb{R},
\end{equation*}
then the following is obvious:
\begin{equation*}
(\partial_t\partial^n_x\tilde u)(t, x, 0)=0, \quad (\partial_t\partial^n_x \tilde v)(t, x, 0)=0,\quad (t, x)\in [0, T]\times \mathbb{R}, \, 0\le n \le m.
\end{equation*}
Thus the first result of \eqref{compatibility-a1} is exactly the compatibility of the solution with the initial data at $t=0$. For the second result of
\eqref{compatibility-a1}, using the equation of \eqref{non-shear-prandtl}, we find that,
fro $0\le n\le m$
\begin{equation*}
(\partial^2_{y}\partial^n_x\tilde{u})(t, x, 0)=0,\quad (\partial_t\partial^2_{y}\partial^n_x\tilde{u})(t, x, 0)=0,\quad (t, x)\in [0, T]\times \mathbb{R}.
\end{equation*}
Derivating  the equation of \eqref{non-shear-prandtl} with $y$,
$$
\partial_t\partial_y\tilde{u} + \partial_y\big((u^s + \tilde{u}) \partial_x\tilde{u}\big) +\partial_y\big(\tilde {v} (u^s_y + \partial_y \tilde{u})\big)
= \partial^3_{y}\tilde{u},
$$
observing
\begin{align*}
\bigg(\partial_y\big((u^s + \tilde{u}) \partial_x\tilde{u}\big) +\partial_y\bigg(\tilde {v} (u^s_y + \partial_y \tilde{u})\big)\bigg)\bigg|_{y=0}=0,
\end{align*}
then we get
\begin{equation*}
(\partial_t\partial_y\tilde{u}))|_{y=0}
= (\partial^3_{y}\tilde{u}_\epsilon)|_{y=0} .
\end{equation*}
Derivating again the equation of \eqref{non-shear-prandtl} with $y$,
$$
\partial_t\partial^2_y\tilde{u} + \partial^2_y\bigg((u^s + \tilde{u}) \partial_x\tilde{u}\bigg) +\partial^2_y\bigg(\tilde {v} (u^s_y + \partial_y \tilde{u})\bigg)
= \partial^4_{y}\tilde{u},
$$
using Leibniz formula
\begin{align*}
&\partial^2_y\bigg((u^s + \tilde{u}) \partial_x\tilde{u}\bigg) +\partial^2_y\bigg(\tilde {v} (u^s_y + \partial_y \tilde{u})\bigg)\\
&=(\partial^2_y(u^s + \tilde{u})) \partial_x\tilde{u} +(\partial^2_y\tilde {v})(u^s_y + \partial_y \tilde{u})\\
&\quad+(u^s + \tilde{u}) \partial^2_y\partial_x\tilde{u} +\tilde {v} \partial^2_y(u^s_y + \partial_y \tilde{u})\\
&\qquad+2(\partial_y(u^s + \tilde{u})) \partial_y\partial_x\tilde{u} +2(\partial_y\tilde {v})\partial_y(u^s_y + \partial_y \tilde{u}),
\end{align*}
thus,
\begin{equation*}
(\partial^4_y \tilde u)(t, x, 0)=\bigg(u^s_y(t, 0) + (\partial_y\tilde{u})(t, x, 0)\bigg) (\partial_y\partial_x\tilde{u})(t, x, 0),
\end{equation*}
and
\begin{equation}\label{boundary-a15}
\begin{split}
(\partial_t\partial^4_y \tilde u)(t, x, 0)=&\bigg(\partial_y u^s(t, 0) + (\partial_y\tilde{u})(t, x, 0)\bigg)\bigg((\partial^3_y\partial_x\tilde{u})(t, x, 0)\bigg) \\
& + \bigg(\partial^3_y u^s(t, 0) + (\partial^3_y\tilde{u})(t, x, 0)\bigg)\bigg((\partial_y\partial_x\tilde{u})(t, x, 0)\bigg).
\end{split}
\end{equation}
For $p=2$, we have
$$
\partial_t\partial^4_y\tilde{u}+ \partial^4_y\bigg((u^s + \tilde{u}) \partial_x\tilde{u}\bigg) +\partial^4_y\bigg(\tilde {v} (u^s_y + \partial_y \tilde{u})\bigg)
= \partial^6_{y}\tilde{u},
$$
using Leibniz formula
\begin{align*}
&\partial^4_y\bigg((u^s + \tilde{u}) \partial_x\tilde{u}\bigg) +\partial^4_y\bigg(\tilde {v} (u^s_y + \partial_y \tilde{u})\bigg)\\
&=(\partial^4_y(u^s + \tilde{u})) \partial_x\tilde{u} +(\partial^4_y\tilde {v})(u^s_y + \partial_y \tilde{u})
+(u^s + \tilde{u}) \partial^4_y\partial_x\tilde{u} +\tilde {v} \partial^4_y(u^s_y + \partial_y \tilde{u})\\
&\qquad+\sum_{1\le j\le 3}C^4_j \bigg((\partial^j_y(u^s + \tilde{u})) \partial^{4-j}_y\partial_x\tilde{u} +(\partial^j_y\tilde {v})\partial^{4-j}_y(u^s_y + \partial_y \tilde{u})\bigg),
\end{align*}
thus,  by \eqref{boundary-a15}
\begin{equation}\label{boundary-16-0}
\begin{split}
&(\partial^6_y \tilde u)(t, x, 0)=
(\partial_t\partial^4_y \tilde u)(t, x, 0)
-(\partial^3_y\partial_x{u})(u^s_y + \partial_y \tilde{u})(t, x, 0)\\
&\quad+\sum_{1\le j\le 3}C^4_j \bigg((\partial^j_y(u^s + \tilde{u})) \partial^{4-j}_y\partial_x\tilde{u} +(\partial^j_y\tilde {v})\partial^{4-j}_y(u^s_y + \partial_y \tilde{u})\bigg)(t, x, 0)\\
&=\quad\quad \bigg(\partial^3_y u^s(t, 0) + (\partial^3_y\tilde{u})(t, x, 0)\bigg)\bigg((\partial_y\partial_x\tilde{u})(t, x, 0)\bigg)\\
&\quad+\sum_{1\le j\le 3}C^4_j \bigg((\partial^j_y(u^s + \tilde{u})) \partial^{4-j}_y\partial_x\tilde{u} -(\partial^{j-1}_y\partial_x\tilde {u})\partial^{4-j}_y(u^s_y + \partial_y \tilde{u})\bigg)(t, x, 0).
\end{split}
\end{equation}
Taking the values at $t=0$, we have proven \eqref{compatibility-a2} for $p=2$. The case of $p\ge 3$ is then by induction.
\end{proof}

\begin{remark}
By the similar methods, we can prove that if $\tilde u$ is a smooth solution of the system \eqref{non-shear-prandtl}, then we have
\begin{equation*}
\begin{cases}
&\tilde{u}(t, x, 0)=0, \,\,(\partial^2_y \tilde{u})(t, x, 0)=0, \,\,\forall (t, x)\in [0, T]\times \mathbb{R},\\
&(\partial^4_y \tilde u)(t, x, 0)=\big(u^s_y(t, 0) + (\partial_y\tilde{u})(t, x, 0)\big) (\partial_y\partial_x\tilde{u})(t, x, 0),\forall (t, x)\in [0, T]\times  \mathbb{R},
\end{cases}
\end{equation*}
and for $4\le 2p\le m$,
\begin{equation}\label{boundary-data1-e}
(\partial^{2(p+1)}_y \tilde{u})(t, x, 0)=\sum^p_{q=2}\sum_{(\alpha, \beta)\in \Lambda_q}C_{\alpha,\beta}\prod\limits_{j=1}^q \partial_x^{\alpha_j}\partial_y^{\beta_j +1} \Big( u^s(t, 0) + \tilde{u}(t, x, 0) \Big),
\end{equation}
for all $ (t, x)\in [0, T]\times \mathbb{R}$, where $\Lambda_q$ is defined in \eqref{Lambda-p}.

See Lemma 5.9 of  \cite{masmoudi2012local} and  Lemma 4 of \cite{masmoudi2013gevrey} for the similar results.
\end{remark}

Remark that the condition $0<\sum\limits_{j=1}^{q} \alpha_j$ implies that, for each terms of  \eqref{boundary-data1-e}, there is at last one factor like $\partial_x^{\alpha_j}\partial_y^{\beta_j +1}  \tilde{u}(t, x, 0)$.


\section{The approximate solutions} \label{section3}

To prove the existence of solution of the Prandtl equation, we study a parabolic regularized equation for which we can get
the existence by using the classical energy method.

\noindent {\bf Nonlinear regularized Prandtl equation.} We study the following nonlinear regularized Prandtl equation, for $0<\epsilon\le 1$,
\begin{equation}\label{shear-prandtl-approxiamte}
\left\{\begin{array}{l}
\partial_t\tilde{u}_\epsilon + (u^s + \tilde{u}_\epsilon) \partial_x\tilde{u}_\epsilon +{v}_\epsilon (u^s_y + \partial_y \tilde{u}_\epsilon)
= \partial^2_{y}\tilde{u}_\epsilon + \epsilon \partial^2_{x}\tilde{u}_\epsilon, \\
\partial_x\tilde{u}_\epsilon +\partial_y{v}_\epsilon =0, \\
\tilde{u}_\epsilon|_{y=0} = {v}_\epsilon|_{y=0} =0 , \ \lim\limits_{y\to+\infty} \tilde{u}_\epsilon = 0, \\
\tilde{u}_\epsilon|_{t=0}=\tilde{u}_{0, \epsilon} =\tilde{u}_0+\epsilon \mu_\epsilon \, ,
\end{array}\right.
\end{equation}
where we choose the corrector $\epsilon \mu_\epsilon $ such that $\tilde{u}_0 +\epsilon \mu_\epsilon $ satisfies the compatibility condition up to order $m+2$   for the regularized system \eqref{shear-prandtl-approxiamte}.

We study now the boundary data of the solution for the regularized nonlinear system \eqref{shear-prandtl-approxiamte} which give also the precise version of the compatibility condition for the system \eqref{shear-prandtl-approxiamte}, see \cite{cannone-non1,cannone-non2} for the Prandtl equation with non-compatible data.

\begin{proposition}\label{prop-comp-b}
Let $m\ge 6$ be an even integer $1<k,  0< \ell<\frac12$ and $k+\ell>\frac 32$, and assume that $\tilde{u}_0$ satisfies the compatibility
conditions \eqref{compatibility-a1} and \eqref{compatibility-a2} for the system
\eqref{non-shear-prandtl}, and
$\mu_\epsilon \in H^{m+3}_{k +\ell'-1}(\mathbb{R}^2_+)$ for some $\frac 12 <\ell'<\ell+\frac 12$
such that  $\tilde{u}_0 +\epsilon \mu_\epsilon $ satisfies the compatibility
conditions up to order $m+2$ for the regularized system
\eqref{shear-prandtl-approxiamte}. If $\tilde{u}_\epsilon \in
L^\infty ([0, T]; H^{m+3}_{k +\ell}(\mathbb{R}^2_+))\cap Lip([0, T];
H^{m+1}_{k +\ell}(\mathbb{R}^2_+))$ is a solution of the
system \eqref{shear-prandtl-approxiamte}, then we have
\begin{equation*}
\begin{cases}
&\tilde{u}_\epsilon(t, x, 0)=0, \,\,(\partial^2_y \tilde{u}_\epsilon)(t, x, 0)=0, \,\,\forall (t, x)\in [0, T]\times \mathbb{R},\\
&(\partial^4_y \tilde u_\epsilon)(t, x, 0)=\big(u^s_y(t, 0) + (\partial_y\tilde{u}_\epsilon)(t, x, 0)\big) (\partial_y\partial_x\tilde{u}_\epsilon)(t, x, 0),\forall (t, x)\in [0, T]\times  \mathbb{R},
\end{cases}
\end{equation*}
and for $4\le 2p\le m$,
\begin{equation}\label{boundary-data1b}
\begin{split}
(\partial^{2(p+1)}_y \tilde{u}_\epsilon)(t, x, 0)=&
\sum^p_{q=2}\sum^{q-1}_{l=0}\epsilon^l\sum_{(\alpha^l, \beta^l)\in \Lambda^l_q}C_{\alpha^l,\beta^l}
\\&
\qquad\times\, \prod\limits_{j=1}^q \partial_x^{\alpha^l_j}\partial_y^{\beta^l_j +1}
\big( u^s(t, 0) + \tilde{u}_\epsilon(t, x, 0) \big),
\end{split}
\end{equation}
for all $ (t, x)\in [0, T]\times \mathbb{R}$, where
\begin{equation*}
\begin{split}	
\Lambda^l_q=&\bigg\{(\alpha, \beta)=(\alpha_1, \cdots, \alpha_p; \beta_1, \cdots, \beta_p)\in
\mathbb{N}^{q}\times \mathbb{N}^q;\\
&\qquad \alpha_j+\beta_j\le 2p-1,\,,~~1\le j\le q; \,\,~\sum^q_{j=1}3\alpha_j + \beta_j = 2p +4l+1;\\
&\qquad\qquad\sum\limits_{j=1}^{q}\beta_j \le 2p -2l-2,~\,0<\sum\limits_{j=1}^{q} \alpha_j \le p +2l - 1\bigg\}.
\end{split}
\end{equation*}
\end{proposition}
\begin{remark}\label{remark3.2}.
\begin{itemize}
\item[1.]  Remark that the condition $0<\sum\limits_{j=1}^{q} \alpha^l_j$ implies that, for each terms of  \eqref{boundary-data1b}, there are at last one factor like $\partial_x^{\alpha^l_j}\partial_y^{\beta^l_j +1}  \tilde{u}_\epsilon(t, x, 0)$.
\item[2.]  Here we change the notation for the wighted index of function space, in fact, 
using the notations of Theorem \ref{main-theorem}, we have
$$
\ell=\nu-\delta'+1,\quad \ell'=\nu+1.
$$
\end{itemize}
\end{remark}

\begin{proof} Firstly, for $ p\le \frac m2$, we have  $\partial^{2p+2}_y\tilde{u}_\epsilon \in L^\infty ([0, T]; H^{1}_{k +\ell + 2p + 1}(\mathbb{R}^2_+))$.  So the trace of $\partial^{2p+2}_y\tilde{u}_\epsilon$ exists on $y=0$.

Using the boundary condition of \eqref{shear-prandtl-approxiamte}, we have,
for $0\le n\le m+2$,
\begin{equation*}
\partial^n_x\tilde u_\epsilon(t, x, 0)=0, \quad \partial^n_xv_\epsilon(t, x, 0)=0,\quad (t, x)\in [0, T]\times \mathbb{R},
\end{equation*}
and for $0\le n\le m$
\begin{equation*}
(\partial_t\partial^n_x\tilde u_\epsilon)(t, x, 0)=0, \quad (\partial_t\partial^n_x v_\epsilon)(t, x, 0)=0,\quad (t, x)\in [0, T]\times \mathbb{R}.
\end{equation*}
From the equation of \eqref{shear-prandtl-approxiamte}, we get also
\begin{equation}\label{boundary-12b}
(\partial^2_{y}\partial^n_x\tilde{u}_\epsilon)(t, x, 0)=0,\quad (\partial_t\partial^2_{y}\partial^n_x\tilde{u}_\epsilon)(t, x, 0)=0,\quad (t, x)\in [0, T]\times \mathbb{R}.
\end{equation}
On the other hand,
$$
\partial_t\partial_y\tilde{u}_\epsilon + \partial_y\big((u^s + \tilde{u}_\epsilon) \partial_x\tilde{u}_\epsilon\big) +\partial_y\big({v}_\epsilon (u^s_y + \partial_y \tilde{u}_\epsilon)\big)
= \partial^3_{y}\tilde{u}_\epsilon + \epsilon \partial^2_{x}\partial_y\tilde{u}_\epsilon,
$$
observing
\begin{align*}
\big[\partial_y\big((u^s + \tilde{u}_\epsilon) \partial_x\tilde{u}_\epsilon\big) +\partial_y\bigg({v}_\epsilon (u^s_y + \partial_y \tilde{u}_\epsilon)\big)\big]\big|_{y=0}=0,
\end{align*}
we get
\begin{equation*}
(\partial_t\partial_y\tilde{u}_\epsilon)|_{y=0}
= (\partial^3_{y}\tilde{u}_\epsilon)|_{y=0} + \epsilon (\partial^2_{x}\partial_y\tilde{u}_\epsilon)|_{y=0}.
\end{equation*}
We have also
$$
\partial_t\partial^2_y\tilde{u}_\epsilon + \partial^2_y\big((u^s + \tilde{u}_\epsilon) \partial_x\tilde{u}_\epsilon\big) +\partial^2_y\big({v}_\epsilon (u^s_y + \partial_y \tilde{u}_\epsilon)\big)
= \partial^4_{y}\tilde{u}_\epsilon + \epsilon \partial^2_{x}\partial^2_y\tilde{u}_\epsilon,
$$
using Leibniz formula
\begin{align*}
&\partial^2_y\big((u^s + \tilde{u}_\epsilon) \partial_x\tilde{u}_\epsilon\big) +\partial^2_y\big({v}_\epsilon (u^s_y + \partial_y \tilde{u}_\epsilon)\big)\\
&=(\partial^2_y(u^s + \tilde{u}_\epsilon)) \partial_x\tilde{u}_\epsilon +(\partial^2_y{v}_\epsilon)(u^s_y + \partial_y \tilde{u}_\epsilon)\\
&\quad+(u^s + \tilde{u}_\epsilon) \partial^2_y\partial_x\tilde{u}_\epsilon +{v}_\epsilon \partial^2_y(u^s_y + \partial_y \tilde{u}_\epsilon)\\
&\qquad+2(\partial_y(u^s + \tilde{u}_\epsilon)) \partial_y\partial_x\tilde{u}_\epsilon +2(\partial_y{v}_\epsilon)\partial_y(u^s_y + \partial_y \tilde{u}_\epsilon),
\end{align*}
thus,
\begin{equation}\label{boundary-14}
(\partial^4_y \tilde u_\epsilon)(t, x, 0)=\left(u^s_y(t, 0) + (\partial_y\tilde{u}_\epsilon)(t, x, 0)\right) (\partial_y\partial_x\tilde{u}_\epsilon)(t, x, 0).
\end{equation}
Applying $\partial_t$ to  \eqref{boundary-14},  we have
\begin{equation*}
\begin{split}
&(\partial_t\partial^4_y \tilde u_\epsilon)(t, x, 0)=\left(\partial^3_y u^s(t, 0) + (\partial^3_y\tilde{u}_\epsilon)(t, x, 0)+\epsilon (\partial^2_x\partial_y \tilde u_\epsilon)(t, x, 0)\right)(\partial_y\partial_x\tilde{u}_\epsilon)(t, x, 0)\\
&+\left(u^s_y(t, 0) + (\partial_y\tilde{u}_\epsilon)(t, x, 0)\right) \left((\partial^3_y\partial_x\tilde{u}_\epsilon)(t, x, 0)+\epsilon (\partial^3_x\partial_y \tilde u_\epsilon)(t, x, 0)\right).
\end{split}
\end{equation*}
On the other hand, we have
$$
\partial_t\partial^4_y\tilde{u}_\epsilon + \partial^4_y\big((u^s + \tilde{u}_\epsilon) \partial_x\tilde{u}_\epsilon\big) +\partial^4_y\big({v}_\epsilon (u^s_y + \partial_y \tilde{u}_\epsilon)\big)
= \partial^6_{y}\tilde{u}_\epsilon + \epsilon \partial^2_{x}\partial^4_y\tilde{u}_\epsilon,
$$
using Leibniz formula
\begin{align*}
&\partial^4_y\big((u^s + \tilde{u}_\epsilon) \partial_x\tilde{u}_\epsilon\big) +\partial^4_y\big({v}_\epsilon (u^s_y + \partial_y \tilde{u}_\epsilon)\big)\\
&=(\partial^4_y(u^s + \tilde{u}_\epsilon)) \partial_x\tilde{u}_\epsilon +(\partial^4_y{v}_\epsilon)(u^s_y + \partial_y \tilde{u}_\epsilon)\\
&\quad+(u^s + \tilde{u}_\epsilon) \partial^4_y\partial_x\tilde{u}_\epsilon +{v}_\epsilon \partial^4_y(u^s_y + \partial_y \tilde{u}_\epsilon)\\
&\qquad+\sum_{1\le j\le 3}C^4_j \big((\partial^j_y(u^s + \tilde{u}_\epsilon)) \partial^{4-j}_y\partial_x\tilde{u}_\epsilon +(\partial^j_y{v}_\epsilon)\partial^{4-j}_y(u^s_y + \partial_y \tilde{u}_\epsilon)\big),
\end{align*}
thus,
\begin{equation*}
\begin{split}
&(\partial^6_y \tilde u_\epsilon)(t, x, 0)=
(\partial_t\partial^4_y \tilde u_\epsilon)(t, x, 0)
-(\partial^3_y\partial_x{u}_\epsilon)(u^s_y + \partial_y \tilde{u}_\epsilon)(t, x, 0)\\
&\quad+\sum_{1\le j\le 3}C^4_j \big[(\partial^j_y(u^s + \tilde{u}_\epsilon)) \partial^{4-j}_y\partial_x\tilde{u}_\epsilon +(\partial^j_y{v}_\epsilon)\partial^{4-j}_y(u^s_y + \partial_y \tilde{u}_\epsilon)\big](t, x, 0)\\
&\qquad\qquad -\underline{\epsilon \partial^2_{x}\partial^4_y\tilde{u}_\epsilon(t, x, 0)}.
\end{split}
\end{equation*}
Using \eqref{boundary-14}, we get then
\begin{equation}\label{boundary-16}
\begin{split}
&(\partial^6_y \tilde u_\epsilon)(t, x, 0)
=
\big(\partial^3_y u^s(t, 0) +  \partial^3_y\tilde{u}_\epsilon (t, x, 0)\big)\partial_y\partial_x\tilde{u}_\epsilon(t, x, 0)\\
&\hskip 5cm  -\underline{2\epsilon  \partial_x\partial_y\tilde{u}_\epsilon(t, x, 0)  (\partial_y\partial_x^2\tilde{u}_\epsilon)(t, x, 0)}\\
&+\sum_{1\le j\le 3}C^4_j \big[(\partial^j_y(u^s + \tilde{u}_\epsilon)) \partial^{4-j}_y\partial_x\tilde{u}_\epsilon - \partial^{j - 1}_y\partial_x \tilde{u}_\epsilon\partial^{4-j}_y(u^s_y + \partial_y \tilde{u}_\epsilon)\big](t, x, 0),
\end{split}
\end{equation}
Compared to \eqref{boundary-16-0}, the underlined term is the new term.

This is the Proposition \ref{prop-comp-b} for $p=2$. We can  complete the proof of Proposition \ref{prop-comp-b} by induction.
\end{proof}

The proof of the above Proposition implies also the following result.
\begin{corollary}\label{coro-boundary}
Let $m\ge 6$ be an even integer, assume that $\tilde{u}_0$ satisfies the compatibility conditions \eqref{compatibility-a1} - \eqref{compatibility-a2} for the system \eqref{non-shear-prandtl} and $\partial_y\tilde{u}_{0}\in H^{m+2}_{k+\ell'}(\mathbb{R}^2_+)$, then there exists $\epsilon_0>0$, and for any $0<\epsilon\le \epsilon_0$ there exists $\mu_\epsilon \in H^{m+3}_{k +\ell'-1}(\mathbb{R}^2_+)$ such that  $\tilde{u}_0 +\epsilon \mu_\epsilon $ satisfies the compatibility condition up to order $m+2$  for the regularized system \eqref{shear-prandtl-approxiamte}. Moreover, for any $m\le \tilde m\le m+2$
$$
\|\partial_y\tilde{u}_{0, \epsilon}\|_{H^{\tilde m}_{k+\ell'}(\mathbb{R}^2_+)}\le \frac 32 \|
\partial_y\tilde{u}_{0}\|_{H^{\tilde m}_{k+\ell'}(\mathbb{R}^2_+)},
$$
and
$$
\lim_{\epsilon\to 0}\|\partial_y\tilde{u}_{0, \epsilon}-\partial_y\tilde{u}_{0}\|_{H^{\tilde m}_{k+\ell'}(\mathbb{R}^2_+)}=0.
$$
\end{corollary}

\begin{proof} We use the proof of the Proposition \ref{prop-comp-b}.

Taking the values at $t=0$ for \eqref{boundary-12b}, then \eqref{compatibility-a1} implies that the function $\mu_\epsilon$ satisfies
\begin{equation*}
(\partial^n_x\mu_\epsilon )(x, 0)=0, \quad (\partial^2_y \partial^n_x\mu_\epsilon )(x, 0)=0,\quad x\in \mathbb{R}\,.
\end{equation*}
Taking $t=0$ for \eqref{boundary-14}, we have
\begin{align*}
(\partial^4_y \tilde u_0)(x, 0)+\epsilon(\partial^4_y \mu_\epsilon)(x, 0))=&\big[\partial_yu^s_0(0) + (\partial_y\tilde{u}_0)(x, 0)+\epsilon(\partial_y \mu_\epsilon)(x, 0)\big]\\
&\times\big[(\partial_y\partial_x\tilde{u}_0)(x, 0)+\epsilon(\partial_y \partial_x\mu_\epsilon)(x, 0)\big],
\end{align*}
using \eqref{compatibility-a1}, we have that $\mu_\epsilon$ satisfies
\begin{equation*}
\begin{split}
(\partial^4_y \mu_\epsilon )(x, 0))=&\big(\partial_yu^s_0(0) + (\partial_y\tilde{u}_0)(x, 0)\big)(\partial_y \partial_x\mu_\epsilon )(x, 0)\\
&+(\partial_y \mu_\epsilon )(x, 0)(\partial_y\partial_x\tilde{u}_0)(x, 0)\\
&+\epsilon(\partial_y \partial_x\mu_\epsilon )(x, 0)(\partial_y \partial_x\mu_\epsilon )(x, 0).
\end{split}
\end{equation*}
We have also
\begin{equation*}
\begin{split}
(\partial_t\partial^4_y \tilde u_\epsilon)(0, x, 0)=&\big(\partial^3_y u^s_0(0) + (\partial^3_y\tilde{u}_\epsilon)(0, x, 0)+\epsilon (\partial^2_x\partial_y \tilde u_\epsilon)(0, x, 0)\big)\\
&\times \big((\partial^3_y\partial_x\tilde{u}_\epsilon)(0, x, 0)+\epsilon (\partial^3_x\partial_y \tilde u_\epsilon)(0, x, 0)\big).
\end{split}
\end{equation*}
Taking the values at $t=0$ for \eqref{boundary-16}, we obtain a restraint condition for $(\partial^6_y \mu_\epsilon)(x, 0)$,
\begin{align*}
	\partial_y^6  \mu_\epsilon (x, 0) & =( (\partial_y^3 u^s_0 + \partial_y^3 \tilde{u}_0  ) \partial_y \partial_x \mu_\epsilon)|_{y=0} +  \partial_y^3 \mu_\epsilon \partial_y \partial_x \tilde{u}_0|_{y=0} + \epsilon  \partial_y^3 \mu_\epsilon \partial_y \partial_x \mu_\epsilon|_{y=0}\\
	&  - \underline{2　 \partial_x\partial_y\tilde{u}_0(x, 0)  (\partial_y\partial_x^2\tilde{u}_0)( x, 0)}- 2　\epsilon \partial_x\partial_y\tilde{u}_0( x, 0)  (\partial_y\partial_x^2\mu_\epsilon)(t, x, 0)  \\
	& - 2　\epsilon \partial_x\partial_y\mu_\epsilon(t, x, 0)  (\partial_y\partial_x^2\tilde{u}_0)(t, x, 0) - 2 \epsilon^2　 \partial_x\partial_y\mu_\epsilon(x, 0)  (\partial_y\partial_x^2\mu_\epsilon)(x, 0)  \\
	& + \sum\limits_{1 \le j \le 3}C_j^4 \big[ \partial_y^j\big( u^s_0 + \tilde{u}_0 \big)\partial_y^{4 - j}\partial_x \mu_\epsilon + \partial_y^j \mu \partial_y^{4 - j} \partial_x \tilde{u}_0 + \epsilon\partial_y^j \mu \partial_y^{4 - j} \partial_x \mu_\epsilon  \big]\big|_{y=0}\\
	& - \sum\limits_{ 1 \le j \le 3 }C_j^4 \big[  \partial_y^{j-1} \partial_x \tilde{u}_0 \partial_y^{4 - j} \mu_\epsilon + \epsilon \partial_y^{j - 1} \partial_x \mu_\epsilon \partial_y^{4 - j} \partial_y \mu_\epsilon    \big]\big|_{y = 0}\\
	& - \sum\limits_{ 1 \le j \le 3 }C_j^4    \partial_y^{j - 1} \partial_x \mu_\epsilon \partial_y^{4 - j}( \partial_y u^s_0 + \partial_y \tilde{u}_0 )   \big|_{y = 0},
\end{align*}
thus
\begin{equation}\label{mu-6}
\begin{split}
		\partial_y^6  \mu_\epsilon (x, 0) 	& =  - ~ \underline{2　 \partial_x\partial_y\tilde{u}_0(x, 0)  (\partial_y\partial_x^2\tilde{u}_0)( x, 0)}\\
	& + \sum\limits_{\alpha_1, \beta_1; \alpha_2, \beta_2}C_{\alpha_1, \beta_1; \alpha_2, \beta_2} \partial_x^{\alpha_1}\partial_y^{\beta_1  + 1} ( u^s_0 + \tilde{u}_0) \partial_x^{\alpha_1}\partial_y^{\beta_1  + 1}\mu_\epsilon(x, 0)\\
	& + \sum\limits_{\alpha_1, \beta_1; \alpha_2, \beta_2}C_{\alpha_1, \beta_1; \alpha_2, \beta_2} \partial_x^{\alpha_1}\partial_y^{\beta_1  + 1} \mu_\epsilon \partial_x^{\alpha_1}\partial_y^{\beta_1  + 1}\mu_\epsilon(x, 0),
\end{split}
\end{equation}
where the summation is for the index $\alpha_2+\beta_2\le 3;\,
\alpha_1+\beta_1+\alpha_2+\beta_2\le 3$. The underlined term in
the above equality is deduced from the underlined term in \eqref{boundary-16}.
All these underlined terms are from the added regularizing term $\epsilon\partial_x^2 \tilde{u}$
in the equation \eqref{shear-prandtl-approxiamte}.
This means that the regularizing term $\epsilon \partial_x^2 \tilde{u}$
has an affect on the boundary. This is why we add a corrector term.

More generally,  for $6\le 2p\le m$, we have that $(\partial^{2(p+1)}_y \mu_\epsilon)(x, 0)
$ is a linear combination of the terms of the form
$$
\prod\limits_{j=1}^{q_1}\left( \partial_x^{\alpha^1_j}\partial_y^{\beta^1_j +1} \big( u^s_0 + \tilde{u}_0 \big)\right)\bigg|_{y=0},\,\quad \prod\limits_{i=1}^{q_2}\left( \partial_x^{\alpha^2_i}\partial_y^{\beta^2_i+1} \mu_\epsilon\right)\bigg|_{y=0}\,,
$$
and $$
\prod\limits_{j=1}^{q_1}\left( \partial_x^{\alpha^1_j}\partial_y^{\beta^1_j +1} \big( u^s_0 + \tilde{u}_0 \big)\right)\bigg|_{y=0}\,\times \, \prod\limits_{i=1}^{q_2}\left( \partial_x^{\alpha^2_i}\partial_y^{\beta^2_i+1} \mu_\epsilon\right)\bigg|_{y=0}\,,
$$
where the coefficients of the combination can be depends on $\epsilon$ but with
a non-negative power. We have also $\alpha^l_j+\beta^l_j+1\le 2p, l=1, 2$, thus $(\partial^{2(p+1)}_y \mu_\epsilon)(x, 0)$ is determined by the low order derivatives of $\mu_\epsilon$ and these of  $\tilde u_0$.

We now construct a polynomial function $\tilde \mu_\epsilon$ on $y$ by the following
Taylor expansion,
$$
\tilde \mu_\epsilon(x, y)=\sum^{\frac m2+1}_{p=3} \tilde \mu^{2p}_\epsilon(x)\frac{ y^{2p}}{(2p)!}\,,
$$
where
$$
\tilde \mu^{6}_\epsilon(x)=-2 (\partial_x\partial_y\tilde{u}_0)(x, 0)(\partial_y\partial^2_x\tilde{u}_0)(x, 0),
$$
and $\tilde \mu^{2p}_\epsilon(x)$  will give successively by  $(\partial^{2q}_y \mu_\epsilon)(x, 0)$  with  $(\partial^{2q+1}_y \mu_\epsilon)(x, 0)=0, q=0, \cdots, m$, and it is then determined by $(\partial^\alpha_x\partial^\beta_y\tilde u_0)|_{y=0}$.
Finally we take $\mu_\epsilon= \chi(y)\tilde \mu_\epsilon$ with $\chi\in C^\infty([0, +\infty[);\, \chi(y)=1,\, 0\le y\le 1;\, \chi(y)=0,\, y\ge 2$. Thus we   complete the proof of the Corollary.
\end{proof}

\begin{remark}\label{remark-corrector}
Suppose that $\tilde u_0$ satisfies the compatibility conditions up to  order $m+2$
 for the system \eqref{non-shear-prandtl} with $m\ge 4$, then for the regularized
 system  \eqref{shear-prandtl-approxiamte},  if we want to obtain the smooth solution $\tilde w_\epsilon$,  we have to add a non-trivial corrector $\mu_\epsilon$ to the initial data such that $\tilde u_0+\epsilon\mu_\epsilon$ satisfies the compatibility conditions
 up to  order $m+2$  for the system \eqref{shear-prandtl-approxiamte}. In fact,  if we take $\mu_\epsilon$ with
$$
(\partial^{j}_y \mu_\epsilon)(x, 0)=0,\quad 0\le j\le 5,
$$
 then \eqref{mu-6} implies
$$
(\partial^{6}_y \mu_\epsilon)(x, 0)=-2 (\partial_x\partial_y\tilde{u}_0)(x, 0)(\partial_y\partial^2_x\tilde{u}_0)(x, 0),
$$	
which is not  equal to $0$. So added a corrector is necessary for the initial data of
the regularized system.
\end{remark}

We will prove the  the existence of the approximate solutions of  the system 
\eqref{shear-prandtl-approxiamte} by using the following equation of vorticity 
$ \tilde{w}_\epsilon=\partial_y\tilde{u}_\epsilon $, it reads
\begin{equation}
\label{shear-prandtl-approxiamte-vorticity}
\begin{cases}
& \partial_t\tilde{w}_\epsilon + (u^s + \tilde{u}_\epsilon) \partial_x\tilde{w}_\epsilon +{v}_\epsilon (u^s_{yy} + \partial_y\tilde{w}_{\epsilon})
= \partial^2_{y}\tilde{w}_\epsilon + \epsilon \partial^2_{x}\tilde{w}_\epsilon, \\
& \partial_y\tilde{w}_{\epsilon}|_{y=0}=0,\\
& \tilde{w}_{\epsilon}|_{t=0}=\tilde{w}_{0, \epsilon}=\tilde{w}_0+\epsilon \partial_y\mu_{\epsilon},
\end{cases}
\end{equation}
where 
\begin{equation}\label{u-v-w}
\tilde{u}_\epsilon(t, x, y)=-\int^{+\infty}_y \tilde{w}_\epsilon(t, x, \tilde y) d\tilde y,\quad
\tilde{v}_\epsilon(t, x, y)=-\int^{y}_0\partial_x \tilde{u}_\epsilon(t, x, \tilde y) d\tilde y.
\end{equation}
We have the following theorem for the existence of approximate solutions
\begin{theorem}\label{theorem3.1}
Let $\partial_y \tilde{u}_{0}\in H^{m+2}_{k+\ell}(\mathbb{R}^2_+)$, and $m\ge 6$ be an even integer, $k>1, 0\le \ell<\frac12, k+\ell>\frac32$,
assume that $\tilde{u}_0$ satisfies the compatibility
conditions of order $m+2$ for the system \eqref{non-shear-prandtl}.  Suppose that the shear flow satisfies
$$
|\partial^{p+1}_y u^s(t, y)|\le C\langle y \rangle^{-k-p}, \quad (t, y)\in [0, T_1]\times \mathbb{R}_+,\,\, 0\le p\le m+2.
$$
Then, for any $0<\epsilon\le \epsilon_0$ and $0<\bar\zeta$, there exits $T_\epsilon>0$ which   depends on
$\epsilon$ and $\bar\zeta$,  such that if
$$
\|\tilde{w}_{0}\|_{H^{m+2}_{k+\ell}(\mathbb{R}^2_+)}\le \bar\zeta,
$$
then the system \eqref{shear-prandtl-approxiamte-vorticity}-\eqref{u-v-w} admits a unique solution
$$
\tilde{w}_\epsilon\in L^\infty([0, T_\epsilon]; H^{m+2}_{k+\ell}(\mathbb{R}^2_+)),
$$
which satisfies
\begin{equation}\label{2-estimate}
\|\tilde{w}_\epsilon\|_{L^\infty([0, T_\epsilon];H^m_{k+\ell}(\mathbb{R}^2_+))}\le \frac 43
\|\tilde{w}_{0, \epsilon}\|_{H^m_{k+\ell}(\mathbb{R}^2_+)}\le 2
\|\tilde{w}_0\|_{H^m_{k+\ell}(\mathbb{R}^2_+)}.
\end{equation}
\end{theorem}

\begin{remark}.

\begin{itemize}
\item[(1)] Remark that $T_\epsilon$ depends   on $\epsilon$ and $\bar\zeta$, and
$T_\epsilon\to 0$ as $\epsilon \to 0$. So this is not a bounded estimate
for the approximate solution  sequences $\{u^s+\tilde{u}_\epsilon; 0<\epsilon\le \epsilon_0\}$ where $\epsilon_0>0$ is given in Corollary \ref{coro-boundary}.
When the initial data $\tilde u_{0}$ is small enough, we observe  that $u^s+\tilde{u}_\epsilon$ preserves  the monotonicity and convexity of the shear flow on $[0, T_\epsilon]$.

\item [(2)] In this theorem, for the regularized Prandtl equation, there are not constrain conditions
on the initial date, meaning that we don't need the monotonicity or convexcity of
shear flow $u^s$, and $\bar\zeta$ is also arbitrary.
\end{itemize}
\end{remark}

If $ \tilde{w}_\epsilon$ is a  solution of the system 
\eqref{shear-prandtl-approxiamte-vorticity}-\eqref{u-v-w}, then \eqref{Hardy1} with $\lim_{y\to +\infty} \tilde u_\epsilon=0$ imply
$$
\tilde{u}_\epsilon\in L^\infty([0, T_\epsilon]; H^{m+2}_{k+\ell-1}(\mathbb{R}^2_+)),
$$
and
$$
\tilde v_\epsilon\in L^\infty([0, T_\epsilon]; L^\infty(\mathbb{R}_{y, +}; H^{m+1}(\mathbb{R}_x)).
$$
Integrating the equation 
of \eqref{shear-prandtl-approxiamte-vorticity} over 
$[y, +\infty[$ imply that   $(\tilde u_\epsilon, \tilde v_\epsilon)$ is a solution of the system 
\eqref{shear-prandtl-approxiamte}, except the boundary condition to check: 
\begin{equation}\label{u-w-0}
\tilde{u}_\epsilon(t, x, 0)=-\int^{+\infty}_0 \tilde{w}_\epsilon(t, x, \tilde y) d\tilde y=0,\quad (t, x)\in [0, T_\epsilon]\times \mathbb{R}.
\end{equation}
In fact,  noting $f(t, x)=-\int^{+\infty}_0 \tilde{w}_\epsilon(t, x, \tilde y) d\tilde y
=\tilde{u}_\epsilon(t, x, 0)$, 
a direct calculate give 
\begin{equation} \label{3.00}
\begin{cases}
& \partial_t f+f \partial_x f =  \epsilon \partial^2_{x}f, \quad (t, x)\in ]0, T_\epsilon]\times \mathbb{R};\\
& f|_{t=0}=0,
\end{cases}
\end{equation}
here we use
\begin{align*}
 \int^{+\infty}_0 {v}_\epsilon (u^s_{yy} + \partial_y\tilde{w}_{\epsilon}) dy&=
 \big[{v}_\epsilon (u^s_{y} + \tilde{w}_{\epsilon})\big]^{+\infty}_0
-\int^\infty_0 (\partial_y{v}_\epsilon) (u^s_{y} + \tilde{w}_{\epsilon}) dy\\
&=\int^\infty_0 (\partial_x{u}_\epsilon) \partial_y(u^s + \tilde{u}_{\epsilon}) dy\\
&=
 \big[(\partial_x{u}_\epsilon) (u^s + \tilde{u}_{\epsilon})\big]\big|^{+\infty}_0
-\int^\infty_0 (\partial_x{w}_\epsilon) (u^s+ \tilde{u}_{\epsilon}) dy\\
&=- f \partial_x f
-\int^\infty_0 (\partial_x{w}_\epsilon) (u^s+ \tilde{u}_{\epsilon}) dy.
\end{align*}
Since $f\in L^\infty([0, T_\epsilon], H^{m+2}(\mathbb{R}))$, the uniqueness of solution for equation 
\eqref{3.00} imply that $f=0$ on $[0, T_\epsilon]\times \mathbb{R}$. \eqref{u-w-0} imply also
\begin{equation*}
\tilde{u}_\epsilon(t, x, y)=-\int^{+\infty}_y \tilde{w}_\epsilon(t, x, \tilde y) d\tilde y=
\int^{y}_0 \tilde{w}_\epsilon(t, x, \tilde y) d\tilde y,\quad (t, x, y)\in  [0, T_\epsilon]\times \mathbb{R}^2_+.
\end{equation*}

We will prove Theorem \ref{theorem3.1} by the following three Propositions, where the first one is devoted to the local existence of approximate solution $\tilde{w}_\epsilon$ of \eqref{shear-prandtl-approxiamte-vorticity}.

\begin{proposition}\label{prop3.0}
Let $\tilde{w}_{0, \epsilon}\in H^{m+2}_{k+\ell}(\mathbb{R}^2_+)$, $m\ge 6$ be an even integer, $k>1, 0\le \ell<\frac12, k+\ell> \frac 32$, and satisfy the compatibility conditions up to order $m+2$ for \eqref{shear-prandtl-approxiamte-vorticity}.
Suppose that the shear flow satisfies
$$
|\partial^{p+1}_y u^s(t, y)|\le C\langle y \rangle^{-k-p}, \quad (t, y)\in [0, T_1]\times \mathbb{R}_+,\,\, 0\le p\le m+2.
$$
Then, for any $0<\epsilon\le 1$ and $\bar\zeta>0$, there exits $T_\epsilon>0$  such that if
$$
\|\tilde{w}_{0, \epsilon}\|_{H^{m+2}_{k+\ell}(\mathbb{R}^2_+)}\le \bar\zeta,
$$
then the system \eqref{shear-prandtl-approxiamte-vorticity} admits a unique solution
$$
\tilde{w}_\epsilon\in L^\infty([0, T_\epsilon]; H^{m+2}_{k+\ell}(\mathbb{R}^2_+))\, .
$$
\end{proposition}
\begin{remark}\label{remark3.5}
If $\tilde{w}_{0}\in H^{m+2}_{k+\ell}(\mathbb{R}^2_+)$ is the initial data in Theorem \ref{theorem3.1}, using Corollary \ref{coro-boundary}, there exists $\epsilon_0>0$, and for any $0<\epsilon\le \epsilon_0$, there exists $\mu_\epsilon \in H^{m+3}_{k+\ell}(\mathbb{R}^2_+)$ such that  $\tilde{w}_{0,\epsilon}= \tilde{w}_{0}+\epsilon \partial_y \mu_\epsilon $ satisfies the compatibility conditions up to order $m+2$ for the system \eqref{shear-prandtl-approxiamte-vorticity}, and
$$
\|\tilde{w}_{0, \epsilon}\|_{H^{m+2}_{k+\ell}(\mathbb{R}^2_+)}\le \frac 32 \|\tilde{w}_{0}\|_{H^{m+2}_{k+\ell}(\mathbb{R}^2_+)}.
$$
Then, using Proposition \ref{prop3.0}, we obtain also the existence of the approximate solution under the assumption of Theorem \ref{theorem3.1}.
\end{remark}

The proof of this Proposition is standard since the equation in \eqref{shear-prandtl-approxiamte-vorticity} is a parabolic type equation.
Firstly, we establish the {\it \`a priori} estimate and then prove the existence of solution by the standard iteration and weak convergence methods. Because we work in the weighted  Sobolev space and the computation is not so trivial, we give a detailed proof in the Appendix  \ref{section-a3}, to make the paper self-contained.
So the rest of this section is devoted to proving the  estimate \eqref{2-estimate}.

\noindent {\bf Uniform estimate with loss of $x$-derivative }
In the proof of the Proposition \ref{prop3.0} (see Lemma \ref{lemmab.2}), we already get the {\it \`a priori} estimate for $\tilde{w}_\epsilon$. Now we try to prove  the estimate \eqref{2-estimate} in a new way, and our object is to establish an uniform estimate with respect to $\epsilon>0$. We first treat the easy part in this subsection.

We define the non-isotropic Sobolev norm,
\begin{equation}\label{norm-1}
\|f\|^2_{H^{m, m-1}_{k+\ell}(\mathbb{R}^2_+)}=
\sum_{|\alpha_1+\alpha_2|\le m, \alpha_1\le m-1}\|\langle y\rangle^{k+\ell+\alpha_2}\,\partial^{\alpha_1}_x \partial^{\alpha_2}_y f\|_{L^{2}(\mathbb{R}^2_+)}^2,
\end{equation}
where we  don't have the $m$-order derivative with respect to $x$-variable.  Then
$$
\|f\|^2_{H^{m}_{k+\ell}(\mathbb{R}^2_+)}=
\|f\|^2_{H^{m,m-1}_{k+\ell}(\mathbb{R}^2_+)}+\|\partial^m_x
f\|^2_{L^{2}_{k+\ell}(\mathbb{R}^2_+)}.
$$

\begin{proposition}\label{prop3.1}
Let $m\ge 6$ be an even integer, $k>1, 0< \ell<\frac12, k+\ell> \frac 32$, and assume that $\tilde{w}_\epsilon\in L^\infty([0, T_\epsilon]; H^{m+2}_{k+\ell}(\mathbb{R}^2_+))$ is a solution to \eqref{shear-prandtl-approxiamte-vorticity}, then we have
	\begin{equation}
	\label{approx-less-k}
	\begin{split}
	&\frac{d}{dt}\|\tilde{w}_\epsilon\|^2_{H^{m, m-1}_{k+\ell}(\mathbb{R}^2_+)}+ \|\partial_y\tilde{w}_\epsilon\|^2_{H^{m, m-1}_{k+\ell}(\mathbb{R}^2_+)}\\
	&\qquad+ \epsilon\|\partial_x\tilde{w}_\epsilon\|^2_{H^{m, m-1}_{k+\ell}(\mathbb{R}^2_+)}
	\le C_1\bigg( \| \tilde{w}_\epsilon\|_{H^m_{k+\ell}(\mathbb{R}^2_+)}^2 + \| \tilde{w}_\epsilon\|_{H^m_{k+\ell}(\mathbb{R}^2_+)}^m \bigg),
	\end{split}
	\end{equation}
	where $C_1>0$ is independent of $\epsilon$.
\end{proposition}

\noindent
{\bf Remark.} The above estimate is uniform with respect to $\epsilon>0$, but on the left hand of \eqref{approx-less-k}, we missing the terms $\|\partial^{m}_x\tilde{w}_\epsilon\|_{L^{2}_{k+\ell}}^2$. This is because that we can't control the term
$$
\partial^{m}_x\tilde{v}_\epsilon(t, x, y)=- \int^y_0\partial^{m+1}_x\tilde{u}_\epsilon(t, x, \tilde{y}) d\tilde{y},
$$
which is the major difficulty in the study of the Prandtl equation. We will study this term in the next Proposition with
a non-uniform estimate firstly, and then focus on proving the uniform estimate in  the rest part of this paper.

\begin{proof}
	For $|\alpha|=\alpha_1+\alpha_2\le m, \alpha_1\le m-1$, we have
\begin{equation}\label{non-approx-est-less-s}
\begin{split}
	&\partial_t \partial^{\alpha} \tilde{w}_\epsilon - \epsilon \partial^2_x\partial^{\alpha} \tilde{w}_\epsilon - \partial_y^2 \partial^{\alpha}\partial\tilde{w}_\epsilon \\
	&= - \partial^{\alpha} \big((u^s + \tilde{u}_\epsilon)\partial_x  \tilde{w}_\epsilon \big) - \partial^{\alpha} \big( \tilde{v}_\epsilon ( u^s_{yy}+ \partial_y\tilde{w}_{\epsilon} ) \big).
\end{split}
\end{equation}	
Multiplying the \eqref{non-approx-est-less-s} with $ \langle y \rangle^{2(k+\ell+{\alpha_2})} \partial^{\alpha} \tilde{w}_\epsilon $, and integrating over $\mathbb{R}^2_+$,
\begin{equation*}
\begin{split}
	&\int_{\mathbb{R}^2_+} (\partial_t  \partial^{\alpha} \tilde{w}_\epsilon) \langle y \rangle^{2(k+\ell)+2{\alpha_2}} \partial^{\alpha} \tilde{w}_\epsilon dx dy - \epsilon
\int_{\mathbb{R}^2_+} (\partial^2_x\partial^{\alpha} \tilde{w}_\epsilon) \langle y \rangle^{2(k+\ell)+2{\alpha_2}} \partial^{\alpha} \tilde{w}_\epsilon dx dy \\
&\qquad\qquad- \int_{\mathbb{R}^2_+} (\partial^2_y\partial^{\alpha} \tilde{w}_\epsilon) \langle y \rangle^{2(k+\ell)+2{\alpha_2}} \partial^{\alpha} \tilde{w}_\epsilon dx dy  \\
	&= - \int_{\mathbb{R}^2_+} \partial^{\alpha} \big((u^s + \tilde{u}_\epsilon)\partial_x  \tilde{w}_\epsilon -  \tilde{v}_\epsilon ( u^s_{yy}+ \partial_y\tilde{w}_{\epsilon} )\big) \langle y \rangle^{2(k+\ell)+2{\alpha_2}} \partial^{\alpha}
\tilde{w}_\epsilon dx dy .
\end{split}
\end{equation*}
Remark that for $\tilde{w}_\epsilon\in L^\infty([0, T_\epsilon]; H^{m+2}_{k+\ell}(\mathbb{R}^2_+))$, all above integrations are in the classical sense. We deal with each term on the left hand respectively.  After integration by part, we have
\begin{align*}
	& \int_{\mathbb{R}^2_+} (\partial_t  \partial^{\alpha} \tilde{w}_\epsilon) \langle y \rangle^{2(k+\ell)+2{\alpha_2}(\mathbb{R}^2_+)} \partial^{\alpha} \tilde{w}_\epsilon dx dy =\frac 12 \frac{d}{ dt}\| \partial^{\alpha} \tilde{w}_\epsilon\|_{L^2_{k+\ell+\alpha_2}(\mathbb{R}^2_+)}^2,\\
	& -\epsilon\int_{\mathbb{R}^2_+}  ( \partial_x^{2}  \partial^{\alpha} \tilde{w}_\epsilon )\langle y \rangle^{2(k+\ell)+2{\alpha_2}(\mathbb{R}^2_+)} \partial^{\alpha} \tilde{w}_\epsilon dx dy =\epsilon\|\partial_x  \partial^{\alpha} \tilde{w}_\epsilon\|_{L^2_{k+\ell+\alpha_2}(\mathbb{R}^2_+)}^2,
	\end{align*}
and
\begin{align*}
&\quad  -\int_{\mathbb{R}^2_+} \partial_y^{2} \partial^{\alpha} \tilde{w}_\epsilon \langle y \rangle^{2(k+\ell)+2{\alpha_2}} \partial^{\alpha} \tilde{w}_\epsilon dx dy \\
&=\|\partial_y  \partial^{\alpha} \tilde{w}_\epsilon\|_{L^2_{k+\ell+\alpha_2}(\mathbb{R}^2_+)}^2+\int_{\mathbb{R}^2_+}  \partial^{\alpha}\partial_y\tilde{w}_\epsilon (\langle y \rangle^{2(k+\ell)+2{\alpha_2}} )'\partial^{\alpha} \tilde{w}_\epsilon dx dy\\
&\qquad\qquad+\int_{\mathbb{R}}  \left(\partial^{\alpha}\partial_y\tilde{w}_\epsilon \partial^{\alpha} \tilde{w}_\epsilon\right)\big|_{y=0} dx .
\end{align*}
Cauchy-Schwarz inequality implies
\begin{align*}
&\left|\int_{\mathbb{R}^2_+}  \partial^{\alpha}\partial_y\tilde{w}_\epsilon (\langle y \rangle^{2(k+\ell)+2{\alpha_2}} )'\partial^{\alpha} \tilde{w}_\epsilon dx dy\right|\\
&\qquad\le\frac{1}{16} \|\partial_y \partial^{\alpha} \tilde{w}_\epsilon\|_{L^2_{k+\ell+\alpha_2}(\mathbb{R}^2_+)}^2 +C\| \partial^{\alpha} \tilde{w}_\epsilon\|_{L^2_{k+\ell+\alpha_2-1}(\mathbb{R}^2_+)}^2.
\end{align*}
We study now the term
$$
\int_{\mathbb{R}}  \left(\partial^{\alpha}\partial_y\tilde{w}_\epsilon \partial^{\alpha} \tilde{w}_\epsilon\right)\big|_{y=0} dx.
$$
{\bf Case : $|\alpha|\le m-1$}, using the trace Lemma \ref{lemma-trace}, we have
\begin{align*}
\left|\int_{\mathbb{R}}  \left(\partial^{\alpha}\partial_y\tilde{w}_\epsilon \partial^{\alpha} \tilde{w}_\epsilon\right)\big|_{y=0} dx\right|
&\le\|(\partial^{\alpha}\partial_y\tilde{w}_\epsilon)|_{y=0}\|_{L^2(\mathbb{R})}
\|(\partial^{\alpha}\tilde{w}_\epsilon)|_{y=0}\|_{L^2(\mathbb{R})}\\
&\le C\|\partial^{\alpha}\partial^2_y\tilde{w}_\epsilon\|_{L^2_{k+\ell}(\mathbb{R}^2_+)}
\|\partial^{\alpha}\partial_y\tilde{w}_\epsilon\|_{L^2_{k+\ell}(\mathbb{R}^2_+)}\\
&\le C\|\partial_y\tilde{w}_\epsilon\|_{H^{m, m-1}_{k+\ell}(\mathbb{R}^2_+)}\|\tilde{w}_\epsilon\|_{H^{m}_{k+\ell}
(\mathbb{R}^2_+)}\\
&\le \frac 1{16}\|\partial_y\tilde{w}_\epsilon\|^2_{H^{m, m-1}_{k+\ell}(\mathbb{R}^2_+)}+C\|\tilde{w}_\epsilon\|^2_{H^{m}_{k+\ell}
(\mathbb{R}^2_+)}.
\end{align*}
{\bf Case : $\alpha_1=m-1, \alpha_2=1$}, using \eqref{boundary-12b}, we have
$$
(\partial^{\alpha}\tilde{w}_\epsilon)|_{y=0}=
(\partial_x^{\alpha_1}\partial_y^{2} \tilde{u}_\epsilon)|_{y=0} = 0,
$$
thus
$$
\int_{\mathbb{R}}  \left(\partial^{\alpha}\partial_y\tilde{w}_\epsilon \partial^{\alpha} \tilde{w}_\epsilon\right)|_{y=0} dx=0.
$$
{\bf Case : $\alpha_1=0, \alpha_2=m$}. Only in this case, we need to suppose that $m$ is even. Using again the trace Lemma \ref{lemma-trace}, we have
\begin{align*}
\left|\int_{\mathbb{R}}  \left(\partial^{m+1}_y\tilde{w}_\epsilon \partial^{m}_y \tilde{w}_\epsilon\right)|_{y=0} dx\right|
&\le\|(\partial^{m+2}_y\tilde{u}_\epsilon)|_{y=0}\|_{L^2(\mathbb{R})}
\|(\partial^{m}_y\tilde{w}_\epsilon)|_{y=0}\|_{L^2(\mathbb{R})}\\
&\le C\|(\partial^{m+2}_y\tilde{u}_\epsilon)|_{y=0}\|_{L^2(\mathbb{R})}
\|\partial^{m+1}_y\tilde{w}_\epsilon\|_{L^2_{k+\ell}(\mathbb{R}^2_+)}\\
&\le \frac 1{16}\|\partial_y\tilde{w}_\epsilon\|^2_{H^{m, m-1}_{k+\ell}(\mathbb{R}^2_+)}+C\|(\partial^{m+2}_y\tilde{u}_\epsilon)
|_{y=0}\|^2_{L^2(\mathbb{R})}.
\end{align*}
Using Proposition \ref{prop-comp-b} and the trace Lemma \ref{lemma-trace}, we can estimate the above last term $\|(\partial^{m+2}_y\tilde{u}_\epsilon)
|_{y=0}\|^2_{L^2(\mathbb{R})}$ by a finite summation of the following forms
 $$
  \|\prod^{p}_{j=1} (\partial^{\alpha_j}_x\partial^{\beta_j+1}_y( u^s + \tilde u_\epsilon))|_{y=0}\|^2_{L^2(\mathbb{R})}\le C\|\partial_y\prod^{p}_{j=1} (\partial^{\alpha_j}_x\partial^{\beta_j+1}_y( u^s + \tilde u_\epsilon)) \|^2_{L^2_{\frac 12+\delta}(\mathbb{R}^2_+)}
$$
with $2\le p\le \frac m2$, $\alpha_j + \beta_j \le m -1$ and $\{j; \alpha_j>0\}\not=\emptyset$.
Then using Sobolev inequality and $m\ge 6$, we get
$$
\|(\partial^{m+2}_y\tilde{u}_\epsilon)
|_{y=0}\|_{L^2(\mathbb{R})}\le C \|\tilde{w}_\epsilon\|^{m/2}_{H^{m}_{k+\ell}
(\mathbb{R}^2_+)}.
$$
{\bf Case : $1\le \alpha_1\le m-2, \alpha_1+\alpha_2=m, \alpha_2$ even}, using the same argument  to the precedent case, we have
\begin{align*}
&\left|\int_{\mathbb{R}}(\partial^{\alpha}\partial_y\tilde{w}_\epsilon \partial^{\alpha} \tilde{w}_\epsilon)|_{y=0}dx \right|=\left|\int_{\mathbb{R}}
(\partial^{\alpha_1}_x\partial^{\alpha_2+1}_y\tilde{w}_\epsilon \partial^{\alpha_1}_x\partial^{\alpha_2}_y \tilde{w}_\epsilon)|_{y=0}dx \right|\\
&\qquad\le \|(\partial^{\alpha_1}_x\partial^{\alpha_2+1}_y\tilde{w}_\epsilon)|_{y=0}\|_{L^2(\mathbb{R})} \|(\partial^{\alpha_1}_x\partial^{\alpha_2}_y \tilde{w}_\epsilon)|_{y=0}\|_{L^2(\mathbb{R})}\\
&\qquad\le \frac 1{16}\|\partial_y\tilde{w}_\epsilon\|^2_{H^{m, m-1}_{k+\ell}(\mathbb{R}^2_+)}+C\|(\partial^{\alpha_1}_x
\partial^{\alpha_2+2}_y\tilde{u}_\epsilon)|_{y=0}\|^2_{L^2(\mathbb{R})}\\
&\qquad\le \frac 1{16}\|\partial_y\tilde{w}_\epsilon\|^2_{H^{m, m-1}_{k+\ell}(\mathbb{R}^2_+)}+C\|\tilde{w}_\epsilon\|^{\alpha_2}_{H^{m}_{k+\ell}
(\mathbb{R}^2_+)}.
\end{align*}

{\bf Case : $1\le \alpha_1\le m-2, \alpha_1+\alpha_2=m, \alpha_2$ odd}, integration by part with respect to $x$ variable implies
\begin{align*}
&\left|\int_{\mathbb{R}}
(\partial^{\alpha_1}_x\partial^{\alpha_2+1}_y\tilde{w}_\epsilon \partial^{\alpha_1}_x\partial^{\alpha_2}_y \tilde{w}_\epsilon)|_{y=0}dx \right|=\left|\int_{\mathbb{R}}
(\partial^{\alpha_1-1}_x\partial^{\alpha_2+1}_y\tilde{w}_\epsilon \partial^{\alpha_1+1}_x\partial^{\alpha_2}_y \tilde{w}_\epsilon)|_{y=0}dx \right|\\
&\qquad\le \|(\partial^{\alpha_1-1}_x\partial^{\alpha_2+1}_y\tilde{w}_\epsilon)|_{y=0}\|_{L^2(\mathbb{R})} \|(\partial^{\alpha_1+1}_x\partial^{\alpha_2}_y \tilde{w}_\epsilon)|_{y=0}\|_{L^2(\mathbb{R})}\\
&\qquad\le \frac 1{16}\|\partial_y\tilde{w}_\epsilon\|^2_{H^{m, m-1}_{k+\ell}(\mathbb{R}^2_+)}+C\|(\partial^{\alpha_1+1}_x
\partial^{\alpha_2+1}_y\tilde{u}_\epsilon)|_{y=0}\|^2_{L^2(\mathbb{R})}\\
&\qquad\le \frac 1{16}\|\partial_y\tilde{w}_\epsilon\|^2_{H^{m, m-1}_{k+\ell}(\mathbb{R}^2_+)}+C\|\tilde{w}_\epsilon\|^{\alpha_2-1}_{H^{m}_{k+\ell}
(\mathbb{R}^2_+)}.
\end{align*}

Finally, we have proven
\begin{align*}
\begin{split}
	& \int_{\mathbb{R}^2_+} \big(\partial_t \partial^{\alpha} \tilde{w}_\epsilon - \partial_y^2\partial^{\alpha} \tilde{w}_\epsilon - \epsilon \partial^2_x\partial^{\alpha} \tilde{w}_\epsilon\big) \langle y \rangle^{2 (k+\ell+\alpha_2)} \partial^{\alpha} \tilde{w}_\epsilon dx dy\\
	& \ge\frac12 \frac{d}{ dt}\|\partial^{\alpha} \tilde{w}_\epsilon\|_{L^2_{k+\ell+\alpha_2}}^2+ \epsilon\|\partial_x\partial^{\alpha} \tilde{w}_\epsilon\|_{L^2_{k+\ell+\alpha_2}}^2 +\|\partial_y \partial^{\alpha} \tilde{w}_\epsilon\|_{L^2_{k+\ell+\alpha_2}}^2\\
	&\qquad-\frac{1}{4} \|\partial_y\tilde{w}_\epsilon\|^2_{H^{m, m-1}_{k+\ell}(\mathbb{R}^2_+)}-C\|\tilde{w}_\epsilon\|^{m}_{H^{m}_{k+\ell}
(\mathbb{R}^2_+)}.
\end{split}
\end{align*}
	
We estimate now the right hand of \eqref{non-approx-est-less-s}. For the first item, we need to split it into two parts
	\begin{align*}
	- \partial^{\alpha} \big((u^s + \tilde{u}_\epsilon)\partial_x  \tilde{w}_\epsilon \big) =  -   (u^s + \tilde{u}_\epsilon)\partial_x \partial^{\alpha} \tilde{w}_\epsilon +  [ (u^s + \tilde{u}_\epsilon), \partial^{\alpha}]\partial_x \tilde{w}_\epsilon.
	\end{align*}
	Firstly, we have
	\begin{align*}
	\int_{\mathbb{R}^2_+} \big((u^s + \tilde{u}_\epsilon) \partial_x \partial^{\alpha} \tilde{w}_\epsilon\big) \langle y \rangle^{2(k+\ell+{\alpha_2})}\partial^{\alpha} \tilde{w}_\epsilon dx dy \le  \|\partial_x \tilde{u}_\epsilon\|_{L^\infty}\|\partial^{\alpha} \tilde{w}_\epsilon\|_{L^2_{k+\ell+{\alpha_2}}}^2,
	\end{align*}
	then using \eqref{sobolev-1}, we get
	\begin{align*}
	\left|\int_{\mathbb{R}^2_+} \big((u^s + \tilde{u}_\epsilon) \partial_x \partial^{\alpha} \tilde{w}_\epsilon\big)\langle y \rangle^{2(\ell+{\alpha_2})}\partial^{\alpha} \tilde{w}_\epsilon dx dy\right| \le  \| \tilde{w}_\epsilon\|_{H^3_1}\|\partial^{\alpha} \tilde{w}_\epsilon\|_{L^2_{k+\ell+{\alpha_2}}}^2.
	\end{align*}
	
	For the commutator operator, in fact, it can be written as
	\begin{align*}
	&  [ (u^s + \tilde{u}_\epsilon), \partial^{\alpha}]\partial_x \tilde{w}_\epsilon = \sum\limits_{\beta \le \alpha,\, 1\le|\beta|}C^\beta_\alpha \,\,\partial^{\beta}( u^s + \tilde{u}_\epsilon)  \partial^{\alpha - \beta}\partial_x \tilde{w}_\epsilon.
	\end{align*}
	Then for $|\alpha|\le m, m\ge 4$, using the Sobolev inequality again and Lemma \ref{inequality-hardy},
	$$
	\| [ (u^s + \tilde{u}), \partial^{\alpha}]\partial_x \tilde{w}_\epsilon\|_{L^2_{k+\ell+\alpha_2}}\le C( \|\tilde{w}_\epsilon\|_{H^m_{k+\ell}}
	+\|\tilde{w}_\epsilon\|^2_{H^m_{k+\ell}}).
	$$
	Thus
	\begin{align*}
	\left|\int_{\mathbb{R}^2_+} \langle y \rangle^{2(k+\ell+\alpha_2)}\big(  [ (u^s + \tilde{u}_\epsilon), \partial^{\alpha}]\partial_x \tilde{w}_\epsilon \big)\cdot \partial^{\alpha}  \tilde{w}_\epsilon dx dy \right|\le C \big(  \|\tilde{w}_\epsilon\|^2_{H^m_{k+\ell}}
	+\|\tilde{w}_\epsilon\|^3_{H^m_{k+\ell}} \big) ,
	\end{align*}
	and
	\begin{align*}
	\left|\int_{\mathbb{R}^2_+} \langle y \rangle^{2(k+\ell+\alpha_2)}\big(  \partial^{\alpha} \big( (u^s + \tilde{u}_\epsilon)\partial_x  \tilde{w}_\epsilon \big)\big)\partial^{\alpha}  \tilde{w}_\epsilon dx dy\right| \le C \big(  \|\tilde{w}_\epsilon\|^2_{H^m_{k+\ell}}
	+\|\tilde{w}_\epsilon\|^3_{H^m_{k+\ell}} \big),
	\end{align*}
	where $C$ is independent of $\epsilon$.
	
For the next one, similar to  the first term in \eqref{non-approx-est-less-s}, we have
	\begin{align*}
	\partial^{\alpha} \big( \tilde{v}_\epsilon ( u^s_{yy}+\partial_y\tilde{w}_\epsilon ) \big) & = \tilde{v}_\epsilon \partial_y \partial^{\alpha} \tilde{w}_\epsilon -  [\tilde{v}_\epsilon, \partial^{\alpha} ] \partial_y \tilde{w}_\epsilon+ \partial^{\alpha}
	(\tilde{v}_\epsilon u^s_{yy} ).
	\end{align*}
	Then
	\begin{align*}
	\left|\int_{\mathbb{R}^2_+} \tilde{v}_\epsilon \langle y \rangle^{ 2(k+\ell+\alpha_2)} (\partial_y\partial^{\alpha} \tilde{w}_\epsilon)\cdot \partial^{\alpha} \tilde{w}_\epsilon dx dy\right|
	&\le \|\tilde{v}_\epsilon\|_{L^\infty(\mathbb{R}^2_+)}
\|\partial_y \tilde{w}_\epsilon\|_{H^m_{k+\ell}}\| \tilde{w}_\epsilon\|_{H^m_{k+\ell}}\\
&\le \frac 1{4} \|\partial_y \tilde{w}_\epsilon\|^2_{H^m_{k+\ell}(\mathbb{R}^2_+)}+C\| \tilde{w}_\epsilon\|^4_{H^m_{k+\ell}(\mathbb{R}^2_+)}
	\end{align*}
	where we have used
	\begin{equation*}
	\begin{split}
	&\|\tilde{v}_{\epsilon}\|_{L^\infty(\mathbb{R}^2_+)}  \le C\|\partial_x\tilde{u}_{\epsilon}\|_{L^\infty(\mathbb{R}_x; L^2_{\frac 12 +\delta}(\mathbb{R}_{y, +}))}\\
	&\le C \int_{\mathbb{R}^2_+}\langle y\rangle ^{1+2\delta}(|\partial_x\tilde{u}_{\epsilon}|^2+|\partial^2_x\tilde{u}_{\epsilon}|^2)
	dx dy\\
	&\le C \int_{\mathbb{R}^2_+}\langle y\rangle ^{3+2\delta}(|\partial_x\tilde{w}_{\epsilon}|^2+|\partial^2_x\tilde{w}_{\epsilon}|^2)
	dx dy\le C\| \tilde{w}_\epsilon\|_{H^2_{\frac 32 +\delta}},
	\end{split}
	\end{equation*}
	where $\delta>0$ is small.
	
	Noticing that
	\begin{align*}
	[ \tilde{v}_\epsilon, \partial^{\alpha} ] \partial_y \tilde{w}_\epsilon & = \sum\limits_{ \beta \le \alpha, 1\le |\beta|} C^\beta_\alpha \,\,\, \partial^{\beta} \tilde{v}_\epsilon \,\,\partial^{\alpha - \beta}\partial_y \tilde{w}_\epsilon.
	\end{align*}
	Since $H^m_\ell$ is an algebra for $m\ge 6$, we only need to pay attention to the order of derivative in the above formula.	 Firstly for $|\beta|\ge 1$, we have for $|\alpha-\beta|+1\le m$,
	$$
	-\partial^{\beta} \tilde{v}_\epsilon=\partial^{\beta_1}_x \partial^{\beta_2}_y \int^y_0\tilde{u}_{\epsilon, x} d\tilde {y}
	=\left\{\begin{array}{ll}
	\partial^{\beta_1+1}_x \partial^{\beta_2-1}_y \tilde{u}_{\epsilon},&\quad \beta_2\ge 1,\\
	\int^y_0 \partial^{\beta_1+1}_x  \tilde{u}_{\epsilon} d\tilde {y},&\quad\beta_2=0.
	\end{array}\right.
	$$
	Now using the hypothesis $\beta\le \alpha, 1\le |\beta|$ and $\beta_1\le \alpha_1\le m-1$, using Lemma \ref{inequality-hardy}, we get
	$$
	\| [ \tilde{v}_\epsilon, \partial^{\alpha} ] \partial_y \tilde{w}_\epsilon \|_{L^2_{k+\ell+\alpha_2}}\le C\|\tilde{w}_\epsilon\|^2_{H^m_{k+\ell}}.
	$$
	On the other hand, if $\alpha_2=0$, using $-1+\ell<-\frac 12$, we can get
	$$
	\| \partial^{m-1}_x( \tilde{v}_\epsilon u^s_{yy}) \|_{L^2_{k+\ell}}\le
	C\|\partial^{m}_x\tilde{u}_\epsilon\|_{L^2_{\frac 12+\delta}(\mathbb{R}^2_+)}\|u^s_{yy} \|_{L^2_{k+\ell}(\mathbb{R}_+)}
	\le C\|\tilde{w}_\epsilon\|_{H^m_{\frac 32+\delta}}.
	$$
	Similar computation for other cases, we can get, for $\alpha_2>0, \alpha_1+\alpha_2\le m$,
	$$
	\| \partial^{\alpha}( \tilde{v}_\epsilon u^s_{yy}) \|_{L^2_{k+\ell+\alpha_2}}\le C\|\tilde{w}_\epsilon\|_{H^m_{k+\ell}}.
	$$
Combining the above estimates, we have finished the proof of the Proposition \ref{prop3.1}.
\end{proof}

\noindent {\bf Smallness of approximate solutions.} To close the energy estimate, we still need  to estimate the term $\partial^m_x \tilde{w}_\epsilon$.

\begin{proposition} \label{prop3.2}
	Under the hypothesis
	of Theorem \ref{theorem3.1}, and with the same notations as in Proposition \ref{prop3.1}, we have
	\begin{align}
	\label{approx-k}
	\begin{split}
	& \frac 12\frac{d}{dt}\|\partial_x^m \tilde{w}_\epsilon\|_{L^2_{k+\ell}}^2 + \frac{3\epsilon}{4}\|\partial_x^{m+1}\tilde{w}_\epsilon\|_{L^2_{k+\ell}}^2 + \frac{3}{4} \|\partial_y \partial_x^m \tilde{w}_\epsilon\|_{L^2_{k+\ell}}^2 \\
	& \le C\big( \| \tilde{w}_\epsilon\|_{H^m_{k+\ell}}^2 + \| \tilde{w}_\epsilon\|_{H^m_{k+\ell}}^3 \big) + \frac{32}{\epsilon}\big(\|\tilde{w}_\epsilon\|_{H^m_{k+\ell}}^4+\|\tilde{w}_\epsilon\|_{H^m_{k+\ell}}^2\big).
	\end{split}
	\end{align}
\end{proposition}

\begin{proof}
	We have
	\begin{align*}
	\partial_t \partial_x^m  \tilde{w}_\epsilon - \partial_y^2 \partial_x^m \tilde{w}_\epsilon - \epsilon \partial_x^m\partial^2_x \tilde{w}_{\epsilon}= - \partial_x^m  \big((u^s + \tilde{u}_\epsilon)\partial_x  \tilde{w}_\epsilon \big) - \partial_x^m \big( \tilde{v}_\epsilon (\partial_y\tilde{w}_\epsilon  + u^s_{yy}) \big),
	\end{align*}
	then the same computations as in Proposition \ref{prop3.1} give
	\begin{align}
	\label{approx-k-part-1}
	\begin{split}
	& \frac{d}{2 dt}\|\partial_x^m \tilde{w}_\epsilon\|_{L^2_{{k+\ell}}}^2 + \epsilon\|\partial_x^{m+1} \tilde{w}_\epsilon\|_{L^2_{k+\ell}}^2 + \frac{3}{4} \|\partial_y \partial_x^m\tilde{w}_\epsilon\|_{L^2_{k+\ell}}^2\\
	&\le C( \|\tilde{w}_\epsilon \|_{H^m_{k+\ell}}^2+\|\tilde{w}_\epsilon \|_{H^m_{k+\ell}}^3)\\
	&+\left|\int_{\mathbb{R}^2_+} \partial_x^m \big( \tilde{v}_\epsilon (\partial_y\tilde{w}_\epsilon  + u^s_{yy}) \big)  \langle y \rangle^{2(k+\ell)} \partial_x^m \tilde{w}_\epsilon dxdy\right|,
	\end{split}
	\end{align}
where the boundary terms is more easy to control, since
$$
(\partial_y \partial_x^m\tilde{w}_\epsilon)(t, x, 0)=(\partial^2_y \partial_x^m\tilde{u}_\epsilon)(t, x, 0)=0,\,\,\,(t, x)\in [0, T]\times \mathbb{R}.
$$
The estimate of the last term on right hand  is the main obstacle for the study
of the Prandtl equations.
	\begin{align*}
	\partial_x^m \big( \tilde{v}_\epsilon (\partial_y\tilde{w}_\epsilon  + u^s_{yy}) \big) & = \tilde{v}_\epsilon \partial_x^m   \partial_y \tilde{w}_\epsilon +  (\partial^m_x\tilde{v}_\epsilon) (\partial_y \tilde{w}_\epsilon+u^s_{yy})\\
	&\quad +\sum_{1\le j\le m-1}C^j_m\, \partial^j_x\tilde{v}_\epsilon\partial^{m-j}_x \partial_y\tilde{w}_\epsilon.
	\end{align*}
For the first term	
	\begin{align*}
	\int_{\mathbb{R}^2_+} \tilde{v}_\epsilon(\partial_x^m \partial_y  \tilde{w}_\epsilon)  \langle y \rangle^{ 2(k+\ell)} (\partial_x^m   \tilde{w}_\epsilon)dx dy 
	&=\frac{1}{2}\int \tilde{v}_\epsilon \langle y \rangle^{ 2(k+\ell)} \partial_y(\partial_x^m \tilde{w}_\epsilon)^2 dx dy \\
	& =
	\frac{1}{2}\int \tilde{u}_{\epsilon, x} \langle y \rangle^{ 2(k+\ell)} (\partial_x^m  \tilde{w}_\epsilon)^2 dx dy\\
	& - \ell \int \tilde{v}_\epsilon \langle y \rangle^{ 2(k+\ell) - 1} (\partial_x^m   \tilde{w}_\epsilon)^2 dx dy\\
	& \le C \|\tilde{w}_\epsilon\|_{H^m_{k+\ell}}^3,
	\end{align*} 	
where we have used $\tilde{v}_\epsilon|_{y=0}=0$, and
	\begin{align*}
	\left|\int_{\mathbb{R}^2_+} \big(\sum_{1\le j\le m-1}C^j_m\, \partial^j_x\tilde{v}_\epsilon\partial^{m-j}_x \partial_y\tilde{w}_\epsilon)  \langle y \rangle^{ 2(k+\ell)} (\partial_x^m   \tilde{w}_\epsilon\big)dx dy\right|  \le C \|\tilde{w}_\epsilon\|_{H^m_{k+\ell}}^3.
	\end{align*} 	
	Finally for the worst term, we have
	\begin{align*}
	&\left|\int_{\mathbb{R}^2_+} (\partial^m_x \tilde{v}_\epsilon)(\partial_y  \tilde{w}_\epsilon
	+u^s_{yy})  \langle y \rangle^{ 2(k+\ell)} (\partial_x^m   \tilde{w}_\epsilon)dx dy \right|\\
	&\le C\|\partial^m_x \tilde{v}_\epsilon\|_{L^2(\mathbb{R}_x; L^\infty(\mathbb{R}_+))}  \|\partial_y\tilde{w}_\epsilon\|_{L^\infty(\mathbb{R}_x; L^2_{k+\ell}(\mathbb{R}_+))} \|\tilde{w}_\epsilon\|_{H^m_{k+\ell}}\\
	&\qquad\qquad +\|\partial^m_x \tilde{v}_\epsilon u^s_{yy}\|_{L^2_{k+\ell}(\mathbb{R}^2_+)}  \|\tilde{w}_\epsilon\|_{H^m_{k+\ell}}.
	\end{align*} 	
	On the other hand, observing
	$$
	\partial^m_x \tilde{v}_\epsilon (t, x, y)=-\int^y_0 \partial^{m+1}_x \tilde{u}_\epsilon(t, x, \tilde{y})d\tilde y,
	$$
then using Sobolev inequality and Lemma \ref{inequality-hardy}, for $\delta>0$ small,
	$$
	\|\partial^m_x \tilde{v}_\epsilon\|_{L^2(\mathbb{R}_x; L^\infty(\mathbb{R}_+))}\le C\|\partial^{m+1}_x \tilde{u}_\epsilon\|_{L^2_{\frac 12+\delta}(\mathbb{R}^2_+)} \le C\|\partial^{m+1}_x \tilde{w}_\epsilon\|_{L^2_{\frac 32+\delta}(\mathbb{R}^2_+)},
	$$
we get
	$$
	\|\partial^m_x \tilde{v}_\epsilon\|_{L^2(\mathbb{R}_x; L^\infty(\mathbb{R}_+))} \le C\|\partial^{m+1}_x \tilde{w}_\epsilon\|_{L^2_{\frac 32+\delta}(\mathbb{R}^2_+)}.
	$$
	Using  the hypothesis for the shear flow $u^s$ and $\ell-1<-\frac 12$,
	\begin{align*}
	\|\partial^m_x (\tilde{v}_\epsilon u^s_{yy})\|_{L^2_{k+\ell}(\mathbb{R}^2_+)}&\le \|\partial^m_x \tilde{v}_\epsilon\|_{L^2(\mathbb{R}_x; L^\infty(\mathbb{R}_+))}\| u^s_{yy}\|_{L^2_{k+\ell}(\mathbb{R}_+)}\\
	&\le C\|\partial^{m+1}_x \tilde{w}_\epsilon\|_{L^2_{\frac 32+\delta}(\mathbb{R}^2_+)},
	\end{align*}
	and for $k+\ell\ge \frac 32+\delta$,
	$$
	\|\partial_y\tilde{w}_\epsilon\|_{L^\infty(\mathbb{R}_x; L^2_{k+\ell}(\mathbb{R}_+))} \le C \|\partial_y\tilde{w}_\epsilon\|_{H^1(\mathbb{R}_x; L^2_{k+\ell}(\mathbb{R}_+))}\le C\|\tilde{w}_\epsilon\|_{H^m_{k+\ell}(\mathbb{R}^2_+)}.
	$$
Thus, we have
	\begin{align}
	\label{approx-k-part-3}
	\begin{split}
	&\int  \big(\partial_x^m  \big( \tilde{v}_\epsilon (\partial_y\tilde{w}_\epsilon  + u^s_{yy}) \big)\big) \, \langle y \rangle^{2(k+\ell)} \partial_x^m \tilde{w}_\epsilon dx dy \\
	& \le C \| \tilde{w}_\epsilon \|_{H^m_{k+\ell}}^3 + \frac{32}{\epsilon}\big(\|\tilde{w}_\epsilon\|_{H^m_{k+\ell}}^4+\|\tilde{w}_\epsilon\|_{H^m_{k+\ell}}^2\big) + \frac{\epsilon}{4}\| \partial_x^{m+1} \tilde{w}_\epsilon\|_{L^2_{\frac 32+\delta}}^2.
	\end{split}
	\end{align}
	From \eqref{approx-k-part-1} and \eqref{approx-k-part-3}, we have, if $k+\ell >\frac 32$,
	\begin{align*}
	\begin{split}
	&\frac 12 \frac{d}{dt}\|\partial_x^m \tilde{w}_\epsilon\|_{L^2_{k+\ell}}^2 + \frac{3\epsilon}{4}\|\partial_x^{m+1}\tilde{w}_\epsilon\|_{L^2_{k+\ell}}^2 + \frac{3}{4} \|\partial_y \partial_x^m \tilde{w}_\epsilon\|_{L^2_{k+\ell}}^2 \\
	& \le C\big( \| \tilde{w}_\epsilon\|_{H^m_{k+\ell}}^2 + \| \tilde{w}_\epsilon\|_{H^m_{k+\ell}}^3 \big) + \frac{32}{\epsilon}\big(\|\tilde{w}_\epsilon\|_{H^m_{k+\ell}}^4+\|\tilde{w}_\epsilon\|_{H^m_{k+\ell}}^2\big).
	\end{split}
	\end{align*}
\end{proof}

\begin{proof}[{\bf End of proof of Theorem \ref{theorem3.1}}]
	
	Combining \eqref{approx-less-k} and \eqref{approx-k},  for $m\ge 6, k>1, \frac 32 -k<\ell<\frac 12 $ and $0<\epsilon\le 1$, we get
	\begin{align}
	\label{approx-total}
	\frac{d}{dt}\| \tilde{w}_\epsilon\|_{H^m_{k+\ell}(\mathbb{R}^2_+)}^2 &\le \frac{C}{\epsilon}\big( \| \tilde{w}_\epsilon\|_{H^m_{k+\ell}(\mathbb{R}^2_+)}^2 + \| \tilde{w}_\epsilon\|_{H^m_{k+\ell}(\mathbb{R}^2_+)}^m \big),
	\end{align}
	with $C>0$ independent of $\epsilon$.
	
From \eqref{approx-total}, by the nonlinear Gronwall's inequality, we have
	$$
	\| \tilde{w}_\epsilon(t)\|^{m-2}_{H^m_{k+\ell}(\mathbb{R}^2_+)} \le\, \frac{\|\tilde{w}_\epsilon(0)\|_{H^m_{k+\ell}}^{m-2}}
	{e^{-\frac{C}{\epsilon}t(\frac m2-1)}-(\frac m2-1)\frac{C}{\epsilon}t
\|\tilde{w}_\epsilon(0)\|_{H^m_{k+\ell}}^{m-2} },\,\,~~0< t \le T_\epsilon,
	$$
where we choose $T_\epsilon>0$ such that
\begin{align}
\label{time-1}	
\left(e^{-\frac{C}{\epsilon}T_\epsilon(\frac m2-1)}-(\frac m2-1)\frac{C}{\epsilon}T_\epsilon
\bar \zeta^{\,m-2} \right)^{-1}=\left(\frac 43\right)^{m-2}.
\end{align}
Finally, we get for any $\| \tilde{w}_\epsilon(0)\|_{H^m_{k+\ell}}\le \bar\zeta$, and $0<\epsilon\le \epsilon_0$,
	\begin{align*}\label{bound-3}
	\| \tilde{w}_\epsilon(t)\|_{H^m_{k+\ell}(\mathbb{R}^2_+)} \le \frac 43\| \tilde{w}_\epsilon(0)\|_{H^m_{k+\ell}(\mathbb{R}^2_+)}\le 2\| \tilde{w}_0\|_{H^m_{k+\ell}(\mathbb{R}^2_+)},~~~~~0< t \le T_\epsilon.
	\end{align*}
\end{proof}

The rest of this paper is dedicated to improve the results of Proposition \ref{prop3.2}, and try to get an uniform estimate with respect to $\epsilon$. Of course, we have to recall the assumption on the shear flow in the main Theorem \ref{main-theorem}.


\section{Formal transformations}\label{section4}

Since the estimate \eqref{approx-less-k} is independent of $\epsilon$, we only need to treat
\eqref{approx-k} in a new way to get an estimate which is also independent of $\epsilon$. To simplify the notations, from now on, we drop the notation tilde and sub-index $\epsilon$, that is, with no confusion, we take
$$
u=\tilde{u}_\epsilon,\quad v=\tilde{v}_\epsilon,\quad w=\tilde{w}_\epsilon.
$$

Let $w\in L^\infty ([0, T]; H^{m}_{k+\ell}(\mathbb{R}^2_+), m\ge 6, k>1, 0<\ell<\frac12,~~ \frac{1}{2}<\ell' < \ell + \frac{1}{2},~ k+\ell>\frac 32$ be a classical solution of \eqref{shear-prandtl-approxiamte-vorticity} which satisfies the following {\em \`a priori} condition
\begin{equation}\label{apriori}
\|w\|_{ L^\infty ([0, T]; H^{m}_{k+\ell}(\mathbb{R}^2_+))}\le  \zeta.
\end{equation}
Then \eqref{sobolev-1} gives
\begin{align*}
\|\langle y\rangle ^{k+\ell} w\|_{ L^\infty ([0, T]\times\mathbb{R}^2_+)}\le
&C(\|\langle y\rangle ^{\frac 12+\delta}(\langle y\rangle ^{k+\ell}w)_y \|_{ L^\infty ([0, T]; L^2(\mathbb{R}^2_+))}\\
&+
\|\langle y\rangle ^{\frac 12+\delta}(\langle y\rangle ^{k+\ell}w)_{xy} \|_{ L^\infty ([0, T]; L^2(\mathbb{R}^2_+))})\\
&\le C_m \|w\|_{ L^\infty ([0, T]; H^{m}_{k+\ell}(\mathbb{R}^2_+))},
\end{align*}
which implies
$$
|\partial_y u(t, x, y)|=|w(t, x, y)|\le C_m\, \zeta \, \langle y\rangle^{-k-\ell},\quad (t, x, y)\in [0, T]\times \mathbb{R}^2_+.
$$
We assume that {\bf  $\zeta$ is small enough} such that
\begin{equation}\label{C0}
C_m\, \zeta \le \frac{\tilde c_1}{4},
\end{equation}
where $C_m$ is the above Sobolev embedding constant.
Then we have  for $\ell\ge 0$,
\begin{align}
\label{pior-2}\frac{\tilde c_1}4 \langle y \rangle^{-k} \le |u^s_y + u_y|\le  4\tilde c_2 \langle y \rangle^{-k},\quad (t, x, y)\in [0, T]\times \mathbb{R} \times \mathbb{R}^+.
\end{align}

\noindent {\bf The formal transformation of equations.} Under the conditions \eqref{C0} and \eqref{pior-2}, in this subsection,  we will introduce the following formal transformations of system \eqref{shear-prandtl-approxiamte}.
{
 Set, for $0\le n\le m$
\[
g_n = \left( \frac{\partial^n_x u}{u^s_y + u_y} \right)_y,\,~\eta_1 = \frac{u_{xy}}{u^s_y + u_y},~ \eta_2 = \frac{u^s_{yy} + u_{yy}}{{u^s_y +u_y}}, \,\forall (t, x, y)\in [0, T]\times \mathbb{R}^2_+ .
\]
Formally, we will use the following notations
\[
\partial^{-1}_y g_n(t, x, y) = \frac{\partial^n_x u}{u^s_y + {u}_y}(t,x, y),\,
\partial_y \partial_y^{-1} g_n = g_n,~\forall (t, x, y)\in [0, T]\times \mathbb{R}^2_+
\]

 Applying  $\partial^n_x$ to  \eqref{shear-prandtl-approxiamte}, we have
\begin{align}\label{shear-prandtl-approxiamte-aa}
\begin{split}
\partial_t\partial^n_x u + (u^s + {u}) \partial_x\partial^n_x u &+(\partial^n_x {v}) (u^s_y + \partial_y {u})\\
&= \partial^2_{y}\partial^n_x u + \epsilon \partial^2_{x}\partial^n_x u+
A^1_n+A^2_n,
\end{split}
\end{align}
where
\begin{align*}
A^1_n=-[\partial^n_x,\, (u^s + {u})] \partial_x u=-
\sum_{i=1}^{n}C^i_n\partial_x^i {u} \, \partial^{n + 1 -i}_x {u}, \,\,\\
A^2_n=-[\partial^n_x,\, (u^s_y + \partial_y {u})] {v}= -\sum_{i=1}^{n}C^i_n\partial_x^i {w} \,\partial^{n  -i}_x {v}.
\end{align*}
Dividing \eqref{shear-prandtl-approxiamte-aa} with $(u^s_y + {u}_y)$ and performing $\partial_y$ on the resulting equation, observing
$$
 \partial_x\partial^n_x u +\partial_y\partial^n_x v= \partial^n_x(\partial_x u +\partial_y v)=0,
$$
we have for $j=1, 2$,
\begin{equation*}
\begin{split}
& \partial_y\left(\frac{\partial_t {\partial^n_x u}}{u^s_y + {u}_y}\right) +  (u^s + {u}) \partial_y\left(\frac{\partial_x{\partial^n_x u}}
{u^s_y + {u}_y}\right)\\
&\qquad\qquad=  \partial_y\left(\frac{\partial^2_{y}{\partial^n_x u} + \epsilon \partial^2_{x}{\partial^n_x u}}{u^s_y + {u}_y}\right)
+ \partial_y\left(\frac{A^1_n+A^2_n}{u^s_y + {u}_y}\right).
\end{split}
\end{equation*}
We compute each term on the support of $ $,
 \begin{align*}
 \partial_y\left(\frac{\partial_t {\partial^n_x u} }{u^s_y + {u}_y}\right) &= \partial_y\bigg(\partial_t \frac{{\partial^n_x u} }{u^s_y + {u}_y} + \partial_y^{-1} g_n \, \frac{\partial_t  {u}_y + \partial_t u^s_y }{u^s_y + {u}_y} \bigg)\\
 &= \partial_t g_n + \partial_y\bigg( \partial_y^{-1} g_n \, \frac{ \partial_t u^s_y+\partial_t  {u}_y}{u^s_y + \tilde{u}_y} \bigg),
 \end{align*}
 \begin{align*}
 		(u^s + {u})\partial_y \left(\frac{\partial_x {\partial^n_x u}}{u^s_y + {u}_y}\right) & = (u^s + {u})\bigg\{\partial_x\partial_y\left(\frac{{\partial^n_x u}}{u^s_y + {u}_y}\right) +\partial_y \left(\frac{{\partial^n_xu}}{u^s_y +{u}_y}\right) \, \frac{{u}_{xy}}{u^s_y + {u}_y} \\
 &\hskip 3cm + \left(\frac{\partial^n_x{u}}{u^s_y + {u}_y}\right)\partial_y\left( \frac{{u}_{xy}}{u^s_y + {u}_y}\right)\bigg\}\\
 		& =  (u^s + {u})(\partial_x g_n + g_n\, \eta_1 + \partial_y^{-1} g_n \, \partial_y \eta_1),
 \end{align*}
 \begin{align*}
 \frac{{\partial^2_{y} {\partial^n_x u}}}{u^s_y + {u}_y}  =\partial^2_y \left(\frac{{ {\partial^n_x u}}}{u^s_y + {u}_y} \right)  + 2\left(\frac{{ \partial_y {u}}}{u^s_y + {u}_y}\right)\frac{u^s_{yy} + {u}_{yy}}{{u^s_y + {u}_y}} - {\partial^n_x u} \, \partial^2_y\left(\frac{1}{u^s_y + {u}_y}\right),
 \end{align*}
 \begin{align*}
\partial^2_y \left(\frac{1}{u^s_y + {u}_y}\right) = - \partial_y\left(\frac{u_{yy}^s + {u}_{yy}}{(u^s_y + {u}_y)^2}\right)= -\frac{u_{yyy}^s + {u}_{yyy}}{(u^s_y + {u}_y)^2} + 2 \left(\frac{u_{yy}^s + {u}_{yy}}{(u^s_y + {u}_y)}\right)^2 \frac{1}{u^s_y + {u}_y},
 \end{align*}
 \begin{align*}
 \frac{{ \partial_y {\partial^n_x u}}}{u^s_y + {u}_y}\frac{u^s_{yy} +{u}_{yy}}{{u^s_y + {u}_y}} =  \left(\frac{{{\partial^n_xu}}}{u^s_y + {u}_y}\right)_y \frac{u^s_{yy} + {u}_{yy}}{{u^s_y + {u}_y}} - \frac{{{\partial^n_xu}}}{u^s_y + {u}_y}\left(\frac{u_{yy}^s + {u}_{yy}}{(u^s_y + {u}_y)}\right)^2.
\end{align*}
 So 		
\begin{align*}
\frac{{\partial^2_{y}{\partial^n_x u}}}{u^s_y + {u}_y} = \partial_y g_n  + 2(g_n \eta_2 - 2 \partial^{-1}_y g_n \eta_2^2) + \partial^{-1}_y g_n \left(\frac{u_{yyy}^s + {u}_{yyy}}{(u^s_y + \tilde{u}_y)}\right),
\end{align*}
\begin{align*}
\partial_y\left(\frac{{\partial^2_{y} {\partial^n_x u}}}{u^s_y + {u}_y}\right)&=\partial^2_y g_n  + 2 (\partial_y g_n) \eta_2 + 2g_n \partial_y \eta_2 -  4g_n\eta_2^2\\
 &- 8 \partial_y^{-1} g_n \eta_2 \partial_y \eta_2 + \partial_y\bigg(\partial_y^{-1} g_n\, \frac{u_{yyy}^s + {u}_{yyy}}{u^s_y + {u}_y}\bigg).
\end{align*}
Similarly, we have
\begin{align*}
\frac{{\partial^2_{x} {\partial^n_x u}}}{u^s_y + {u}_y} &=  \partial^2_x\left(\frac{{{\partial^n_x u}}}{u^s_y + {u}_y} \right)  +  2\left(\frac{{{\partial^n_x u}}}{u^s_y + {u}_y}\right)_x\frac{ {u}_{xy}}{{u^s_y + {u}_y}} \\
&-  2\frac{{{\partial^n_x u}}}{u^s_y +{u}_y}\left( \frac{{u}_{xy}}{(u^s_y +{u}_y)}\right)^2 + \frac{{{\partial^n_x u}}}{u^s_y + {u}_y}\frac{{u}_{xxy}}{(u^s_y +{u}_y)},
\end{align*} 			
\begin{align*}
\partial_y \left(\frac{{\partial^2_{x}{\partial^n_x u}}}{u^s_y + {u}_y}\right)& = \partial^2_xg_n + 2\partial_x  g_n \eta_1 + 2\partial_x \partial_y^{-1} g_n \partial_y \eta_1  \\
  &-  2g_n\eta_1^2 - {4 \partial_y^{-1} g_n \eta_1 \partial_y \eta_1} +\partial_y\bigg(\partial_y^{-1} g_n\,\frac{   {u}_{xxy}}{u^s_y +{u}_y}\bigg)\, .
\end{align*}
For the boundary condition, we only need to pay attention to $j=1$.  From \eqref{shear-prandtl-approxiamte-aa} and the boundary condition for $(u, v)$ in
\eqref{shear-prandtl-approxiamte}, we observe
$$
\partial^n_x u|_{y=0}=0,\,\,\partial^2_y \partial^n_x u|_{y=0}=0,\,\,   (u^s_y+u_y)|_{y=0}\not=0.
$$
At the same time,
\begin{align*}
0=\frac{{\partial^2_{y}{\partial^n_x u}}}{u^s_y + {u}_y}\bigg|_{y=0} &= \partial_y g_n|_{y=0} + 2(g_n \eta_2 - 2 (\partial^{-1}_y g_n )\eta_2^2)|_{y=0}\\ &
\qquad+ \partial^{-1}_y g_n \left(\frac{u_{yyy}^s + {u}_{yyy}}{(u^s_y + \tilde{u}_y)}\right)\bigg|_{y=0},
\end{align*}
and
$$
\eta_2|_{y=0} = \frac{u^s_{yy} + u_{yy}}{{u^s_y +u_y}}\bigg|_{y=0}=0, \quad
\partial^{-1}_y g_n(t, x, y)|_{y=0} = \frac{\partial^n_x u}{u^s_y + {u}_y}(t,x, y)\bigg|_{y=0}=0,\,
$$
we get then
$$
(\partial_y g_n)|_{y=0}=0,\quad 0\le n\le m.
$$
Finally, we have, for $j=1, 2$,
\begin{equation}
\begin{cases}
\label{non-monotone-transformation}
\partial_t   g_n + ( u^s + {u})  \partial_x g_n -  \partial^2_y g_n - \epsilon   \partial^2_xg_n\\
\qquad\qquad\qquad  - \epsilon \,2\,   (\partial_x \partial_y^{-1} g_n) \partial_y \eta_1 =   M_n,\\
  (\partial_y g_n)|_{y=0}=0, \\
  g_n|_{t=0}=  g_{n,0},
\end{cases}
\end{equation}
with $M_n=\sum^6_{j=1}M_j^n$,
\begin{align*}
&  M^n_1 = -(u^s + {u})( g_n \eta_1 +( \partial_y^{-1} g_n ) \partial_y \eta_1)  ,\\
&  M^n_2 =  2(\partial_y g_n) \eta_2 + 2g_n (\partial_y \eta_2 - 2\eta_2^2) - 8 (\partial_y^{-1} g_n)\, \eta_2 \partial_y \eta_2 ,\\
&  M^n_3 = \epsilon \big(2(\partial_x  g_n) \eta_1 - 2 g_n\,\eta_1^2 - 4{(\partial_y^{-1} g_n) \eta_1 \partial_y \eta_1}\big), \\
&  M^n_4 = \partial_y\bigg(\partial_y^{-1} g_n \, \frac{(u^s + {u}) {w}_x + {v} ({w}_y  + u^s_{yy})}{u^s_y + {u}_y}\bigg),\\
&  M^n_5  = -\partial_y\bigg(\frac{\sum_{i=1}^{n}C^i_n\partial_x^i {u} \cdot \partial^{n + 1 -i}_x {u} }{ u^s_y +{u}_y} \bigg)\,,\\
& M^n_6  = -\partial_y\bigg(\frac{ \sum_{i=1}^{n}C^i_n\partial_x^i {w} \cdot \partial^{n  -i}_x {v}}{ u^s_y +{u}_y} \bigg)\, ,
\end{align*}
where we have used the relation,
$$
\partial_t u^s_y+\partial_t  {u}_y -(u_{yyy}^s + {u}_{yyy})-\epsilon {u}_{xxy}=-(u^s + {u}) {w}_x + {v} (u^s_{yy}+{w}_y).
$$

\section{Uniform estimate}\label{section5}

In the future application(see Lemma \ref{lemma-g-h-w}), we need that the weight 
of $g_m$ big then $\frac 12$, but from the definition,  
$ w\in H^{m+2}_{k + \ell}(\mathbb{R}^2_+)$ imply only  
$g_m\in H^{2}_{\ell}(\mathbb{R}^2_+)$ with $0<\ell<\frac 12$. 
So the first step is to improve this weights if the weight of the initial data is more big. We first have
\begin{lemma}\label{lemma-initial-deta}
If $\tilde w_0\in H^{m+2}_{k + \ell'}(\mathbb{R}^2_+), m\ge 6, k>1, 0< \ell<\frac12,~~ \frac{1}{2}<\ell' < \ell + \frac{1}{2},~ k+\ell>\frac 32$ which satisfies \eqref{apriori}-\eqref{C0} with $0<\zeta\le1$, then $(  g_m)(0)\in H^2_{k+\ell}(\mathbb{R}^2_+)$, and we have
$$
\|(  g_m)(0)\|_{H^2_{\ell'}(\mathbb{R}^2_+)}\le C \|\tilde w_0\|_{H^{m+2}_{k + \ell''}(\mathbb{R}^2_+)}.
$$
\end{lemma}
\noindent
{\bf Remark.} 
In fact, observing
\begin{align*}
  g_m(0)=  \left( \frac{\partial_x^m \tilde{u}_0}{u^s_{0,y} + \tilde u_{0, y}} \right)_y=
  \frac{\partial_y\partial_x^m \tilde{u}_0}{u^s_{0,y} + \tilde u_{0, y}}-  \frac{\partial_x^m \tilde{u}_0}{u^s_{0,y} + \tilde u_{0, y}} \eta_2(0),
\end{align*}
then \eqref{pior-2} implies
\begin{align*}
\langle y \rangle ^{k+\ell}|  g_m(0)|\le
C\langle y \rangle ^{k + \ell'}|  \partial_x^m \tilde{w}_0|+C\langle y \rangle ^{k + \ell'-1}|  \partial_x^m \tilde{u}_0|,
\end{align*}
which finishes the proof of  this Lemma.

\begin{proposition}\label{lemma-non-x-k-monotone-part1-2}
Let $w \in L^\infty ([0, T]; H^{m+2}_{k+\ell}(\mathbb{R}^2_+)), m\ge 6, k>1, 0\le \ell<\frac12, ~\ell'> \frac{1}{2},~~\ell' - \ell < \frac{1}{2},~k+\ell>\frac 32$, satisfy  \eqref{apriori}-\eqref{C0} with $0<\zeta\le1$. Assume that the shear flow $u^s$ verifies the conclusion of Lemma \ref{shear-profile},
and $g_n$ satisfies the equation \eqref{non-monotone-transformation} for $ 1\le n\le m$, then we have the following estimates, for $t\in [0, T]$
\begin{equation}
\label{uniform-part2-2}
\begin{split}
\frac{d}{dt}\sum^m_{n=1}\|   g_n \|_{L^2_{\ell'}(\mathbb{R}^2_+)}^2 & + \sum^m_{n=1}\|   \partial_y g_n\|_{L^2_{\ell'}(\mathbb{R}^2_+)}^2 + \epsilon \sum^m_{n=1}\|  \partial_x g_n\|_{L^2_{\ell'}(\mathbb{R}^2_+)}^2\\
& \le C_2(\sum^m_{n=1}\|  g_n\|_{L^2_{\ell'}(\mathbb{R}^2_+)}^2 + \|{w}\|_{H^m_{k+\ell}(\mathbb{R}^2_+)}^2),
\end{split}
\end{equation}
where $C_2$ is independent of $\epsilon$.
\end{proposition}

\noindent
{\bf\em Approach of the proof for the Proposition \ref{lemma-non-x-k-monotone-part1-2}:} We can't prove \eqref{uniform-part2-2} directly, since
the approximate solution $w_\epsilon$ obtained in Theorem \ref{theorem3.1} is belongs to
$L^\infty([0, T_\epsilon]; H^{m+2}_{k+\ell}(\mathbb{R}^2_+))$, which implies  only  $  g_n\in L^\infty([0, T_\epsilon]; H^{2}_{\ell}(\mathbb{R}^2_+))$. Then we can't use $\langle y\rangle^{2\ell'}  g_n \in L^\infty([0, T_\epsilon]; H^{2}_{\ell-2\ell'}(\mathbb{R}^2_+)) $ as the test function to the equation \eqref{non-monotone-transformation}. To overcome this difficulty, we consider that \eqref{non-monotone-transformation} as a linear system for $ g_n, n=1, \cdots, m$ with the coefficients
and the source terms depends on $w$ and their derivatives up to order $m$, we will clarify this confirmation in the following proof of the the Proposition \ref{lemma-non-x-k-monotone-part1-2}. We prove now the estimate
\eqref{uniform-part2-2} by the following approach: For the linear system  \eqref{non-monotone-transformation}, we prove firstly \eqref{uniform-part2-2} as {\em \`a priori} estimate. Lemma \ref{lemma-initial-deta} imply that $ {g}_{n}(0)\in H^2_{\ell'}(\mathbb{R}^2_+), n=1,\cdots, m$, then by using Hahn-Banach theorem, this {\em \`a priori} estimate imply the existence of solutions
$$
  g_n\in L^\infty([0, T]; H^2_{\ell'}(\mathbb{R}^2_+)),\quad n=1, \cdots, m.
$$
Finally, by uniqueness,  we can prove the estimate \eqref{uniform-part2-2} by proving  it as  {\em \`a priori} estimate. So that the proof of the Proposition \ref{lemma-non-x-k-monotone-part1-2} is reduced to the proof of the {\em \`a priori} estimate \eqref{uniform-part2-2}.

\begin{proof}[{\bf Proof of the {\em \`a priori} estimate \eqref{uniform-part2-2}}]
Multiplying the linear system \eqref{non-monotone-transformation} by $\langle y\rangle^{2\ell'}g_{n}\in L^\infty([0, T]; H^2_{-\ell'}(\mathbb{R}^2_+)) $ and integrating over $\mathbb{R} \times \mathbb{R}^+$. We start to  deal with the left hand of \eqref{non-monotone-transformation} first, we have
\begin{align*}
\int_{\mathbb{R}^2_+} \partial_t g_{n} \, \langle y\rangle^{2\ell'} g_{n} dx dy = \frac 12\frac{d}{dt}\|g_{n}\|_{L^2_{\ell'}(\mathbb{R}^2_+)}^2,
\end{align*}
and
\begin{align*}
\int_{\mathbb{R}^2_+}   ( u^s + {u})\partial_x g_{n} \, \langle y\rangle^{2\ell'} g_{n}dx dy
=& \frac{1}{2}\int_{\mathbb{R}^2_+}   ( u^s + {u}) \cdot \partial_x (\langle y\rangle^{2\ell'} g_{n}^2) dx dy \\
&\le \frac 12 \|{u}_x\|_{L^\infty(\mathbb{R}^2_+)}\|g_{n}\|_{L^2_{\ell'}(\mathbb{R}^2_+)}^2\\
&\le C\|w\|_{H^2_{1}(\mathbb{R}^2_+)}\|g_{n}\|_{L^2_{\ell'}(\mathbb{R}^2_+)}^2.
\end{align*}
Integrating by part, where the boundary value is vanish,
\begin{align*}
- \int_{\mathbb{R}^2_+}    \partial_y^2 g_{n} \, \langle y\rangle^{2\ell'} g_{n}dx dy 
&=  \|\partial_y g_{n} \|_{L^2_{\ell'}(\mathbb{R}^2_+)}^2 +  \int_{\mathbb{R}^2_+}   \partial_y g_{n} (\langle y\rangle^{2\ell'})' g_{n} dx dy\\
& \ge \frac{3}{4}\|  \partial_y g_{n} \|_{L^2_{\ell'}(\mathbb{R}^2_+)}^2 - 4  \| g_{n}\|_{L^2(\mathbb{R}^2_+)}^2,
\end{align*}
and
\begin{align*}
-\epsilon\int_{\mathbb{R}^2_+}  \partial_x^2 g_{n} \, \langle y\rangle^{2\ell'}  g_{n} dx dy =\epsilon\|  \partial_x g_{n}\|_{L^2_{\ell'}(\mathbb{R}^2_+)}^2.
\end{align*}
We have also
\begin{align*}
&-\epsilon \int_{\mathbb{R}^2_+}  \big(\partial_x \partial_y^{-1} g_{n} \big)\partial_y \eta_1 \langle y\rangle^{2\ell'}g_{n} dx dy \\
& = \epsilon \int_{\mathbb{R}^2_+} \partial_y^{-1} g_{n} \partial_y \eta_1 \langle y\rangle^{2\ell'}   \partial_x g_{n} dx dy\\
 &\quad+ \epsilon \int_{\mathbb{R}^2_+}   \partial_y^{-1} g_{n} (\partial_y \partial_x \eta_1)\langle y\rangle^{2\ell'} g_{n}dx dy\\
& \le \epsilon\| \partial_y^{-1} g_{n} \partial_y \eta_1 \|_{L^2_{\ell'}(\mathbb{R}^2_+)}^2 + \frac{\epsilon}{8}\|\partial_x g_{n}\|_{L^2_{\ell'}(\mathbb{R}^2_+)}^2\\
& \quad+ \epsilon\| \partial_y^{-1} g_{n}\partial_y \partial_x \eta_1\|_{L^2(\mathbb{R}^2_+)}^2 + \epsilon\|g_{n}\|_{L^2_{\ell'}(\mathbb{R}^2_+)}^2.
\end{align*}
So by \eqref{non-monotone-transformation} and $0<\epsilon\le 1$, we obtain
\begin{equation*}
\begin{split}
&\frac{d}{dt}\|g_{n}\|^2_{L^2_{\ell'}(\mathbb{R}^2_+)}+\| \partial_yg_{n}\|^2_{L^2_{\ell'}(\mathbb{R}^2_+)}
+\epsilon \|\partial_x g_{n}\|^2_{L^2_{\ell'}(\mathbb{R}^2_+)}\\
&\le C \|g_{n}\|^2_{L^2_{\ell'}(\mathbb{R}^2_+)}+\|(\partial_y^{-1} g_{n}) \partial_y \eta_1 \|_{L^2_{\ell'}(\mathbb{R}^2_+)}^2 \\
&\qquad+ \| (\partial_y^{-1} g_{n}) \partial_y \partial_x \eta_1\|_{L^2(\mathbb{R}^2_+)}^2 + 2\sum_{j=1}^{6}\bigg|\int_{\mathbb{R}^2_+}   {M}_{j} \, \langle y\rangle^{2\ell'}g_{n} dx dy\bigg|.
\end{split}
\end{equation*}
Then we can finish the proof of the {\em \`a priori} estimate \eqref{uniform-part2-2} by the following four Lemmas.
\end{proof}

\begin{lemma}\label{lemma5.1}
Under the assumption of Proposition \ref{lemma-non-x-k-monotone-part1-2}, we have
\begin{equation*}
\begin{split}
\| \partial_y^{-1} g_{n} \partial_y \eta_1 \|_{L^2_{\ell'}(\mathbb{R}^2_+)}^2 + \|  \partial_y^{-1} g_{n} \partial_y \partial_x \eta_1\|_{L^2(\mathbb{R}^2_+)}^2 \le C \|g_{n}\|_{L^2_{\ell'}(\mathbb{R}^2_+)}^2.
\end{split}
\end{equation*}
where $\tilde{C}$ is independent of $\epsilon$.
\end{lemma}

\begin{proof}
Notice that \eqref{apriori} and \eqref{C0} imply
\begin{align*}
&|\eta_1| \le C \langle y \rangle^{-\ell},\quad |\partial_x \eta_1 |\le C\langle y \rangle^{-\ell},\\
&|\partial_y \eta_1 |\le C \langle y \rangle^{-\ell - 1},\quad
|\partial_y \partial_x \eta_1| \le C \langle y \rangle^{-\ell - 1}.
\end{align*}
Then $\ell'>\frac 12, \ell'-\ell<\frac12$, imply
\begin{align*}
\begin{split}
\|\partial_y^{-1} g_{n} (\partial_y \partial_x \eta_1)\|_{L^2_{\ell'}(\mathbb{R}^2_+)}^2 & \le C\int_{\mathbb{R}^2_+}\langle y \rangle^{2(\ell' - \ell-1)}
\Big(\int^y_0 g_n (t, x, \tilde y)d\tilde y \Big)^2 dx dy\\
& \le C\| g_{n} \|_{L^2_{\ell'}(\mathbb{R}^2_+)}^2.
\end{split}
\end{align*}
Similarly, we also obtain
\begin{align*}
\begin{split}
\| \partial_y^{-1} g_{n} \partial_y \eta_1\|_{L^2_{\ell'}(\mathbb{R}^2_+)}^2 \le C\| g_{n} \|_{L^2_{ \ell'}(\mathbb{R}^2_+)}^2.
\end{split}
\end{align*}
\end{proof}
\begin{lemma}\label{lemma5.2}
Under the assumption of Proposition \ref{lemma-non-x-k-monotone-part1-2}, we have
\begin{equation*}
\begin{split}
\left|\int_{\mathbb{R}^2_+}  \sum^4_{j = 0}\tilde{M}^n_{j } \, \langle y\rangle^{2\ell'}g_{n} dx dy\right|\le &
\frac18\|\partial_y g_{n}\|^2_{L^2_{\ell'}(\mathbb{R}^2_+)}+
\frac{\epsilon}8\| \partial_x g_{n}\|^2_{L^2_{\ell'}(\mathbb{R}^2_+)} \\
&\quad+ \tilde{C}(\|g_{n}\|_{L^2_{\ell'}(\mathbb{R}^2_+)}^2
+\|w\|_{H^m_{k+\ell}(\mathbb{R}^2_+)}^2),
\end{split}
\end{equation*}
where $\tilde{C}$ is independent of $\epsilon$.
\end{lemma}
\begin{proof}
 Recalling $  {M}^n_{1 } = -(u^s + {u})\big( g_{n} \eta_1 +( \partial_y^{-1} g_{n} ) \partial_y \eta_1  \big) $, by Lemma \ref{lemma5.1},
\begin{align*}
\left|\int_{\mathbb{R}^2_+} (u^s + {u})g_{n} \eta_1 \, \langle y\rangle^{2\ell'}g_{n}  dxdy\right| &\le C\|g_{n}\|_{L^2_{\ell'}(\mathbb{R}^2_+)}^2,\\
\int_{\mathbb{R}^2_+} |(u^s + {u})  ( \partial_y^{-1} g_n) \partial_y \eta_1\,  \langle y \rangle^{2\ell'} g_{n} | dy dx &\le C \| w\|_{H^n_{k+ \ell}}^2 + C \| g_{n}\|_{L^2_{\ell'}}^2.
\end{align*}
Besides, we have
\begin{align*}
\left|\int_{\mathbb{R}^2_+}  {M}^n_{1 } \,  \langle y\rangle^{2\ell'}g_{n} dx dy\right|\le {C}(\|g_{n}\|_{L^2_{\ell'}(\mathbb{R}^2_+)}^2
+\|w\|_{H^n_{\ell'}(\mathbb{R}^2_+)}^2).
\end{align*}

The estimates of $ {M}^n_2$ and $ {M}^n_3$ needs the following decay rate of $\eta_2$:
\begin{align*}
&|\eta_2| \le C \langle y \rangle^{-1},\quad |\partial_x \eta_2|
\le C \langle y \rangle^{-\ell-1},\\
&|\partial_y \eta_2 |\le C \langle y \rangle^{-2},\quad |\partial_y \partial_x \eta_2|
\le C \langle y \rangle^{-\ell - 2}.
\end{align*}
Recall $  {M}_{2 }^n=  2 \partial_y g_{n}  \eta_2 + 2g_{n} (\partial_y \eta_2 - 2\eta_2^2) - 8  \partial_y^{-1} g_{n}\, \eta_2 \partial_y \eta_2$. We have
\begin{align*}
& \left|\int_{\mathbb{R}^2_+}  g_{n} (\partial_y \eta_2 - \eta_2^2) \, \langle y\rangle^{2\ell'} g_{n}dx dy\right| \le C\|g_{n}\|_{L^2_{\ell'}(\mathbb{R}^2_+)}^2,\\
&\left|\int_{\mathbb{R}^2_+} (\partial_y g_{n}) \eta_2 \, \langle y\rangle^{2\ell'}g_{n} dx dy\right| \le C\|g_{n}\|_{L^2_{\ell'}(\mathbb{R}^2_+)}^2 + \frac{1}{8}\| \partial_y g_{n}\|_{L^2_{\ell'}(\mathbb{R}^2_+)}^2,\\
 &\left|2 \int_{\mathbb{R}^2_+}  \partial_y^{-1} g_{n} \eta_2 \, \partial_y \eta_2 \langle y\rangle^{2\ell'}g_{n} dx dy\right| \\
 &\qquad\qquad \le C\| \langle y \rangle^{\ell'-3}\partial_y^{-1} g_n \|_{L^2}^2 + \| g_{n}\|_{L^2_{\ell'}(\mathbb{R}^2_+)}^2\le C\|g_{n}\|_{L^2_{\ell'}(\mathbb{R}^2_+)}^2.
\end{align*}
All together, we conclude
\begin{align*}
\left|\int_{\mathbb{R}^2_+}   {M}^n_{2 } \langle y\rangle^{2\ell'}g_{n} dx dy\right| &\le C(\|g_{n}\|_{L^2_{\ell'}(\mathbb{R}^2_+)}^2 +\|w\|_{H^n_{\ell'}(\mathbb{R}^2_+)}^2)+ \frac{1}{8}\| \partial_y  g_n\|_{L^2_{\ell'}(\mathbb{R}^2_+)}^2,
\end{align*}
and exactly same computation gives also
\begin{align*}
\left|\int_{\mathbb{R}^2_+}  {M}^n_{3 } \langle y\rangle^{2\ell'}g_{n} dx dy\right| &\le C(\|g_{n}\|_{L^2_{\ell'}(\mathbb{R}^2_+)}^2 +\|w\|_{H^n_{k+\ell}(\mathbb{R}^2_+)}^2)+ \frac{\epsilon}{8}\|\partial_x  g_n\|_{L^2_{\ell'}(\mathbb{R}^2_+)}^2.
\end{align*}
Now using  \eqref{apriori}-\eqref{C0} and $m\ge 6$, with the same computation as above, we can get
$$
\left|\int_{\mathbb{R}^2_+}  {M}^n_{4 } \langle y\rangle^{2\ell'}g_{n} dx dy\right| \le C\left(\|g_{n}\|_{L^2_{\ell'}(\mathbb{R}^2_+)}^2 +\|w\|_{H^n_{k+\ell}(\mathbb{R}^2_+)}^2\right).
$$
which finishes the proof of Lemma \ref{lemma5.2}.
\end{proof}

\begin{lemma}\label{lemma5.3}
Under the assumption of Proposition \ref{lemma-non-x-k-monotone-part1-2}, we have
\begin{equation*}
\begin{split}
&\left|\int_{\mathbb{R}^2_+} {M}^n_5 \, \langle y\rangle^{2\ell'}g_{n} dx dy\right|\le
\tilde{C}\left(\sum^n_{p=1}\|\tilde g_p\|_{L^2_{\ell'}(\mathbb{R}^2_+)}^2
+\|w\|_{H^n_{k + \ell}(\mathbb{R}^2_+)}^2\right),
\end{split}
\end{equation*}
where $\tilde{C}$ is independent of $\epsilon$.
\end{lemma}
\noindent

\begin{proof} Recall,
\begin{align*}
\tilde{M}_{5 }^n & =\sum_{i\ge 4}C^i_ng_{i}\, \partial^{n + 1 -i}_x u +
\sum_{1\le i\le 3}C^i_n\partial_x^i u\, g_{n+1-i }\\
&\quad+\sum_{i\ge 4}C^i_n   \partial_y^{-1} g_{n}  \partial^{n + 1 -i}_x w  + \sum_{1\le i\le 3}C^i_n\partial_x^i w\,  \partial_y^{-1} g_{n+ 1 - i } ,
\end{align*}
here if $n\le 3$, we have only the last term.
Then, for { $\|w\|_{H^{m}_{k+\ell}}\le \zeta\le 1, m\ge 6$},
\begin{align*}
&\sum_{i\ge 4}C^i_n\|g_{i}\, \partial^{n + 1 -i}_x u\|_{L^2_{\ell'}(\mathbb{R}^2_+)} +
\sum_{1\le i\le 3}\|\partial_x^i u\, g_{n+1-i }\|_
{L^2_{\ell'}(\mathbb{R}^2_+)}\\
&\le  \sum_{i\ge 4}C^i_n\|g_{i}\|_{L^2_{\ell'}(\mathbb{R}^2_+)}\|\partial^{n + 1 -i}_x u\|_{L^\infty(\mathbb{R}^2_+)} \\
&\qquad+
\sum_{1\le i\le 3}C^i_n\|\partial_x^i u\|_{L^\infty(\mathbb{R}^2_+)}\, \|\tilde g_{n+1-i }\|_
{L^2_{\ell'}(\mathbb{R}^2_+)}\\
&\le C \sum_{i\ge 4}C^i_n\|g_{i}\|_{L^2_{\ell'}(\mathbb{R}^2_+)}\|w\|_{H^{n + 3 -i}_1} \\
&\qquad+ C
\sum_{1\le i\le 3}C^i_n\|w\|_{H^{i + 3}_1}\, \|g_{n+1-i }\|_
{L^2_{\ell'}(\mathbb{R}^2_+)}\\
& \le C \sum_{i=1}^{n} \| g_{i}\|_{L^2_{\ell'}}.
\end{align*}
Similarly, for the second line in ${M}_{5}$, by Lemma \ref{lemma5.1}, we have
\begin{align*}
\sum_{i\ge 4}C^i_n\| (\partial^{-1}_y g_{i}) \partial^{n + 1 -i}_x w\|_{L^2_{\ell'}(\mathbb{R}^2_+)} &\le \sum_{i\ge 4}C^i_n\| \langle y \rangle^{\ell' - \ell - 1} (\partial^{-1}_y g_{i}) \|_{L^2(\mathbb{R}^2_+)} \| \langle y \rangle^{\ell + 1}\partial^{n + 1 -i}_x w\|_{L^\infty}\\
& \le C \sum_{i = 1}^{n} \|g_{i}\|_{L^2_{\ell'}(\mathbb{R}^2_+)}.
 \end{align*}
We have  proven Lemma \ref{lemma5.3}.
\end{proof}

\begin{lemma}\label{lemma5.4}
Under the assumption of Proposition \ref{lemma-non-x-k-monotone-part1-2}, we have
\begin{equation*}
\begin{split}
&\left|\int_{\mathbb{R}^2_+} {M}^n_6 \, \langle y\rangle^{2\ell'}g_{n} dx dy\right|\\
&\quad\le \frac 1{8m} \sum^n_{p=1}\|\partial_y  g_p\|^2_{L^2_{\ell'}(\mathbb{R}^2_+)}+
\tilde{C}\left(\sum^n_{p=1}\| g_p\|_{L^2_{\ell'}(\mathbb{R}^2_+)}^2
+\|w\|_{H^n_{k+\ell}(\mathbb{R}^2_+)}^2\right),
\end{split}
\end{equation*}
where $\tilde{C}$ is independents of $\epsilon$.
\end{lemma}

\begin{proof} Recall
\begin{align*}
{M}_{6 }  & =  \sum_{i=1}^{n}C^i_n  g_{i}\eta_2   \partial^{n -i}_x {v}+
\sum_{i=1}^{n}C^i_n g_{i}  \partial^{n+1  -i}_x {u}  + \sum_{i=1}^{n}C^i_n  \partial_y g_{i}   \partial^{n -i}_x {v}\\
& + \sum_{i=1}^{n}   \partial_y^{-1} g_{i}  \left( C_n^i\partial^{n -i}_x {v} \partial_y \eta_2 + C_n^i \partial^{n+1  -i}_x {u} \eta_2\right).
\end{align*}
{In ${M}_6^n$, we just study the term $\partial_y g_{1} \partial_x^{n - 1} v $ as an example, the others terms are similar,
\begin{align*}
\int_{\mathbb{R}^2_+}  \partial_y g_{1 } \partial_x^{n - 1} v \,  \langle y \rangle^{2\ell'} g_{n} & = - \int_{\mathbb{R}^2_+}   g_{1 } \partial_x^{n - 1} v \,  \langle y \rangle^{2\ell'}  \partial_y g_{n}\\
& + \int_{\mathbb{R}^2_+}    g_{1 }  \partial_x^{n  } u\,  \langle y \rangle^{2\ell'}  g_{n} dx dy,
\end{align*}
\begin{align*}
\int_{\mathbb{R}^2_+}   g_{1 } \partial_x^{n - 1} v \,  \langle y \rangle^{2\ell'}  \partial_y g_{n} dx dy & \le \frac{1}{8m} \|\partial_y g_{n}\|_{L^2_{\ell'}}^2 + C\| g_{1 } \partial_x^{n - 1} v \|_{L^2_{\ell'}}^2,
\end{align*}
\begin{align*}
 \| g_{1 } \partial_x^{n - 1} v \|_{L^2_{\ell'}}^2 & \le \sup_{x \in \mathbb{R}} \int_{0}^{+\infty} \langle y \rangle^{2\ell'}g_{1 }^2 dy \, \sup_{y \in \mathbb{R}_+} \int_{-\infty}^{+\infty} \left|\int_{0}^{y}\partial_x^{n} u dz\right|^2 dy\\
 & \le \bigg(  \|g_{1 }\|_{L^2_{\ell'}(\mathbb{R}_+^2)}^2 +  \|\partial_x g_{1 }\|_{L^2_{\ell'}(\mathbb{R}_+^2)}^2  \bigg) \int_{-\infty}^{+\infty} \left|\int_{0}^{+\infty}|\partial_x^{n} u| dz\right|^2 dy\\
 & \le C \bigg(  \|g_{1 }\|_{L^2_{\ell'}(\mathbb{R}_+^2)}^2 +  \|\partial_x g_{1 }\|_{L^2_{\ell'}(\mathbb{R}_+^2)}^2  \bigg)\\
 &\qquad \times \int_{-\infty}^{+\infty} \left|\int_{0}^{+\infty}\langle y \rangle^{- k - \ell + 1} \, \langle y \rangle^{k + \ell - 1}|\partial_x^{n} u| dz\right|^2 dy\\
 & \le C \bigg(  \|g_{1 }\|_{L^2_{\ell'}(\mathbb{R}_+^2)}^2 +  \|  g_{2 }\|_{L^2_{\ell'}(\mathbb{R}_+^2)}^2 + \|w\|_{H^m_{k +\ell}}^2\bigg)\\
 & \qquad\times \int_{-\infty}^{+\infty} \left|\int_{0}^{+\infty}\langle y \rangle^{- k - \ell + 1} \, \langle y \rangle^{k + \ell - 1}|\partial_x^{n} u| dz\right|^2 dy\\
 & \le C\sum_{i=1}^{2}\|g_{i}\|_{L^2_{\ell'}}^2 + C \|w\|_{H^m_{k +\ell}}^2.
\end{align*}
Here we have used Lemma \ref{lemma5.1} and
\[ k + \ell - 1 > \frac{1}{2},~~ \| w\|_{H^m_{k + \ell}} \le 1, \]
and
\[ \partial_x g_j =  g_{j+1} - g_j \eta_1 -  \partial_y^{-1} g_{n}  \cdot \partial_y \eta_1.  \]
}
By the similar trick, we have completed the proof of this lemma.
\end{proof}

\section{Existence of the solution}\label{section7}

Now, we can  conclude   the following energy estimate for the sequence of
approximate solutions.

\begin{theorem}\label{energy}
Assume $u^s$ satisfies Lemma \ref{shear-profile}. Let $m\ge 6$ be an even integer, $k + \ell >\frac{3}{2}, 0 < \ell<\frac12, ~\ \frac{1}{2}<\ell' < \ell + \frac{1}{2},~$, and $\tilde{u}_{0}\in  H^{m+3}_{k + \ell'-1}(\mathbb{R}^2_+)$ which  satisfies the compatibility conditions \eqref{compatibility-a1}-\eqref{compatibility-a2}. Suppose that
$\tilde{w}_\epsilon \in L^\infty ([0, T]; H^{m+2}_{k+\ell}(\mathbb{R}^2_+))$ is a solution to \eqref{shear-prandtl-approxiamte-vorticity} such that
 \begin{equation*}
\|\tilde{w}_\epsilon\|_{ L^\infty ([0, T]; H^m_{k+\ell}(\mathbb{R}^2_+)}\le  \zeta
\end{equation*}
with
$$
0<\zeta\le1, \quad C_m\zeta\le \frac{\tilde c_1}{2},
$$
where $0<T\le T_1$ and $T_1$ is the lifespan of shear flow $u^s$ in the Lemma \ref{shear-profile}, $C_m$ is the Sobolev embedding constant in \eqref{C0}. Then there exists
$C_T>0, \tilde C_T>0$  such that,
\begin{equation}\label{energy estimate-A}
\|\tilde{w}_\epsilon\|_{ L^\infty ([0, T]; H^m_{k+\ell}(\mathbb{R}^2_+))}\le  C_T\|\tilde{u}_{0}\|_{H^{m+1}_{k + \ell'-1}(\mathbb{R}^2_+)},
\end{equation}
where $C_T>0$ is increasing with respect to $0<T\le T_1$ and  independent of $0<\epsilon\le 1$.
\end{theorem}
Firstly,  we collect some results to be used from Section \ref{section3} - \ref{section5}. We come back to the notations with tilde and the sub-index $\epsilon$. Then $g^\epsilon_m, h^\epsilon_m$ are the the functions defined by $\tilde{u}_\epsilon$.
Under the hypothesis of Theorem \ref{energy}, we have proven the estimates \eqref{approx-less-k} and \eqref{uniform-part2-2}
\begin{equation}
\label{approx-less-k-b}
\begin{split}
\frac{d}{ dt}\| \tilde{w}_\epsilon\|_{H^{m, m-1}_{k+\ell}(\mathbb{R}^2_+)}^2 &+ \|\partial_y\tilde{w}_\epsilon\|_{H^{m, m-1}_{k+\ell}(\mathbb{R}^2_+)}^2\\
&+ \epsilon\|\partial_x\tilde{w}_\epsilon\|_{H^{m, m- 1}_{k+\ell}(\mathbb{R}^2_+)}^2
\le C_1 \| \tilde{w}_\epsilon\|_{H^m_{k+\ell}(\mathbb{R}^2_+)}^2,
\end{split}
\end{equation}

\begin{equation}
\label{uniform-part2-1b}
\begin{split}
\frac{d}{dt}\sum^m_{n=1}\|   g^\epsilon_n\|_{L^2_{\ell'}(\mathbb{R}^2_+)}^2 & + \sum^m_{n=1}\|   \partial_y g^\epsilon_n\|_{L^2_{\ell'}(\mathbb{R}^2_+)}^2 + \epsilon \sum^m_{n=1}\|  \partial_x g^\epsilon_n\|_{L^2_{\ell'}(\mathbb{R}^2_+)}^2\\
& \le C_2(\sum^m_{n=1}\|  g^\epsilon_n\|_{L^2_{\ell'}(\mathbb{R}^2_+)}^2 + \|\tilde{w}_\epsilon\|_{H^m_{k+\ell}(\mathbb{R}^2_+)}^2)\,,
\end{split}
\end{equation}

\begin{lemma}\label{lemma-initial}
For the inital date, we have
\begin{align*}
T^\epsilon_m(g, w)(0)&= \sum^m_{n=1}\|   g^\epsilon_n(0)\|_{L^2_{\ell'}(\mathbb{R}^2_+)}^2+\| \tilde{w}_\epsilon(0)\|_{H^{m, m -1}_{k+\ell}(\mathbb{R}^2_+)}^2\\
&\le C\| \tilde{u}_0\|_{H^{m+1}_{k + \ell'-1}(\mathbb{R}^2_+)}^2,
\end{align*}
where $C$ is independent of $\epsilon$.
\end{lemma}
\begin{proof}
Notice for any $1\le n\le m$,
\[
  g^\epsilon_n
=  \big(\frac{\partial_x^n \tilde{u}_\epsilon}{u^s_y + \tilde{w}_\epsilon}\big)_y
=   \frac{\partial_x^n\partial_y \tilde{u}_\epsilon}{u^s_y + \tilde{w}_\epsilon}
 -   \frac{\partial_x^n \tilde{u}_\epsilon}{u^s_y + \tilde{w}_\epsilon} \eta_2,
\]
and $\tilde{u}_\epsilon(0)=\tilde{u}_0$, then we deduce, for any $1\le n\le m$,
\begin{align*}
&\|  g^\epsilon_n(0)\|_{L^2_{\ell'}(\mathbb{R}^2_+)}^2\le 2\left\|   \frac{\partial_x^n\partial_y \tilde{u}_0}{u^s_{0, y} + \tilde{w}_0}\right\|_{L^2_{\ell'}(\mathbb{R}^2_+)}^2+2\left\|   \frac{\partial_x^n\tilde{u}_0}{u^s_{0, y} + \tilde{w}_0}\eta_2(0)\right\|_{L^2_{\ell'}(\mathbb{R}^2_+)}^2\\
&\le C\big(\| \partial_x^n\partial_y \tilde{u}_0\|_{L^2_{k+\ell'}(\mathbb{R}^2_+)}^2+\|\partial_x^n\tilde{u}_0\|_{L^2_{k+\ell'-1}(\mathbb{R}^2_+)}^2\big)
\le C\|\tilde{u}_0\|_{H^{m+1}_{k+\ell'-1}(\mathbb{R}^2_+)}^2.
\end{align*}
\end{proof}
From \eqref{approx-less-k-b} and \eqref{uniform-part2-1b}, we have
 \begin{align}
\label{uniform-full-1b}
\begin{split}
& \|   g^\epsilon_m\|_{L^2_{\ell'}(\mathbb{R}^2_+)}^2    + \| \tilde{w}_\epsilon\|_{H^{m, m-1}_{k+\ell}(\mathbb{R}^2_+)}^2 \\
& \le  C_8 e^{C_2 t}\int^t_0e^{-C_2 \tau}\| \tilde{w}_\epsilon(\tau)\|_{H^m_{k+\ell}(\mathbb{R}^2_+)}^2 d\tau+C_9e^{C_2 t}\| \tilde{u}_0\|_{H^{m+1}_{k + \ell'-1}(\mathbb{R}^2_+)}^2 .
\end{split}
\end{align}
\begin{lemma}\label{lemma-g-h-w}
We have also the following estimate :
$$
\|\partial^m_x\tilde{w}_\epsilon\|_{L^2_{k+\ell}(\mathbb{R}^2_+)}^2\le
\tilde C \|{g}_m^\epsilon\|_{L^2_{\ell'}}^2.
$$
where $\tilde C$ is independent of $\epsilon$.
\end{lemma}
\begin{proof} By the definition,
\[
\partial_x^m \tilde{u}_\epsilon(t, x, y)
= ( u^s_y+ \tilde{w}_\epsilon)\int_{0}^{y} g^\epsilon_m(t, x, \tilde y) d\tilde y,\quad y\in  \mathbb{R}_+,
\]
Therefore,
\[ 
\partial_x^m \tilde{w} = ( u^s_{yy}+(\tilde{w}_\epsilon)_y)\int_{0}^{y} g^\epsilon_m(t, x, \tilde y) d\tilde y -
( u^s_y+ \tilde{w}_\epsilon) g^\epsilon_m(t, x, y)  ,\,\,~~ y \ge 0 ,
\]
and
\begin{align*}
\| \partial_x^m \tilde{w}\|_{L^2_{k + \ell}}^2 & \le   C \int_{\mathbb{R}_+^2}\langle y \rangle^{2\ell - 2} \bigg( \int_{0}^{y} g_m^\epsilon(t,x, z) dz \bigg)^2 dx dy + \|g_m^\epsilon(t)\|_{L^2_{\ell'}(\mathbb{R}_+^2)}^2\\
& \le C \|g_m^\epsilon(t)\|_{L^2_{\ell'}(\mathbb{R}_+^2)}^2,
\end{align*}
where we have used $\ell - 1 < - \frac{1}{2}$ and $\frac12 < \ell'$.

\end{proof}

\begin{proof}[{\bf End of proof of Theorem \ref{energy}}]

Combining \eqref{uniform-full-1b}, Lemma \ref{lemma-initial} and Lemma \ref{lemma-g-h-w}, we get, for any $t\in ]0, T]$,
\begin{align*}
\begin{split}
\|\tilde{w}_\epsilon(t)\|_{H^{m}_{k+\ell}(\mathbb{R}^2_+)}^2\le & \tilde C_8 e^{C_2 t}\int^t_0e^{-C_2 \tau}\| \tilde{w}_\epsilon(\tau)\|_{H^m_{k+\ell}(\mathbb{R}^2_+)}^2 d\tau\\
&+\tilde C_9e^{C_2 t}\| \tilde{u}_0\|_{H^{m+1}_{k + \ell'-1}(\mathbb{R}^2_+)}^2,
\end{split}
\end{align*}
with $\tilde C_8, \tilde C_9$ independent of $0<\epsilon\le 1$.
We have by Gronwell's inequality that, for any $t\in ]0, T]$,
$$
\|\tilde{w}_\epsilon(t)\|_{H^{m}_{k+\ell}(\mathbb{R}^2_+)}^2\le \tilde C_9e^{(C_2+\tilde C_8) t}\| \tilde{u}_0\|_{H^{m+1}_{k + \ell'-1}(\mathbb{R}^2_+)}^2.
$$
So it is enough to take
\begin{align}\label{bound-2}
C^2_T=\tilde C_9e^{(C_2+\tilde C_8) T}
\end{align}
which gives \eqref{energy estimate-A}, and $C_T$  is increasing with respect to $T$.
We finish the proof of Theorem \ref{energy}.
\end{proof}

\begin{theorem}\label{uniform-existence}
Assume $u^s$ satisfies Lemma \ref{shear-profile}, and let  $\tilde{u}_{0}\in  H^{m+3}_{k + \ell'-1}(\mathbb{R}^2_+)$,  $m\ge 6$ be an even integer, $k>1, 0<\ell<\frac12,~~ \frac{1}{2}<\ell' < \ell + \frac{1}{2},~k+\ell>\frac 32$, and
$$
0<\zeta\le 1\,\,\,\mbox{with}\,\,\, C_m\zeta\le \frac{\tilde c_1}{2},
$$
where $C_m$ is the Sobolev embedding constant. If there exists $0<\zeta_0$ small enough such that,
 \begin{equation*}
\|\tilde{u}_0\|_{H^{m+1}_{k + \ell'-1}(\mathbb{R}^2_+)}\le  \zeta_0,
\end{equation*}
then, there exists $\epsilon_0>0$ and for any $0<\epsilon\le \epsilon_0$, the system \eqref{shear-prandtl-approxiamte-vorticity} admits a unique solution $\tilde{w}_\epsilon$ such that
$$
\|\tilde{w}_\epsilon\|_{L^\infty ([0, T_1]; H^{m}_{k+\ell}(\mathbb{R}^2_+))}\le \zeta,
$$
where $T_1$ is the lifespan of shear flow $u^s$ in the Lemma \ref{shear-profile}.
\end{theorem}

\begin{remark}\label{remark7.1}
Under the uniform monotonic assumption \eqref{shear-critical-momotone}, some results of above theorem holds for any fixed $T>0$. But $\zeta_0$ decreases as $T$ increases, according to the \eqref{c-tilde}.
\end{remark}

\begin{proof}
We fix $0<\epsilon\le 1$, then for any $\tilde{w}_{0}\in  H^{m+2}_{k+\ell}(\mathbb{R}^2_+)$, Theorem \ref{theorem3.1} ensures that,  there exists $\epsilon_0>0$ and for any $0<\epsilon\le \epsilon_0$, there exits $T_\epsilon>0$ such that the system \eqref{shear-prandtl-approxiamte-vorticity} admits a unique solution
$\tilde{w}_\epsilon \in L^\infty ([0, T_\epsilon]; H^{m+2}_{k+\ell}(\mathbb{R}^2_+))$ which satisfies
$$
\|\tilde{w}_\epsilon\|_{L^\infty ([0, T_\epsilon]; H^{m}_{k+\ell}(\mathbb{R}^2_+))}\le \frac 43 \|\tilde{w}_\epsilon(0)\|_{H^{m}_{k+\ell}(\mathbb{R}^2_+)}\le 2 \|\tilde{u}_0\|_{H^{m+1}_{k+\ell-1}(\mathbb{R}^2_+)}.
$$
Now choose $\zeta_0$ such that
\begin{equation*}
\max\{2, C_{T_1}\} \zeta_0\le \frac{\zeta}{2}.
\end{equation*}
On the other hand, taking $\tilde{w}_\epsilon(T_\epsilon)$ as initial data for  the system \eqref{shear-prandtl-approxiamte-vorticity},  Theorem \ref{theorem3.1} ensures  that there exits $T'_\epsilon>0$,
which is defined by \eqref{time-1} with $\bar\zeta=\frac{\zeta}{2}$, such that the system \eqref{shear-prandtl-approxiamte-vorticity} admits a unique solution
$\tilde{w}'_\epsilon \in L^\infty ([T_\epsilon, T_\epsilon+T'_\epsilon]; H^{m}_{k+\ell}(\mathbb{R}^2_+))$ which satisfies
$$
\|\tilde{w}'_\epsilon\|_{L^\infty ([T_\epsilon, T_\epsilon+T'_\epsilon]; H^{m}_{k+\ell}(\mathbb{R}^2_+))}\le \frac 43 \|\tilde{w}_\epsilon(T_\epsilon)\|_{H^{m}_{k+\ell}(\mathbb{R}^2_+)}\le \zeta.
$$
Now, we extend $\tilde{w}_\epsilon$ to $[0, T_\epsilon+T'_\epsilon]$ by $\tilde{w}'_\epsilon$, then we get a solution
$\tilde{w}_\epsilon \in L^\infty ([0, T_\epsilon+T'_\epsilon]; H^{m}_{k+\ell}(\mathbb{R}^2_+))$ which satisfies
$$
\|\tilde{w}_\epsilon\|_{L^\infty ([0, T_\epsilon+T'_\epsilon]; H^{m}_{k+\ell}(\mathbb{R}^2_+))}\le \zeta.
$$
So  if $T_\epsilon+T'_\epsilon<T_1$, we can apply Theorem \ref{energy} to $\tilde{w}_\epsilon$ with $T=T_\epsilon+T'_\epsilon$, and use \eqref{energy estimate-A}, this gives
$$
\|\tilde{w}_\epsilon\|_{L^\infty ([0, T_\epsilon+T'_\epsilon]; H^{m}_{k+\ell}(\mathbb{R}^2_+))}\le C_{T_1} \|\tilde{u}_0\|_{H^{m+1}_{k+\ell-1}(\mathbb{R}^2_+)}\le \frac{\zeta}{2}.
$$
Now taking $\tilde{w}_\epsilon(T_\epsilon+T'_\epsilon)$ as initial data for  the system \eqref{shear-prandtl-approxiamte-vorticity}, applying again Theorem \ref{theorem3.1}, for the same $T'_\epsilon>0$, the system
\eqref{shear-prandtl-approxiamte-vorticity} admits a unique solution
$\tilde{w}'_\epsilon \in L^\infty ([T_\epsilon+T'_\epsilon, T_\epsilon+2T'_\epsilon]; H^{m}_{k+\ell}(\mathbb{R}^2_+))$ which satisfies
$$
\|\tilde{w}'_\epsilon\|_{L^\infty ([T_\epsilon+T'_\epsilon, T_\epsilon+2T'_\epsilon]; H^{m}_{k+\ell}(\mathbb{R}^2_+))}\le \frac 43 \|\tilde{w}_\epsilon(T_\epsilon+T'_\epsilon)\|_{H^{m}_{k+\ell}(\mathbb{R}^2_+)}\le \zeta.
$$
Now, we extend $\tilde{w}_\epsilon$ to $[0, T_\epsilon+2T'_\epsilon]$ by $\tilde{w}'_\epsilon$, then we get a solution
$\tilde{w}_\epsilon \in L^\infty ([0, T_\epsilon+2T'_\epsilon]; H^{m}_{k+\ell}(\mathbb{R}^2_+))$ which satisfies
$$
\|\tilde{w}_\epsilon\|_{L^\infty ([0, T_\epsilon+2T'_\epsilon]; H^{m}_{k+\ell}(\mathbb{R}^2_+))}\le \zeta.
$$
So if $T_\epsilon+2T'_\epsilon<T_1$, we can apply Theorem \ref{energy} to $\tilde{w}_\epsilon$ with $T=T_\epsilon+2T'_\epsilon$, and use \eqref{energy estimate-A}, this gives again
$$
\|\tilde{w}_\epsilon\|_{L^\infty ([0, T_\epsilon+2T'_\epsilon]; H^{m}_{k+\ell}(\mathbb{R}^2_+))}\le C_{T_1} \|\tilde{u}_0\|_{H^{m+1}_{k+\ell-1}(\mathbb{R}^2_+)}\le \frac{\zeta}{2}.
$$
Then by recurrence, we can extend the solution $\tilde{w}_\epsilon$ to $[0, T_1]$, and then the lifespan of approximate solution is equal to  that of shear flow if the initial date $\tilde{u}_0$ is small enough.
\end{proof}

\label{sec-exi}
We have obtained the following estimate, for $m\ge 6$ and $0<\epsilon\le\epsilon_0$,
\begin{align*}
\|\tilde{w}_\epsilon(t)\|_{H^m_{k+\ell}(\mathbb{R}^2_+)} \le \zeta,\quad t \in [0, T_1].
\end{align*}
By using the equation \eqref{shear-prandtl-approxiamte-vorticity} and the Sobolev   inequality, we get, for $0<\delta<1$
\[
\|\tilde{w}_\epsilon\|_{Lip ([0, T_1]; C^{2, \delta}(\mathbb{R}^2_+))}\le M<+\infty.
\]
Then taking a subsequence, we have, for $0<\delta'<\delta$,
\[
\tilde{w}_\epsilon \to \tilde{w}\,\,({\epsilon\,\to\,0}),\,\, \text{locally strong in }~~C^0([0, T_1]; C^{2, \delta'}(\mathbb{R}^2_+))\,,
\]
and
\[
\partial _t \tilde{w} \in L^\infty ([0, T_1]; H^{m-2}_{k+\ell}(\mathbb{R}^2_+)),\quad
\tilde{w} \in L^\infty ([0, T_1]; H^{m}_{k+\ell}(\mathbb{R}^2_+)),
\]
with
\begin{align*}
\|\tilde{w}\|_{L^\infty ([0, T_1]; H^{m}_{k+\ell}(\mathbb{R}^2_+))} \le \zeta.
\end{align*}
Then we have
\begin{align*}
\tilde{u}=\partial^{-1}_y w\in L^\infty([0, T_1];  H^{m}_{k+\ell-1}(\mathbb{R}^2_+)),
\end{align*}
where we use the Hardy inequality \eqref{Hardy1}, since
$$
\lim_{y\to+\infty} \tilde{u}(t, x, y)=-\lim_{y\to+\infty} \int^{+\infty}_y \tilde{w}(t, x, \tilde{y} )d\tilde{y}=0.
$$
In fact, we also have
$$
\lim_{y\to 0} \tilde{u}(t, x, y)=\lim_{y\to 0} \int^y_0 \tilde{w}(t, x, \tilde{y} )d\tilde{y}=0.
$$
Using the condition $k+\ell-1>\frac 12$, we have also
$$
\tilde{v}=-\int^y_0 \tilde{u}_x\, d \tilde{y} \in L^\infty ([0, T_1]; L^\infty(\mathbb{R}_{+, y}); H^{m-1}(\mathbb{R}_x)).
$$
We have proven that, $\tilde{w}$ is a classical solution to the following vorticity Prandtl equation
\begin{align*}
\begin{cases}
& \partial_t\tilde{w} + (u^s + \tilde{u}) \partial_x\tilde{w} + \tilde{v} \partial_y(u^s_y+\tilde{w})
= \partial^2_y\tilde{w},\\
& \partial_y \tilde{w}|_{y=0} = 0,\\
& \tilde{w}|_{t=0} = \tilde{w}_0,
\end{cases}
\end{align*}
and $(\tilde{u}, \tilde{v})$ is a classical solution to \eqref{non-shear-prandtl}.
Finally, $(u, v)=(u^s+\tilde{u}, \tilde{v})$ is a classical solution to \eqref{full-prandtl},
and satisfies \eqref{main-energy}. In conclusion, we have proved the following theorem
which is the existence part of main Theorem \ref{main-theorem}.

\begin{theorem}\label{main-theorem-bis}
Let $m\ge 6$ be an even integer, $k>1, 0< \ell<\frac12,~ \frac 12< \ell' < \ell+ \frac 12,~ k+\ell>\frac 32$,  assume that $u^s_0$ satisfies \eqref{shear-critical-momotone}, the initial date
$\tilde{u}_0 \in H^{m+3}_{k + \ell' -1 }(\mathbb{R}^2_+)$ and $\tilde{u}_0 $  satisfies the compatibility condition \eqref{compatibility-a1}-\eqref{compatibility-a2} up to order $m+2$.  Then there exists $T>0$ such that if
\begin{equation*}
\|\tilde{u}_0  \|_{H^{m+1}_{k + \ell' -1 }(\mathbb{R}^2_+)}\le \delta_0,
\end{equation*}
for some $\delta_0>0$ small enough, then the initial-boundary value problem \eqref{non-shear-prandtl} admits a solution $(\tilde{u}, \tilde{v})$ with
 \begin{align*}
 &\tilde{u}\in L^\infty([0, T]; H^{m}_{k+\ell-1}(\mathbb{R}^2_+)),\quad \partial_y\tilde{u}\in L^\infty([0, T]; H^{m}_{k+\ell}(\mathbb{R}^2_+)).
 \end{align*}
Moreover, we have the following energy estimate,
\begin{align}\label{main-energy}
\begin{split}
\|\partial_y\tilde{u}\|_{L^\infty([0, T]; H^m_{k+\ell}(\mathbb{R}^2_+))} \le C\|\tilde{u}_0 \|^2_{ H^{m+1}_{k + \ell' -1 }(\mathbb{R}^2_+)}.
\end{split}
\end{align}
\end{theorem}

\section{Uniqueness and stability}\label{section8}

Now, we study  the stability of solutions which implies immediately the uniqueness of solution.

Let $\tilde{u}^1, \tilde{u}^2$ be two solutions obtained in Theorem \ref{main-theorem-bis}
with respect to the initial date $\tilde{u}^1_0, \tilde{u}^2_0$ respectively. Denote $\bar u = \tilde{u}^1 - \tilde{u}^2$
and $\bar v= \tilde{v}^1-\tilde{v}^2$, then
\begin{equation*}
\begin{cases}
\partial_t \bar{u} + (u^s + \tilde{u}_1)\partial_x \bar{u}   + (u^s_y + \tilde{u}_{1, y})\bar{v}
= \partial^2_y \bar{u} - \tilde{v}_2 \partial_y\bar{u} -(\partial_x\tilde{u}_2) \bar{u},\\
\partial_x \bar{u}+\partial_y\bar{v}=0,\\
\bar{u}|_{y=0}=\bar{v}|_{y=0}=0,\\
\bar{u}|_{t=0}=\tilde{u}^1_0 - \tilde{u}^2_0 .
\end{cases}
\end{equation*}
So  it is a linear equation for $\bar{u}$. We also have  for the vorticity $\bar w=
\partial_y \bar u$,
\begin{equation}\label{stability-2}
\begin{cases}
\partial_t \bar{w} + (u^s + \tilde{u}_1)\partial_x \bar{w}   + (u^s_{yy} + \tilde{w}_{1, y})\bar{v}
= \partial^2_y \bar{w} - \tilde{v}_2 \partial_y\bar{w} -(\partial_x\tilde{w}_2) \bar{u},\\
\partial_y\bar{w}|_{y=0}=0,\\
\bar{w}|_{t=0}=\tilde{w}^1_0 - \tilde{w}^2_0 .
\end{cases}
\end{equation}

\noindent {\bf Estimate with a loss of $x$-derivative.} Firstly, for the vorticity $\bar w=\partial_y \bar u$, we deduce an energy estimate with a loss of
$x$-derivative with the anisotropic norm defined by \eqref{norm-1}.
\begin{proposition}\label{prop8.1}
Let $\tilde{u}^1, \tilde{u}^2$ be two solutions obtained in Theorem \ref{main-theorem-bis}
with respect to the initial date $\tilde{u}^1_0, \tilde{u}^2_0$, then we have
\begin{equation}
\label{w-bar-less-k}
\begin{split}
\frac{d}{ dt}\| \bar{w}\|_{H^{m-2, m-3}_{k+\ell}(\mathbb{R}^2_+)}^2+ \|\partial_y\bar{w}\|_{H^{m-2, m-3}_{k+\ell}(\mathbb{R}^2_+)}^2\le \bar C_1\| \bar{w}\|_{H^{m-2}_{k+\ell}}^2,
\end{split}
\end{equation}
where the constant $\bar C_1$ depends on the norm of $\tilde{w}^1, \tilde{w}^2$ in $L^\infty([0, T]; H^m_{k+\ell}(\mathbb{R}^2_+))$.
\end{proposition}
\begin{proof}
The proof of this Proposition is similar to the proof of the Proposition \ref{prop3.1}, and we need to use that $m-2$ is even.
We only give the calculation  for the terms which need a different argument. Moreover we also explain
why we only  get the estimate on $\|\bar{w}\|_{H^{m-2}_{k+\ell}}^2$  but require the norm
of $\tilde{w}^1, \tilde{w}^2$ in $L^\infty([0, T]; H^m_{k+\ell}(\mathbb{R}^2_+))$.
With out loss of the generality, we suppose that $\|\bar{w}\|_{H^{m-2}_{k+\ell}}\le 1, \|\tilde{w}^1\|_{H^{m}_{k+\ell}}\le 1$ and $\|\tilde{w}^2\|_{H^{m}_{k+\ell}}\le 1$.

Derivating the equation of \eqref{stability-2} with $\partial^\alpha=\partial^{\alpha}_x\partial^{\alpha_2}_y$,
for $|\alpha|=\alpha_1+\alpha_2\le m-2, \alpha_1\le m-3$,
\begin{align}\label{8.1}
\begin{split}
&\partial_t \partial^{\alpha} \bar{w} - \partial_y^2 \partial^{\alpha}\partial\bar{w}= - \partial^{\alpha} \big((u^s + \tilde{u}_1)\partial_x  \bar{w}+ \tilde{v}_2 \partial_y\bar{w}\\
&\qquad\qquad+( u^s_{yy}+ \tilde{w}_{1, y} )\bar{v} +(\partial_x\tilde{w}_2) \bar{u} \big).
\end{split}
\end{align}	
Multiplying the above equation with $ \langle y \rangle^{k + \ell'+{\alpha_2}} \partial^{\alpha} \bar{w}$,
the same computation as in the proof of the Proposition \ref{prop3.1},  in particular,
the reduction of the boundary-data are the same,  gives
 \begin{align*}
 \begin{split}
 & \int_{\mathbb{R}^2_+} \bigg(\partial_t \partial^{\alpha} \bar{w} - \partial_y^2\partial^{\alpha} \bar{w} \bigg) \langle y \rangle^{2 (k+\ell+\alpha_2)} \partial^{\alpha} \bar{w} dx dy\\
 & \ge\frac 12 \frac{d}{dt}\|\partial^{\alpha}\bar{w}\|_{L^2_{k+\ell+\alpha_2}}^2+ \frac{3}{4} \|\partial_y \bar{w}\|_{H^{m-2, m-3}_{k+\ell}}^2 - C\|\bar{w}\|_{H^{m-2}_{k+\ell}}^2.
  \end{split}
 \end{align*}
 As for the right hand of \eqref{8.1}, for the first item, we split it into two parts
 \begin{align*}
 - \partial^{\alpha} \bigg((u^s + \tilde{u}_1)\partial_x  \bar{w}\bigg) =  -   (u^s + \tilde{u}_1)\partial_x \partial^{\alpha} \bar{w} +  [ (u^s + \tilde{u}_1), \partial^{\alpha}]\partial_x \bar{w}.
\end{align*}
Firstly, we have
\begin{align*}
\left|\int_{\mathbb{R}^2_+} \big((u^s + \tilde{u}_1) \partial_x \partial^{\alpha} \bar{w}\big)\langle y \rangle^{2(\ell+{\alpha_2})}\partial^{\alpha} \bar{w} dx dy\right| \le  \| \tilde{w}_1\|_{H^3_1}\|\partial^{\alpha} \bar{w}\|_{L^2_{k+\ell+{\alpha_2}}}^2.
\end{align*}
For the commutator operator, we have,
\begin{align*}
\| [ (u^s + \tilde{u}_1), \partial^{\alpha}]\partial_x \tilde{w}_\epsilon\|_{L^2_{k+\ell+\alpha_2}}&\le C \|\tilde{w}_1\|_{H^{m-2}_{k+\ell}(\mathbb{R}^2_+)}\|\bar{w}\|_{H^{m-2, m-3}_{k+\ell}(\mathbb{R}^2_+)}.
\end{align*}
Notice that for this term, we don't have the loss of $x$-derivative.

With the similar method for the terms $\tilde{v}_2 \partial_y\bar{w}$, we get
\begin{align*}
\left|\int_{\mathbb{R}^2_+} \tilde{v}_2 \partial_y\bar{w}\langle y \rangle^{2(\ell+{\alpha_2})}\partial^{\alpha} \bar{w} dx dy\right| &\le  \| \tilde{w}_2\|_{H^{m-2}_{k+\ell}(\mathbb{R}^2_+)}\|
\bar{w}\|_{H^{m-2, m-3}_{k+\ell}(\mathbb{R}^2_+)}^2.
\end{align*}
For the next one,  we have
$$
\partial^{\alpha} \bigg(( u^s_{yy}+\partial_y\tilde{w}_1) \bar{v} \bigg)
 = \sum\limits_{ \beta \le \alpha } C^\alpha_\beta\, \partial^{\beta} ( u^s_{yy}+\partial_y\tilde{w}_1) \partial^{\alpha - \beta}\bar{v},
$$
and thus
\begin{align*}
&\left\|\sum\limits_{ \beta \le \alpha, 1\le |\beta|<|\alpha| } C^\alpha_\beta\, \partial^{\beta} ( u^s_{yy}+\partial_y\tilde{w}_1)\partial^{\alpha - \beta}\bar{v}\right\|_{L^2_{k+\ell+{\alpha_2}}}\\
&\qquad\le
C\| \tilde{w}_1\|_{H^{m-2}_{k+\ell}(\mathbb{R}^2_+)}\|
\bar{w}\|_{H^{m-2, m-3}_{k+\ell}(\mathbb{R}^2_+)}.
\end{align*}
On the other hand, using Lemma \ref{inequality-hardy} and $\frac 32 -k<\ell<\frac 12$,
\begin{align*}
&\left\|\big(\partial^{\alpha} ( u^s_{yy}+\partial_y\tilde{w}_1)\big)\bar{v}\right\|_{L^2_{k+\ell+{\alpha_2}}}\le
\left\|\big(\partial^{\alpha} u^s_{yy}\big)\bar{v}\right\|_{L^2_{k+\ell+{\alpha_2}}}+
\left\|\big(\partial^{\alpha} \partial_y\tilde{w}_1\big)\bar{v}\right\|_{L^2_{k+\ell+{\alpha_2}}}\\
&\qquad\qquad\le C \left\|\bar{v}\right\|_{L^2(\mathbb{R}_x; L^\infty(\mathbb{R}_+))}+
C\| \tilde{w}_1\|_{H^{m}_{k+\ell}(\mathbb{R}^2_+)}\left\|\bar{v}\right\|_{L^\infty(\mathbb{R}^2_+)}\\
&\qquad\qquad\le C \left\|\bar{u}_x\right\|_{L^2_{\frac 12+\delta}(\mathbb{R}^2_+)}+
C\| \tilde{w}_1\|_{H^{m}_{k+\ell}(\mathbb{R}^2_+)}(\left\|\bar{u}_x\right\|_{L^2_{\frac 12+\delta}(\mathbb{R}^2_+)}
+\left\|\bar{u}_{xx}\right\|_{L^2_{\frac 12+\delta}(\mathbb{R}^2_+)})\\
&\qquad\qquad\le C(1+\| \tilde{w}_1\|_{H^{m}_{k+\ell}(\mathbb{R}^2_+)})\left\|\bar{w}\right\|_{H^2_{\frac 12+\delta}(\mathbb{R}^2_+)}\\
&\qquad\qquad\le C(1+\| \tilde{w}_1\|_{H^{m}_{k+\ell}(\mathbb{R}^2_+)})\left\|\bar{w}\right\|_{H^2_{k+\ell}(\mathbb{R}^2_+)}.
\end{align*}
So this term requires the norms $\| \tilde{w}_1\|_{H^{m}_{k+\ell}(\mathbb{R}^2_+)})$.

Moreover, if $\alpha_2\not =0$
\begin{align*}
&\left\|( u^s_{yy}+\partial_y\tilde{w}_1)\partial^{\alpha} \bar{v}\right\|_{L^2_{k+\ell+{\alpha_2}}}=
\left\|( u^s_{yy}+\partial_y\tilde{w}_1)\partial^{\alpha_1}_x\partial^{\alpha_2-1} \bar{u}_x\right\|_{L^2_{k+\ell+{\alpha_2}}}\\
&\qquad\qquad\le C (1+\| \tilde{w}_1\|_{H^{m-1}_{k+\ell}(\mathbb{R}^2_+)})\left\|\bar{w}
\right\|_{H^{m-2}_{k+\ell}(\mathbb{R}^2_+)},
\end{align*}
and also if $\alpha_2 =0$
\begin{align*}
&\left\|( u^s_{yy}+\partial_y\tilde{w}_1)\partial^{\alpha_1}_x \bar{v}\right\|_{L^2_{k+\ell}}=
\left\|( u^s_{yy}+\partial_y\tilde{w}_1)\partial^{-1}_y\partial^{\alpha_1}_x \bar{u}_x\right\|_{L^2_{k+\ell}}\\
&\qquad\qquad\le C (1+\| \tilde{w}_1\|_{H^{m-1}_{k+\ell}(\mathbb{R}^2_+)})\left\|\partial^{\alpha_1+1}_x\bar{w}\right\|_{L^2_{\frac 32+\delta}(\mathbb{R}^2_+)}.
\end{align*}
These two cases imply  the loss of $x$-derivative.

Similar argument also gives
\begin{align*}
\left|\int_{\mathbb{R}^2_+} \big(\partial^{\alpha}(\partial_x\tilde{w}_2) \bar{u}\big)\langle y \rangle^{2(\ell+{\alpha_2})}\partial^{\alpha} \bar{w} dx dy\right| \le C \| \tilde{w}_2\|_{H^m_{k+\ell}(\mathbb{R}^2_+)}\| \bar{w}\|_{H^{m-2}_{k+\ell}(\mathbb{R}^2_+)}^2,
\end{align*}
which finishes the proof of the Proposition \ref{prop8.1}.
\end{proof}

\noindent {\bf Estimate on the loss term.} To close the estimate \eqref{main-energy}, we need   to study the terms $\|\partial^{m-2}_x
\bar w\|_{L^2_{k+\ell}(\mathbb{R}^2_+)}$ which is missing in the left hand side of \eqref{w-bar-less-k}.

Similar to the argument in Section \ref{section7}, we will recover this term by the estimate of functions
\begin{align*}
\bar g_n& = \left( \frac{\partial_x^n \bar{u}}{u^s_y + \tilde u_{1,y}} \right)_y, \quad \forall (t, x, y)\in [0, T]\times \mathbb{R} \times \mathbb{R}^+.
\end{align*}

\begin{proposition}
\label{prop8.2b}
Let $\tilde{u}^1, \tilde{u}^2$ be two solutions obtained in Theorem \ref{main-theorem-bis}
with respect to the initial date $\tilde{u}^1_0, \tilde{u}^2_0$, then we have
\begin{equation*}
\begin{split}
\frac{d}{dt}\sum^{m-2}_{n=1}\|    \bar g_n \|_{L^2_{\ell'}(\mathbb{R}^2_+)}^2 & + \sum^{m-2}_{n=1}\|   \partial_y \bar g_n\|_{L^2_{\ell'}(\mathbb{R}^2_+)}^2\\
& \le C_2(\sum^{m-2}_{n=1}\|  \bar g_n\|_{L^2_{\ell'}(\mathbb{R}^2_+)}^2 + \|\bar {w}\|_{H^{m-2}_{k+\ell}}^2),
\end{split}
\end{equation*}
where the constant $\bar C_2$ depends on the norm of $\tilde{w}^1, \tilde{w}^2$ in $L^\infty([0, T]; H^m_{k+\ell}(\mathbb{R}^2_+))$.
\end{proposition}

These Propositions can be proven by using exactly the same calculation  as in Section \ref{section5}. The only difference is that when we use the Leibniz formula,
for the term where the order of derivatives is $|\alpha|=m-2$, it acts on the coefficient which depends on $\tilde{u}^1, \tilde{u}^2$. Therefore,  we need their norm in the order of $(m-2)+1$. So we omit the proof of this Proposition here.

With the similar  argument to   the proof of Theorem \ref{energy}, we get
\begin{equation*}
\|\bar{w}\|_{ L^\infty ([0, T]; H^{m-2}_{k+\ell}(\mathbb{R}^2_+))}\le  C \|\bar{u}_{0}\|_{H^{m+1}_{k + \ell'-1}(\mathbb{R}^2_+)},
\end{equation*}
which finishes  the proof of Theorem \ref{main-theorem}.

\appendix

\section{Some inequalities}
We will use the following Hardy type inequalities.
\begin{lemma}\label{inequality-hardy}
Let $f : \mathbb{R} \times \mathbb{R}^+\to \mathbb{R}$. Then
\begin{itemize}
 \item[(i)] if $\lambda > - \frac{1}{2}$ and $ \lim\limits_{y \to \infty} f(x,y) = 0$, then
 		\begin{equation} \label{Hardy1}
 		\|\langle y \rangle^\lambda f\|_{L^2 (\mathbb{R}^2_+)} \le C_\lambda
 		\|\langle y \rangle^{\lambda +1} \partial_y f\|_{L^2 (\mathbb{R}^2_+)};
 		\end{equation}
\item[(ii)] if $-1 \le \lambda < - \frac{1}{2}$ and $f(x, 0) = 0$, then
 		\begin{equation*}
 		\|\langle y \rangle^\lambda f\|_{L^2 (\mathbb{R}^2_+)} \le C_\lambda
 		\| \langle y \rangle^{\lambda + 1} \partial_y f \|_{L^2 (\mathbb{R}^2_+)}.
 		\end{equation*}
 	\end{itemize}
 Here  $C_\lambda \to +\infty,~ \text{as}~~\lambda \to -\frac 12$.
 \end{lemma}

We need the following trace theorem in the weighted Sobolev space.
\begin{lemma}\label{lemma-trace}
Let $\lambda>\frac 12$, then there exists $C>0$ such that for any function $f$ defined on $\mathbb{R}^2_+$, if $\partial_y f\in L^2_{\lambda}(\mathbb{R}^2_+)$, it admits a trace on $\mathbb{R}_x\times\{0\}$, and satisfies
$$
\|\gamma_0(f)\|_{L^2 (\mathbb{R}_x)}\le C
\|\partial_y f\|_{L^2_\lambda(\mathbb{R}^2_+)},
$$
where $\gamma_0(f)(x)=f(x, 0)$ is the trace operator.
\end{lemma}
The proof of the above two Lemmas is elementary, so we leave it to the reader.

We use also the following Sobolev inequality and algebraic properties of $H^m_{k+\ell}(\mathbb{R}^2_+)$,
\begin{lemma}\label{lemma2.4}
For the suitable functions $f, g$, we have:

\noindent
1) If the function $f$ satisfies $f(x, 0) = 0$ or $\lim_{y \to +\infty}f(x, y)=0$, then for any small $\delta>0$,
\begin{equation}\label{sobolev-1}
\begin{split}
\|f\|_{L^\infty(\mathbb{R}^2_+)}\le
C(\| f_y\|_{L^2_{\frac 12+\delta}(\mathbb{R}^2_+)}+\| f_{x y}\|_{L^2_{\frac 12+\delta}(\mathbb{R}^2_+)}).
\end{split}
\end{equation}
2) For $m\ge 6, k+\ell>\frac 32$, and any $\alpha, \beta\in \mathbb{N}^2$ with  $|\alpha|+|\beta|\le m$, we have
\begin{equation}\label{sobolev-2}
\begin{split}
\|(\partial^\alpha f)(\partial^\beta g)\|_{L^2_{k+\ell+\alpha_2+\beta_2}(\mathbb{R}^2_+)}\le
C\|f\|_{H^m_{k+\ell}(\mathbb{R}^2_+)}\|g\|_{H^m_{k+\ell}(\mathbb{R}^2_+)}.
\end{split}
\end{equation}
3) For $m\ge 6, k+\ell>\frac 32$, and any $\alpha\in \mathbb{N}^2, p\in\mathbb{N}$ with $|\alpha|+p\le m$, we have,
$$
\|(\partial^\alpha f)(\partial^p_x (\partial^{-1}_y g))\|_{L^2_{k+\ell+\alpha_2}(\mathbb{R}^2_+)}\le
C\|f\|_{H^m_{k+\ell}(\mathbb{R}^2_+)}\|g\|_{H^m_{\frac 12+\delta}(\mathbb{R}^2_+)},
$$
where $\partial^{-1}_y$ is the inverse of derivative $\partial_y$, meaning,
$\partial^{-1}_y g=\int^y_0 g(x, \tilde y) \, d\tilde{y}$.
\end{lemma}
\begin{proof} For (1), using $f(x, 0) = 0$, we have
\begin{equation*}
\begin{split}
\|f\|_{L^\infty(\mathbb{R}^2_+)}&= \left\|\int^y_0 (\partial_y f)(x, \tilde y) \, d\tilde{y}\right\|_{L^\infty(\mathbb{R}^2_+)}
\le C\|\partial_y f\|_{L^\infty(\mathbb{R}_x; L^2_{\frac 12+\delta}(\mathbb{R}_+))}\\
&\le
C(\| \partial_y f\|_{L^2_{\frac 12+\delta}(\mathbb{R}^2_+)}+\| \partial_x\partial_y f\|_{L^2_{\frac 12+\delta}(\mathbb{R}^2_+)}).
\end{split}
\end{equation*}
If $\lim_{y\to+\infty} f(x, y)=0$, we use
$$
f(x, y)=-\int^\infty_y (\partial_y f)(x, \tilde y) \, d\tilde{y}.
$$

For (2), firstly, $m\ge 6$ and $|\alpha|+|\beta|\le m$ imply $|\alpha|\le m-2$ or $|\beta|\le m-2$, without loss of generality, we suppose that $|\alpha|\le m-2$.
Then, using the conclusion of (1), we have
\begin{equation*}
\begin{split}
\|(\partial^\alpha f)(\partial^\beta g)\|_{L^2_{k+\ell+\alpha_2+\beta_2}(\mathbb{R}^2_+)}
&\le \|\langle y\rangle^{\alpha_2}(\partial^\alpha f)\|_{L^\infty(\mathbb{R}^2_+)}\|\partial^\beta g\|_{L^2_{k+\ell+\beta_2}(\mathbb{R}^2_+)}\\
&\le
C\|f\|_{H^{|\alpha|+2}_{\frac 12 +\delta}(\mathbb{R}^2_+)}\|\partial^\beta g\|_{L^2_{k+\ell+\beta_2}(\mathbb{R}^2_+)},
\end{split}
\end{equation*}
which give \eqref{sobolev-2}.

For (3), if $|\alpha|\le m-2$, we have
\begin{equation*}
\begin{split}
\|(\partial^\alpha f)&(\partial^p_x (\partial^{-1}_y g))\|_{L^2_{k+\ell+\alpha_2}(\mathbb{R}^2_+)}\\
&\le \|\langle y\rangle^{k+\ell+\alpha_2}(\partial^\alpha f)\|_{L^2(\mathbb{R}_{y, +}; L^\infty(\mathbb{R}_x))}\|\partial^p_x (\partial^{-1}_y g)\|_{L^\infty(\mathbb{R}_{y, +}; L^2(\mathbb{R}_x))}\\
&\le
C\|f\|_{H^{|\alpha|+2}_{k+\ell}(\mathbb{R}^2_+)}\|\partial^p_x g\|_{L^2_{\frac 12+\delta}(\mathbb{R}^2_+)}.
\end{split}
\end{equation*}
If $p\le m-2$, we have
\begin{equation*}
\begin{split}
\|(\partial^\alpha f)&(\partial^p_x (\partial^{-1}_y g))\|_{L^2_{k+\ell+\alpha_2}(\mathbb{R}^2_+)}\\
&\le \|\langle y\rangle^{k+\ell+\alpha_2}(\partial^\alpha f)\|_{L^2(\mathbb{R}^2_+)}\|\partial^p_x (\partial^{-1}_y g)\|_{L^\infty(\mathbb{R}^2_+)}\\
&\le
C\|f\|_{H^{|\alpha|}_{k+\ell}(\mathbb{R}^2_+)}\|\partial^p_x g\|_{L^\infty(\mathbb{R}_x; L^2_{\frac 12+\delta}(\mathbb{R}_{y, +}))}\\
&\le
C\|f\|_{H^{|\alpha|}_{k+\ell}(\mathbb{R}^2_+)}\|g\|_{H^m_{\frac 12+\delta}(\mathbb{R}^2_+)}.
\end{split}
\end{equation*}
We have completed the proof of the Lemma.
\end{proof}


\section{The existence of approximate solutions} \label{section-a3}
Now, we prove the Proposition \ref{prop3.0},  the existence of solution to the vorticity equation   $ \tilde{w}_\epsilon=\partial_y\tilde{u}_\epsilon $ and
  suppose that $m, k, \ell$ and $u^s(t, y)$ satisfy the assumption of Proposition \ref{prop3.0},
\begin{align}
\label{apendix-vorticity-bb}
\begin{cases}
& \partial_t\tilde{w}_\epsilon + (u^s + \tilde{u}_\epsilon) \partial_x\tilde{w}_\epsilon +{v}_\epsilon (u^s_{yy} + \partial_y \tilde{w}_{\epsilon})
= \partial^2_{y}\tilde{w}_\epsilon + \epsilon \partial^2_{x}\tilde{w}_\epsilon, \\
& \partial_y \tilde{w}_{\epsilon}|_{y=0}=0\\
&\tilde{w}_\epsilon|_{t=0}=\tilde{w}_{0, \epsilon},
\end{cases}
\end{align}
where 
\begin{equation*}
\tilde{u}_\epsilon(t, x, y)=-\int^{+\infty}_y \tilde{w}_\epsilon(t, x, \tilde y) d\tilde y,\quad
\tilde{v}_\epsilon(t, x, y)=-\int^{y}_0\partial_x \tilde{u}_\epsilon(t, x, \tilde y) d\tilde y.
\end{equation*}
We will use  the following iteration process to prove the existence of solution,
where $ w^0=\tilde{w}_{0, \epsilon}$,
\begin{align}
\label{apendix-vorticity-iteration}
\begin{cases}
& \partial_tw^n + (u^s+{u}^{n-1} )\partial_x w^{n} + (u^s_{yy} +\partial_y w^{n-1}){v}^{n}
= \partial^2_{y}w^n  + \epsilon \partial^2_{x}w^n , \\
& \partial_y w^n |_{y=0}=0\\
&w^n|_{t=0}=\tilde{w}_{0, \epsilon},
\end{cases}
\end{align}
with
\begin{equation*}
u^{n-1}(t, x, y)=-\int^{+\infty}_y w^{n-1}(t, x, \tilde y) d\tilde y,
\end{equation*}
and
\begin{align*}
v^{n}(t, x, y)&=-\int^{y}_0\partial_x u^{n}(t, x, \tilde y) d\tilde y\\
&=\int^{y}_0\int^{+\infty}_{\tilde y} \partial_x w^{n}(t, x, z) dz d\tilde y.
\end{align*}

Here for the boundary data, we have
$$
\partial^3_{y}w^n|_{y=0}=((u^s_y+{w}^{n-1} )\partial_x w^{n})|_{y=0},
$$
\begin{equation*}
\begin{split}
&\qquad(\partial^5_y  w^n)(t, x, 0)\\
&=
\left(\partial^3_y u^s(t, 0) + \partial^2_y w^{n-1}(t, x, 0)+\epsilon (\partial^2_x w^{n-1})(t, x, 0)\right)( \partial_x w^n )(t, x, 0)\\
&\qquad+\left(u^s_y(t, 0) + (w^{n-1})(t, x, 0)\right) \left((\partial^2_y\partial_x w^{n})(t, x, 0)+\epsilon (\partial^3_x w^{n})(t, x, 0)\right)\\
&\qquad\qquad\qquad\qquad-(\partial_y\partial_x w^{n})(u^s_y + w^{n-1})(t, x, 0)\\
&\quad+\sum_{1\le j\le 3}C^4_j \bigg((\partial^j_y(u^s + u^{n-1})) \partial^{4-j}_y\partial_x u^{n} - (\partial^{j - 1}_y\partial_x \tilde{u}^n )\partial^{4-j}_y(u^s_y + w^{n-1})\bigg)(t, x, 0)\\
&\qquad\qquad -\epsilon \partial^2_{x}\bigg(\left(u^s_y(t, 0) + (w^{n-1})(t, x, 0)\right) (\partial_x w^n )(t, x, 0)\bigg).
\end{split}
\end{equation*}
and also for $3\le p\le \frac m2+1$,  $\partial^{2p+1}_{y}w^n|_{y=0}$ is a linear combination of the terms of the form:
\begin{align}\label{mu-4}
 \prod^{q_1}_{j=1}\bigg(\partial_x^{\alpha_j} \partial_y^{\beta_j + 1}\big( u^s + {u}^n \big) \bigg)\bigg|_{y=0}\times \prod^{q_2}_{l=1}  \bigg(\partial_x^{\tilde \alpha_l} \partial_y^{\tilde\beta_l + 1}\big( u^s +  {u}^{n-i} \big)\bigg)\bigg|_{y=0}\,\, ,
\end{align}
where $2\le  q_1+q_2\le p,\,\, 1\le i \le \min\{n,\, p\} $ and
\begin{align*}
&\alpha_j + \beta_j\le 2p - 1, \,\, 1\le j\le q_1;\,\, \tilde\alpha_l + \tilde\beta_l \le 2p - 1,\,\, 1\le l\le q_2;&\\
&\sum^{q_1}_{j=1} (3\alpha_j + \beta_j) +\sum^{q_2}_{l=1}(3\tilde \alpha_l + \tilde\beta_l )= 2p +1\,;&\\
&~\sum\limits_{j=1}^{q_1}\beta_j+\sum\limits_{l=1}^{q_2}\tilde \beta_l \le 2p -2;\,\,~\sum\limits_{j=1}^{q_1} \alpha_j +\sum\limits_{l=1}^{q_2} \tilde\alpha_l \le p - 1,
\,\,\,0<\sum\limits_{j=1}^{q_1} \alpha_j  .&
\end{align*}
Remark that the condition $0<\sum\limits_{j=1}^{q_1} \alpha_j$ implies that, in  \eqref{mu-4}, there are at last one factor like $\partial_x^{\alpha_j}\partial_y^{\beta_j +1}  {u}^n(t, x, 0)$.

For given $w^{n-1}$, we have ${u}^{n-1}=\partial^{-1}_yw^{n-1} $ and ${v}^{n}=-\partial^{-1}_y{u}^{n}_{x}$. We will prove the existence and boundness of the sequence $\{ w^n, n\in\mathbb{N}\}$ in $L^\infty([0, T_\epsilon]; H^{m+2}_{k+\ell}(\mathbb{R}^2_+))$ to the linear equation \eqref{apendix-vorticity-iteration} firstly, then the existence of solution to \eqref{apendix-vorticity-bb} is guaranteed by using the standard weak convergence methods.

\begin{lemma}\label{lemma-app-vorticity-iteration-1}
Assume that $w^{n-i}\in L^\infty([0, T];  H^{m+2}_{k+\ell}(\mathbb{R}^2_+)), 1\le i \le \min\{n,  \frac{m}{2} +1\}$  and $\tilde w_{0, \epsilon}$ satisfies the compatibility condition up to order $m+2$ for the system \eqref{apendix-vorticity-bb}, then the initial-boundary value problem \eqref{apendix-vorticity-iteration} admit a unique solution $w^n$ such that, for any $t\in [0, T]$,
\begin{equation}\label{appendix-ck-a}
\frac {d}{dt}\|w^n (t)\|_{H^{m+2}_{k+\ell}(\mathbb{R}^2_+)}^2\le B^{n-1}_T\|{w}^n(t)\|^2_{H^{m+2}_{k+\ell}(\mathbb{R}^2_+)}+
D^{n-1}_T\| w^n\|^{m+2}_{H^{m+2}_{k+\ell}(\mathbb{R}^2_+)},
\end{equation}
where
\begin{align*}
B^{n-1}_T=C\bigg(1+&{ \sum\limits_{i=1}^{\min\{n,  {m}/{2} +1\}}}\|w^{n-i}\|_{L^\infty([0, T]; H^{m+2}_{k+\ell}(\mathbb{R}^2_+))}\\
&+
(1+\frac{1}{\epsilon} ) { \sum\limits_{i=1}^{\min\{n,  {m}/{2} +1\}}}\|w^{n-i}\|^2_{L^\infty([0, T]; H^{m+2}_{k+\ell}(\mathbb{R}^2_+))} \bigg),
\end{align*}
and
$$
D^{n-1}_T=C{ \sum\limits_{i=1}^{\min\{n,  {m}/{2} +1\}}}\|w^{n-i}\|^{m+2}_{L^\infty([0, T];H^{m+2}_{k+\ell}(\mathbb{R}^2_+))}\,.
$$
\end{lemma}

\begin{proof}
Once we get  {\em \`a priori} estimate for this linear problem, the existence of solution is guaranteed by the Hahn-Banach theorem. So we  only prove the {\em \`a  priori} estimate of the smooth solutions.

For any $\alpha\in \mathbb{N}^2, |\alpha|\le m+2$, taking the equation \eqref{apendix-vorticity-iteration} with derivative $\partial^\alpha$,
multiplying the resulting equation by $\langle y \rangle^{2k + 2\ell+2\alpha_2} \partial^\alpha w^n  $ and integrating  by part over $\mathbb{R}^2_+$, one obtains that
\begin{align}\label{appendix-ck}
\begin{split}
&\frac{1}{2}\frac{d}{dt}\|w^n \|_{H^{m+2}_{k+\ell}(\mathbb{R}^2_+)}^2 + \|\partial_y w^n  \|_{H^{m+2}_{k+\ell}(\mathbb{R}^2_+)}^2 + \epsilon \|\partial_x w^n  \|_{H^{m+2}_{k+\ell}(\mathbb{R}^2_+)}^2\\
&
=\sum_{|\alpha|\le {m+2}}\int_{\mathbb{R}^2_+}
\langle y \rangle^{2k+ 2\ell +2\alpha_2 }\partial^\alpha\big((u^s+{u}^{n-1})\partial_x w^{n} \\
&\qquad\qquad\qquad \qquad -(\partial^{-1}_y {u}^n_{x})(u^s_{yy} +\partial_y {w}^{n-1})\big)\partial^\alpha w^n dx dy\\
&\quad+\sum_{|\alpha|\le {m+2}}\int_{\mathbb{R}^2_+}
(\langle y \rangle^{2k+ 2\ell +2\alpha_2 })'\partial^\alpha\partial_y w^{n}\partial^\alpha\partial_y w^{n}dxdy \\
&\qquad \qquad +\sum_{|\alpha|\le {m+2}}\int_{\mathbb{R}}
(\partial^\alpha\partial_y w^{n}\partial^\alpha\partial_y w^{n})\big|_{y=0}dx,
\end{split}
\end{align}
With similar analysis to Section \ref{section5}, we have
\begin{align*}
\begin{split}
&\left|\int_{\mathbb{R}^2_+}
\langle y \rangle^{2k+ 2\ell +2\alpha_2 }(u^s+{u}^{n-1})\partial_x\partial^\alpha w^{n}\partial^\alpha w^n dx dy\right|\\
&=\left|-\frac 12 \int_{\mathbb{R}^2_+}
\langle y \rangle^{2k+ 2\ell +2\alpha_2 }\partial_x(u^s+{u}^{n-1})\partial^\alpha w^{n}\partial^\alpha w^n dx dy\right|\\
&\le C \|{u}^{n-1}\|_{L^\infty(\mathbb{R}^2_+)} \|w^n\|^2_{H^{m+2}_{k+\ell}(\mathbb{R}^2_+)},
\end{split}
\end{align*}
and
\begin{align*}
\begin{split}
&\left|\int_{\mathbb{R}^2_+}
\langle y \rangle^{2k+ 2\ell +2\alpha_2 }[\partial^\alpha, (u^s+{u}^{n-1})]\partial_x w^{n}\partial^\alpha w^n dx dy\right|\\
&\le C(1+ \|{w}^{n-1}\|_{H^{m+2}_{k+\ell}(\mathbb{R}^2_+)} ) \|w^n\|^2_{H^{m+2}_{k+\ell}(\mathbb{R}^2_+)}.
\end{split}
\end{align*}
For the second term on the right hand   of \eqref{appendix-ck}, by using the Leibniz formula, we need to pay more attention to the following two terms
\begin{align*}
\begin{split}
&\left|\int_{\mathbb{R}^2_+}
\langle y \rangle^{2k+ 2\ell +2\alpha_2}\big(\partial^\alpha\partial^{-1}_y {u}^n_{x}\big)(u^s_{yy} +\partial_y {w}^{n-1})\partial^\alpha w^n dx dy\right|\\
&\le C(1+\|{w}^{n-1}\|_{H^{m+2}_{k+\ell}(\mathbb{R}^2_+)})\|\partial_x w^n\|_{H^{m+2}_{k+\ell}(\mathbb{R}^2_+)}
\|w^n\|_{H^{m+2}_{k+\ell}(\mathbb{R}^2_+)}\\
&\le\frac{\epsilon}{2}\|\partial_x w^n\|^2_{H^{m+2}_{k+\ell}(\mathbb{R}^2_+)}+
\frac{C}{\epsilon}(1+\|{w}^{n-1}\|_{H^{m+2}_{k+\ell}(\mathbb{R}^2_+)})^2
\|w^n\|^2_{H^{m+2}_{k+\ell}(\mathbb{R}^2_+)},
\end{split}
\end{align*}
and
\begin{align*}
\begin{split}
&\int_{\mathbb{R}^2_+}
\langle y \rangle^{2k+ 2\ell +2\alpha_2 } {v}^n\big)\big(\partial^\alpha\partial_y {w}^{n-1}\big)\partial^\alpha w^n dx dy\\
&=-\int_{\mathbb{R}^2_+}\partial_y \big(
\langle y \rangle^{2k+ 2\ell +2\alpha_2 }(\partial^{-1}_y {u}^n_{x})\big)\big(\partial^\alpha{w}^{n-1}\big)\partial^\alpha w^n dx dy\\
&\quad-\int_{\mathbb{R}^2_+}\big(
\langle y \rangle^{2k+ 2\ell +2\alpha_2 }(\partial^{-1}_y {u}^n_{x})\big)\big(\partial^\alpha {w}^{n-1}\big)\partial_y \partial^\alpha w^n dx dy,
\end{split}
\end{align*}
here we have used $v^n|_{y=0}=0$, thus
\begin{align*}
\begin{split}
&\left|\int_{\mathbb{R}^2_+}
\langle y \rangle^{2k+ 2\ell +2\alpha_2} {v}^n\big)\big(\partial^\alpha\partial_y {w}^{n-1}\big)\partial^\alpha w^n dx dy\right|\\
&\le C\|{w}^{n-1}\|_{H^{m+2}_{k+\ell}(\mathbb{R}^2_+)}
\big(\|w^n\|^2_{H^{m+2}_{k+\ell}(\mathbb{R}^2_+)}+\|\partial_y w^n\|_{H^{m+2}_{k+\ell}(\mathbb{R}^2_+)}
\|w^n\|_{H^{m+2}_{k+\ell}(\mathbb{R}^2_+)}\big).
\end{split}
\end{align*}
For the boundary term, similar to  the proof of Proposition \ref{prop3.1}, we can get
\begin{align*}
&\sum_{|\alpha|\le {m+2}}\left|\int_{\mathbb{R}}
(\partial^\alpha\partial_y w^{n}\partial^\alpha\partial_y w^{n})\big|_{y=0}dx\right|\\
&\le \frac 1{16} \|\partial_y w^n\|^2_{H^{m+2}_{k+\ell}(\mathbb{R}^2_+)}+
C\|w^{n-1}\|^{m+2}_{H^{m+2}_{k+\ell}(\mathbb{R}^2_+)}\| w^n\|^{m+2}_{H^{m+2}_{k+\ell}(\mathbb{R}^2_+)}.
\end{align*}
We get finally
\begin{align*}
\begin{split}
\frac {d}{dt}\|w^n (t)\|_{H^{m+2}_{k+\ell}(\mathbb{R}^2_+)}^2 &+ \|\partial_y w^n(t)  \|_{H^{m+2}_{k+\ell}(\mathbb{R}^2_+)}^2  + \epsilon \|\partial_x w^n(t)  \|_{H^{m+2}_{k+\ell}(\mathbb{R}^2_+)}^2 \\
&\le B^{n-1}_T\|{w}^n(t)\|^2_{H^{m+2}_{k+\ell}(\mathbb{R}^2_+)}+
D^{n-1}_T\| w^n\|^{m+2}_{H^{m+2}_{k+\ell}(\mathbb{R}^2_+)}\, .
\end{split}
\end{align*}
\end{proof}

\begin{lemma}\label{lemmab.2}
Suppose that $m, k, \ell$ and $u^s(t, y)$ satisfy the assumption of Proposition \ref{prop3.0}, $\bar\zeta>0$, then for any $0<\epsilon\le 1$, there exists $T_\epsilon>0$ such that for any $\tilde{w}_{0, \epsilon}\in H^{m+2}_{k+\ell}(\mathbb{R}^2_+)$ with
$$
\|\tilde{w}_{0, \epsilon}\|_{H^{m+2}_{k+\ell}(\mathbb{R}^2_+)}\le \bar \zeta,
$$
the iteration equations \eqref{apendix-vorticity-iteration} admit a sequence of solution $\{w^n, n\in\mathbb{N}\}$ such that, for any $t\in [0, T_\epsilon]$,
$$
\|w^n(t)\|_{H^{m+2}_{k+\ell}(\mathbb{R}^2_+)}\le \frac 43\|\tilde{w}_{0, \epsilon}\|_{H^{m+2}_{k+\ell}(\mathbb{R}^2_+)} ,\quad \forall n\in\mathbb{N}.
$$
\end{lemma}
\noindent
{\bf Remark.}  Here $\bar\zeta$ is aribitary.

\begin{proof}
Integrating \eqref{appendix-ck-a} over $[0, t]$, for $0<t\le T$ and $T>0$ small,
$$
\|w^n(t)\|^{m}_{H^{m+2}_{k+\ell}(\mathbb{R}^2_+)}\le \frac{\|\tilde{w}_{0, \epsilon}\|^{m}_{H^{m+2}_{k+\ell}(\mathbb{R}^2_+)}}{e^{-\frac m2 B^{n-1}_T t}-\frac m2 D^{n-1}_T t\|\tilde{w}_{0, \epsilon}\|^{m}_{H^{m+2}_{k+\ell}(\mathbb{R}^2_+)}}.
$$
We prove the Lemma by induction. For $n=1$, we have
\begin{align*}
B^{0}_T&=C\left(1+\|\tilde w_{0, \epsilon}\|_{ H^{m+2}_{k+\ell}(\mathbb{R}^2_+)}+
(1+\frac{1}{\epsilon} ) \|\tilde w_{0, \epsilon}\|^2_{ H^{m+2}_{k+\ell}(\mathbb{R}^2_+)} \right)\\
&\le C\left(1+\bar \zeta+
(1+\frac{1}{\epsilon} ) \bar \zeta^2 \right) ,
\end{align*}
and
$$
D^{n-1}_T=C\|\tilde w_{0, \epsilon}\|^{m+2}_{H^{m+2}_{k+\ell}(\mathbb{R}^2_+)}\le C\bar \zeta^{m+2}.
$$
Choose $T_\epsilon>0$ small such that
$$
\left(e^{-\frac m2 C\left(1+2\bar \zeta+
4(1+\frac{1}{\epsilon} ) \bar \zeta^2 \right) T_\epsilon}-\frac m2 C(2\bar \zeta)^{m+2} T_\epsilon  (2\bar \zeta)^{m}\right)^{-1}=\left(\frac 43\right)^m,
$$
we get
$$
\|w^1(t) \|_{H^{m+2}_{k+\ell}(\mathbb{R}^2_+)} \le \frac 43 \|\tilde{w}_{0, \epsilon}\|_{H^{m+2}_{k+\ell}(\mathbb{R}^2_+)}.
$$
Now the induction hypothesis is:  for $0\le t\le T_\epsilon$,
$$
\|w^{n-1}(t) \|_{H^{m+2}_{k+\ell}(\mathbb{R}^2_+)} \le \frac 43 \|\tilde{w}_{0, \epsilon}\|_{H^{m+2}_{k+\ell}(\mathbb{R}^2_+)},
$$
thanks to the choose of $T_\epsilon$, we have also
$$
\left(e^{-\frac m2 B^{n-1}_{T_\epsilon} T_\epsilon}-\frac m2 D^{n-1}_{T_\epsilon} T_\epsilon\|\tilde{w}_{0, \epsilon}\|^{m}_{H^{m+2}_{k+\ell}(\mathbb{R}^2_+)}\right)^{_1}\le \left(\frac 43\right)^m
$$
for any $t\in [0, T_\epsilon]$, then we finish the proof of the Lemma \ref{lemmab.2}.
\end{proof}


\section*{Acknowledgments}

The first author was partially supported by `` the Fundamental Research Funds
for the Central Universities'' and the NSF of China (No. 11171261). The second author was
supported by a period of sixteen months scholarship from the State Scholarship Fund of China, and
he would thank   ``Laboratoire de math\'ematiques Rapha\"el Salem de l'Universit\'e de Rouen" for
their hospitality.

\end{document}